\documentclass[11pt]{amsart}
\pdfoutput=1

\usepackage{geometry}
\usepackage{amsmath}
\usepackage{amsthm}
\usepackage{amsfonts}
\usepackage{amssymb}
\usepackage{graphicx}
\usepackage{array}

\usepackage[l>=4em]{diagrams}
\usepackage{paralist}

\newtheorem{thm}{Theorem}[section]
\newtheorem{theorem}[thm]{Theorem}

\newtheorem{cor}[thm]{Corollary}  
\newtheorem{corollary}[thm]{Corollary}  
\newtheorem{lem}[thm]{Lemma}
\newtheorem{lemma}[thm]{Lemma}
\newtheorem{prop}[thm]{Proposition}
\newtheorem{proposition}[thm]{Proposition}
\theoremstyle{definition}

\newtheorem{definition}[thm]{Definition}
\newtheorem{rem}[thm]{Remark}
\newtheorem{remark}[thm]{Remark}

\newcommand{\Z}{{\mathbb{Z}}}
\newcommand{\x}{{\mathbf{x}}}
\newcommand{\y}{{\mathbf{y}}}
\newcommand{\s}{{\mathbf{S}}}
\newcommand{\T}{{\mathcal{T}}}

\newcommand{\G}{{\mathbb{G}}}

\newcommand\Pent{\mathrm{Pent}}
\newcommand\Hex{\mathrm{Hex}}
\newcommand\Rect{\mathrm{Rect}}
\def\cm{\cdot}

\newcommand\Os{\mathbb{O}}

\def\Sym{\mathrm{Sym}}
\def\Interior{\mathrm{Int}}

\def\Os{\mathbb O}

\def\fin\qedhere
\def\pr {{\text{pr}}}

\def\x{\mathbf x}
\def\y{\mathbf y}
\def\z{\mathbf z}

\def\S{\mathbf S}
\def\alphas{\boldsymbol{\alpha}}
\def\betas{\boldsymbol{\beta}}

\newcommand{\Hf}{H_1(E(f))}
\newcommand{\HZ}{(H_1(E(f)),\Z)}

\begin{document}
	
\dedicatory{In memory of Tim Cochran}

\title{Heegaard Floer homology of spatial graphs}         
\author{Shelly Harvey$^{\dag}$}
\address{Department of Mathematics \\Rice University}        
\author{Danielle O'Donnol$^{\dag\dag}$}
\address{Department of Mathematics \\ Indiana University}

\thanks{$^{\dag}$ This work was partially supported National Science Foundation grants DMS-0748458 and DMS-1309070 and by a grant from the Simons Foundation (\#304538 to Shelly Harvey). $^{\dag\dag}$ This work was partially supported the National Science Foundation grant DMS-1406481 and by an AMS-Simons travel grant.}

\begin{abstract}We extend the theory of combinatorial link Floer homology to a class of oriented spatial graphs called transverse spatial graphs.
	To do this, we define the notion of a grid diagram representing a transverse spatial graph, which we call a graph grid diagram.  
	We prove that two graph grid diagrams representing the same transverse spatial graph are related by a sequence of graph grid moves, generalizing the work of Cromwell for links.  
	For a graph grid diagram representing a transverse spatial graph $f:G \rightarrow S^3$, we define a relatively bigraded chain complex (which is a module over a multivariable polynomial ring) and show that its homology is preserved under the graph grid moves; hence it is an invariant of the transverse spatial graph.  In fact, we define both a minus and hat version.  Taking the graded Euler characteristic of the homology of the hat version gives an Alexander type polynomial for the transverse spatial graph.  Specifically, for each transverse spatial graph $f$, we define a balanced sutured manifold $(S^3\smallsetminus f(G), \gamma(f))$.  We show that the graded Euler characteristic is the same as the torsion of $(S^3\smallsetminus f(G), \gamma(f))$ defined by S. Friedl, A. Juh\'{a}sz, and J. Rasmussen. 
\end{abstract}
\maketitle

\section{Introduction}
 
Knot Floer homology, introduced by P. Ozsv\'{a}th and Z. Szab\'{o} \cite{OS3}, and independently by 
 J. Rasmussen \cite{Ras}, is an invariant of knots in $S^3$ that categorifies the Alexander polynomial.  
 Knot Floer homology is widely studied because of its many applications in low-dimensional 
 topology. For example, it detects the unknot \cite{OS4}, detects whether a knot is fibered \cite{Gh,Ni}, and detect the genus of a knot \cite{OS4}.    The theory was generalized to links in \cite{OS5}.
The primary goal of this paper is to extend link Floer homology to a class of oriented spatial graphs in $S^3$, called transverse spatial graphs. 
   
 Originally, knot Floer homology was defined as the homology of a chain complex obtained by counting certain holomorphic disks in a $2g$-dimensional symplectic manifold with some boundary conditions that arose from a (doubly pointed) Heegaard diagram for $S^3$ compatible with the knot.  As such, the chain \emph{groups} were combinatorial but one could not, a priori, compute the boundary map.  However, Sucharit Sarkar discovered a criterion that would ensure that the count of certain holomorphic disks is combinatorial.  This crucial idea was used by C. Manolescu, P. Ozsv\'{a}th, and S. Sarker in \cite{MOS} to give a combinatorial description of link Floer homology using grid diagrams.  Using this description, in \cite{MOST}, C. Manolescu, P. Ozsv\'{a}th, Z. Sz\'{a}bo, and D. Thurston gave a combinatorial proof that link Floer homology is an invariant.  In this paper, we generalize the combinatorial description of Heegaard Floer homology and proof in \cite{MOS, MOST} to transverse spatial graphs.  Specifically, we define a relatively bigraded chain complex (a combinatorial minus version) which is a module over $\mathbb{F}[U_1,\dots, U_V]$ where $V$ is the number of vertices of the transverse spatial graph and $\mathbb{F}=\mathbb{Z}/2\mathbb{Z}$ is the field with two elements. We then show that it is well defined up to quasi-isomorphism.  We note that, independently, in \cite{Bao14}, Y. Bao defined a Heegaard Floer homology (a non-combinatorial hat version), which is an $\mathbb{F}$ vector space, for balanced bipartite spatial graphs.  

Informally, a transverse spatial graph is an oriented spatial graph where the incoming (respectively outgoing) edges are grouped at each vertex and any ambient isotopy must preserve this grouping.  See Figure~\ref{tgraph_diagram_intro} for an example.  Details can be found in Section~\ref{balanced_sp_gr}. 
\begin{figure}[h]
\begin{center}
\includegraphics{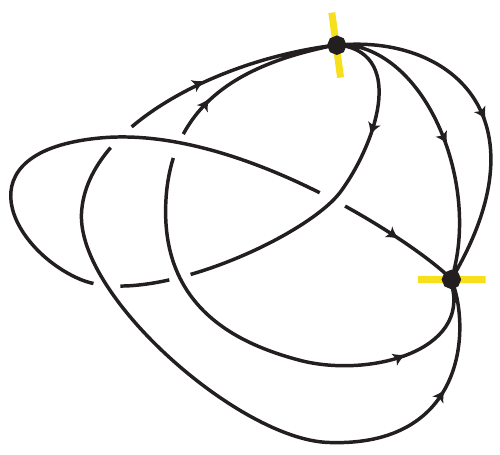}
\caption{An example of a diagram of a transverse spatial graph.  The lines at each vertex indicate the grouping.  }\label{tgraph_diagram_intro}
\end{center}
\end{figure} 
To define the chain complex, we first introduce the notion of a graph grid diagram representing a transverse spatial graph in Section~\ref{sec:grid}.  Roughly, a graph grid diagram is an $n \times n$ grid of squares each of which is decorated with an $X$, an $O$, or is empty, where some of the $O$'s are decorated with an $\ast$, and satisfies the following conditions.  Like for links, there is precisely one $O$ per row and column.  There are no restrictions on the $X$'s but if an $O$ shares a row or column with multiple (or no) $X$'s then it must be decorated with a $\ast$.  Moreover, each \emph{connected component} must contain an $O$ decorated with an $\ast$.  See Subsection~\ref{subsec:graphgrid} for a precise definition and Figure~\ref{grid_ex} for an example.  
  
  \begin{figure}[h]
\begin{center}
\begin{picture}(120,120)
\put(0,0){\includegraphics{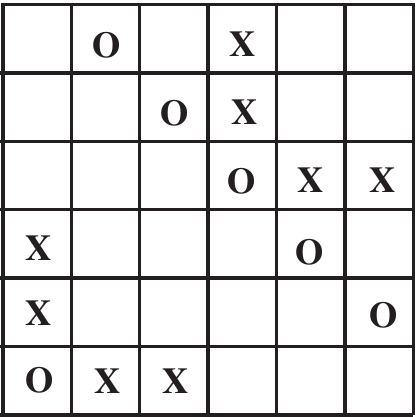}}
\put(13,8){\large{*}}  
\put(71,68){\large{*}}  
\end{picture}
\end{center}
\caption{Example of a graph grid diagram.}\label{grid_ex}
\end{figure}
  
\noindent To each graph grid diagram we associate a transverse spatial graph by connecting the $X$'s to the $O$'s vertically and the $O$'s to the $X$'s horizontally with the convention that the vertical strands go over the horizontal strands.  See Subsection~\ref{subsec:gridtograph} for more details and the figure below for an example. 

  \begin{figure}[h]
\begin{center}
	\begin{picture}(282,120)
\put(0,0){\includegraphics{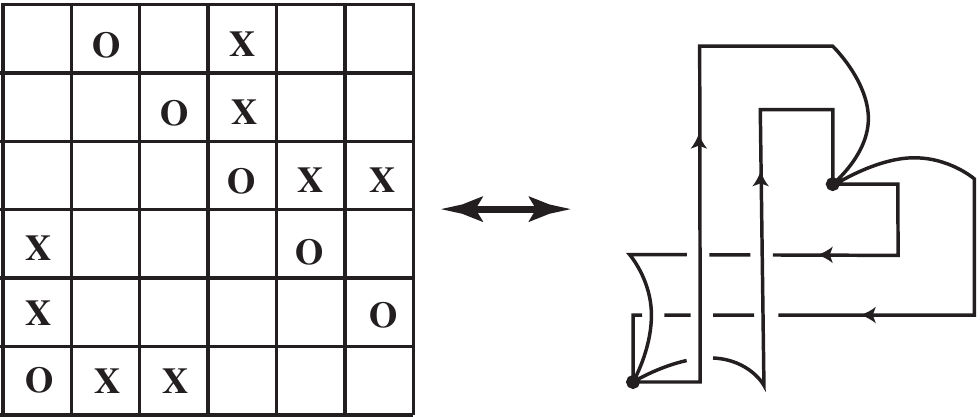}}
\put(13,8){\large{*}}  
\put(71,68){\large{*}}
\end{picture}
\caption{Associating a transverse spatial graph to a graph grid diagram.}\label{grid_ex2}
\end{center}
\end{figure}

We prove that every transverse spatial graph can be represented by a graph grid diagram.  

\newtheorem*{prop_grid_repr_sg_intro}{Proposition~\ref{prop:grid_repr_sg}}
\begin{prop_grid_repr_sg_intro}
	Let $f:G \rightarrow S^3$ be a transverse spatial graph.  Then there is a graph grid diagram $g$ representing $f$.  \end{prop_grid_repr_sg_intro}

\noindent However, this representative is not unique.  We define a set of moves on graph grid diagrams (cyclic permutation, commutation$'$, and stabilization$'$), called graph grid moves, generalizing the grid moves for links.  See Section~\ref{subsec:gridtograph} for their definitions. We prove that any two representatives for the same transverse spatial graph are related by a sequence of graph grid moves.

\newtheorem*{theorem_main}{Theorem~\ref{thm:main_grid}}
  \begin{theorem_main} 	
If $g$ and $g'$ are two graph grid diagrams representing the same transverse spatial graph, then $g$ and $g'$ are related by a finite sequence of graph grid moves.  
 \end{theorem_main}
 
In Section~\ref{sec:def}, to each (saturated) graph grid diagram $g$, we assign a chain complex $(C^-(g),\partial^-)$.  Saturated means that there is at least one $X$ per row and column, and these graph grid diagrams correspond to transverse spatial graphs where the graphs have no sinks nor sources.  We need this technical condition to show that   $\partial^- \circ \partial^-=0$.
The chain groups consists of bigraded free $\mathbb{F}[U_1,\dots, U_V]$-modules  where $V$ is the number of $O$'s decorated with an $\ast$.  
Like in link Floer homology,  the generators of the chain groups are unordered tuples of intersections points between the horizontal and vertical line segments in the grid, with exactly one point on each horizontal and vertical line segment.  
The Maslov grading is defined exactly as in~\cite{MOST}.  Note that this is possible since it only depends on the set of $O$'s on the grid.  For links, the Alexander grading lives in $\mathbb{Z}^m$.  For transverse spatial graphs, we define an Alexander grading that has values in $H_1(S^3\smallsetminus f(G))$ where $f: G \rightarrow S^3$ is the transverse spatial graph associated to $g$.  To compute this, for each point in the lattice of the grid, we define an element of $H_1(S^3 \smallsetminus f(G))$, called the generalized winding number.  It is defined so that if you can get from one point to another by passing an edge of the projection of $f(G)$ coming from $g,$ then the difference between their values is (plus or minus) the homology class of the meridian of that edge.  The Alexander grading of a generator is obtained by taking the sum of the generalized winding numbers of the points of the generator.  Each $U_i$ is associated with an $O$ and we define the Alexander grading so that multiplication by $U_i$ correspond to lowering the Alexander grading by the element of $H_1(S^3 \smallsetminus f(G))$ represented by ``meridian of $O$.'' See Subsection~\ref{subsec:ch_complex} for more details.  The $\partial^-$ map is defined by counting empty rectangles in the (toroidal) grid that do not contain $X$'s.  We show in Section~\ref{sec:def} that $\partial^- \circ \partial^- =0$ and so the homology of $(C^-(g),\partial^-)$ gives a well-defined invariant, for each saturated graph grid diagram.  

For a given transverse spatial graph, there are infinitely many graph grid diagrams representing it.  In Section~\ref{sec:invariance}, we show that the homology of the bigraded chain complex is independent of the choice of saturated graph grid diagram. To prove this, we show that the quasi-isomorphism type of the chain complex is preserved under the three graph grid moves.  

\newtheorem*{theorem_invariance}{Theorem~\ref{thm:invariance}}
\begin{theorem_invariance} If $g_1$ and $g_2$ are saturated graph grid diagrams representing the same transverse spatial graph $f:G \rightarrow S^3$ then $(C^-(g_1),\partial^-)$ is quasi-isomorphic to $(C^-(g_2),\partial^-)$ as relatively-absolutely $(H_1(E(f)),\Z)$ bigraded $\mathbb{F}[U_1, \dots, U_{V}]$-modules.  In particular, $HFG^-(g_1)$ is isomorphic to $HFG^-(g_2)$ as relatively-absolutely $(H_1(E(f)),\Z)$ bigraded $\mathbb{F}[U_1, \dots, U_{V}]$-modules. 
\end{theorem_invariance}

\noindent Thus, we can define the \textbf{graph Floer homology} of the sinkless and sourceless transverse spatial graph $f$, denoted $HFG^-(f)$, to be $HFG^-(g)$ for any saturated grid diagram $g$ representing $f$.  We also define a hat variant, $\widehat{HFG}(f)$ by taking the homology of the chain complex obtained by setting $U_1, \dots, U_V$ to zero.  See Subsection~\ref{subsec:ch_complex} for more details. 

As far as the authors are aware, there have been no papers in the past that have defined an Alexander polynomial for an arbitrary spatial graph that do not just depend on the fundamental group of the exterior. In \cite{Ki58}, Kinoshita defines the Alexander polynomial of an oriented spatial graph as the Alexander polynomial of its exterior.  In contrast, Litherland \cite{Li89} defined the Alexander polynomial of a spatial graph where the underlying graph is a theta graph and his definition does not depend solely on the exterior of the embedding.  At the same time as this paper was being written, Bao, in \cite{Bao14}, independently defined an Alexander polynomial of a balanced bipartite spatial graph and produces a state sum formula for the polynomial.  Her definition is similar to the one we define.  

In Section~\ref{sec:Alex}, we define an Alexander polynomial for any transverse spatial graph $f: G\rightarrow S^3$.  To do this, we associate a balanced sutured manifold $(S^3 \smallsetminus f(G),\gamma(f))$ to $f$.  We then define the Alexander polynomial of $f$, $\Delta_f \in \mathbb{Z}[H_1(S^3 \smallsetminus f(G))]$, to be the torsion invariant associated to balanced sutured manifold $(S^3 \smallsetminus f(G),\gamma(f))$ defined by S. Friedl, A. Juh{\'a}sz, and J. Rasmussen in \cite{FJR11}.  We show that the graded Euler characteristic of $\widehat{HFG}(f)$ is essentially $\Delta_f$.  Note that $\overline{\Delta}_f$ means to send each element in $H_1(S^3 \smallsetminus f(G))$ of $\Delta_f$ to its inverse. 

\newtheorem*{cor_Alex}{Corollary~\ref{cor:Alex}}
\begin{cor_Alex}If $f: G \rightarrow S^3$ is a sinkless and sourceless transverse spatial graph then
	$$ \chi(\widehat{HFG}(f)) \doteq \overline{\Delta}_f. $$ 
	That is, they are the same up to multiplication by units in $\mathbb{Z}[H_1(S^3 \smallsetminus f(G))]$.
\end{cor_Alex}

To prove this, we first prove the stronger result, that the hat version of our graph Floer homology is essentially the same as the sutured Floer homology of $(S^3 \smallsetminus f(G),\gamma(f))$.  Here $rSFH(E(f),\gamma(f))$ means to consider $SFH(E(f),\gamma(f))$ as a bigraded $(H_1(E(f)),\mathbb{Z}_2)$ but where the $H_1(S^3 \smallsetminus f(G))$ Alexander grading is changed by a negative sign. 

\newtheorem*{thm_sfh}{Theorem~\ref{thm:same}}
\begin{thm_sfh}Let  $f: G \rightarrow S^3$ be a sinkless and sourceless transverse spatial graph.  Then 
$$\widehat{HFG}(f) \cong rSFH(E(f),\gamma(f))$$
as relatively $(H_1(E(f)),\mathbb{Z}_2)$ bigraded $\mathbb{F}$-vector spaces.
\end{thm_sfh}

\noindent To complete the proof of Corollary~\ref{cor:Alex}, we use the theorem of S. Friedl, A. Juh{\'a}sz, and J. Rasmussen in \cite{FJR11} that the decategorification of sutured Floer homology is their torsion invariant.   

\vspace{10pt}

\noindent \textbf{Acknowledgements}. We would like to thank Sucharit Sarkar, Andr\'{a}s Juh\'{a}sz, Eamonn Tweedy, Tim Cochran, Katherine Poulsen, Adam Levine, Ciprian Manolescu, Matt Hedden, Robert Lipschitz, and Dylan Thurston for helpful conversations.

\section{Transverse Disk Spatial Graphs}\label{balanced_sp_gr}
Graph Floer homology is a version of Heegaard Floer homology defined for \emph{transverse spatial graphs}.  In this section we will define the term transverse spatial graph and define the notion of  equivalence of transverse spatial graphs.  We will also discuss their diagrams and Reidemeister moves.  

We will work in the PL-category.   A \textbf{graph}, $Y$, is a 1-dimensional complex, consisting of a finite set of vertices ($0$-simplices) and edges ($1$-simplices) between them.  
A \textbf{spatial graph} is an embedding $f$ of a graph $Y$ in $S^3$, $f:Y\rightarrow S^3$.  
A \textbf{diagram} of a spatial graph is a projection of $f(Y)$ to $S^2$ with only transverse double points away from vertices, where the over and under crossings are indicated.  
Two spatial graphs $f_1$ and $f_2$ are \textbf{equivalent} if there is an ambient isotopy between them.  
Notice that the ambient isotopy gives a map $h:S^3\rightarrow S^3$ which sends $f_1(Y)$ to $f_2(Y)$ sending edges to edges and vertices to vertices.  
In \cite{Kauf}, Kauffman showed that: 

\begin{thm}Let $f_1$ and $f_2$ be spatial graphs. 
Then $f_1$ is ambient isotopic to $f_2$ if and only if any diagram of $f_2$ can be obtained from any diagram of $f_1$ by a finite number of graph Reidemeister moves (shown in Figure \ref{Rmoves}) and planar isotopy.
\end{thm}

\begin{figure}[h]
\begin{center}
\includegraphics{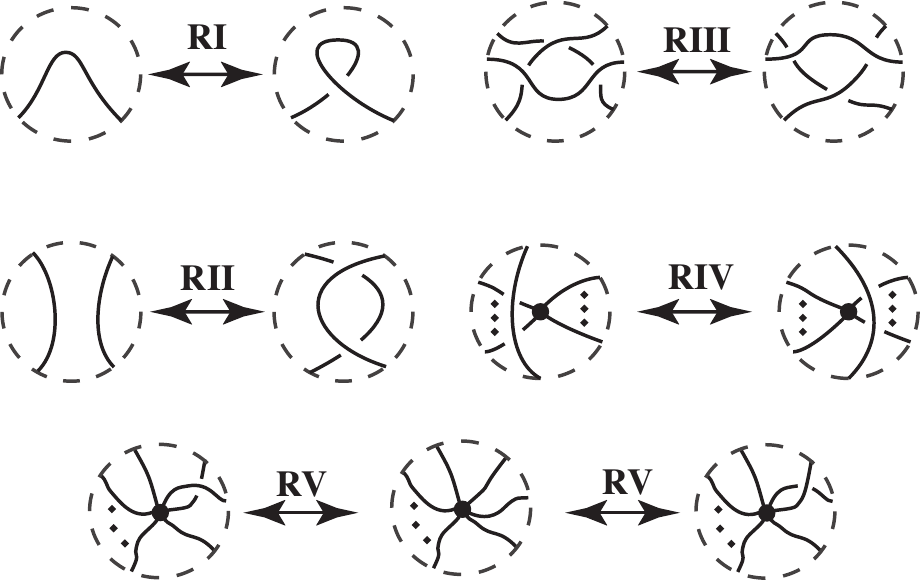}
\caption{\textbf{Graph Reidemeister moves:} Reidemeister moves one, two, and three are the same as those for knots and links.  Reidemeister move RIV, is where an edge can move passed a vertex, either over or under (over is pictured here). Reidemeister move RV is where two of the edges next to the vertex change places.  }\label{Rmoves}
\end{center}
\end{figure}

An \textbf{oriented graph} is a graph together with orientations given on each of the edges.  
Let $\mathcal{D}$ be the 2-complex obtained by gluing $3$ copies of the 2-simplex $[e_0,e_1,e_2]$ together so that their union is a disk and with all three of the $e_0$s identified to a single point in $\mathcal{D}$. Note that $e_0$ is the unique vertex in the interior of $\mathcal{D}$. We say that $\mathcal{D}$ is a \textbf{standard disk} and $e_0$ is the \textbf{vertex associated to $\mathcal{D}$}.
 An \textbf{oriented disk graph, $G$,} is a $2$-complex constructed as follows.  Start with an oriented graph $Y$.  Then, for each vertex $v$ of $Y$, glue a standard disk $\mathcal{D}$ to $Y$ by identifying the vertex associated to $\mathcal{D}$ with $v$. We note that $Y$ is a subset of the oriented disk graph which we call the \textbf{underlying oriented graph} of the oriented disk graph (or the \textbf{underlying graph of $G$} if we do not want to consider the orientations).  
We say that a vertex of an oriented disk graph is a \textbf{graph vertex} (respectively \textbf{graph edge}) if it is a vertex (respectively edge) of its underlying oriented graph.  When it is clear, we will just refer to them as vertices and edges of the oriented disk graph (and will not refer to the other $0$ and $1$-simplices of the oriented disk graph as vertices or edges).  For an oriented disk graph $G$, and a given graph vertex $v$ of $G$, the set of graph edges of $G$ with orientation going towards $v$ are called the \textbf{incoming edges of $v$} and the set of edges with the orientation going away from $v$ are called the \textbf{outgoing edges of $v$}.
\begin{definition}A \textbf{transverse spatial graph} is an embedding $f: G \rightarrow S^3$ of an oriented disk graph $G$ into $S^3$ where each vertex of the graph locally looks like Figure~\ref{tdiskv_fig} and where each standard disk of $G$ lies in a plane.  We call the image of each of the standard disks of $G$, a \textbf{disk} of $f$ and the embedding of the underlying graph of $G$, the \textbf{underlying spatial graph of $f$}.  Two transverse spatial graphs are \textbf{equivalent} if there is an ambient isotopy between them.  

\end{definition}
Note that in a transverse spatial graph the incoming and outgoing edges are each grouped together. 
In the ambient isotopy, at each vertex, both the set of incoming and the set of outgoing edges can move freely.  However, in an ambient isotopy, the incoming and outgoing sets cannot intermingle, because of the disk separates the edges.

\begin{figure}[h]
\begin{center}
\includegraphics[scale=0.5]{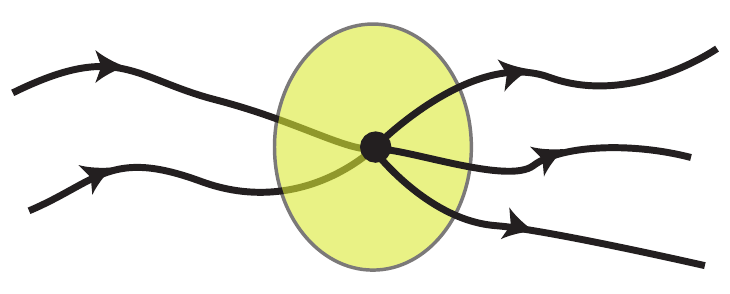}
\caption{A vertex of a transverse spatial graph shown with the disk. }\label{tdiskv_fig}
\end{center}
\end{figure}

\begin{definition}A \textbf{regular projection of a transverse spatial graph $f: G \rightarrow S^3$} is a projection that satisfies the following two conditions.  (1) For each point, $pt$, in the image of the underlying graph of $G$, $f^{-1}(pt)$ contains no more than two points and if $f^{-1}(pt)$ contains two points then neither is a graph vertex. (2) All of the standard disks of $f$ are perpendicular to the plane of projection.  	
	A \textbf{diagram for a transverse spatial graph} $f: G \rightarrow S^3$ is a regular projection of $f$ where all the over and under crossings are indicated.  
\end{definition}	
	
In Figure~\ref{tgraph_diagram_fig}, a diagram of a transverse spatial graph is shown, where the disks are also shown in the projection.  Notice that the incoming edges and outgoing edges are grouped in the projection.    
We will from here forward not indicate disks in diagrams, because the position of the disk is already clear from the diagram.

\begin{figure}[h]
\begin{center}
\includegraphics{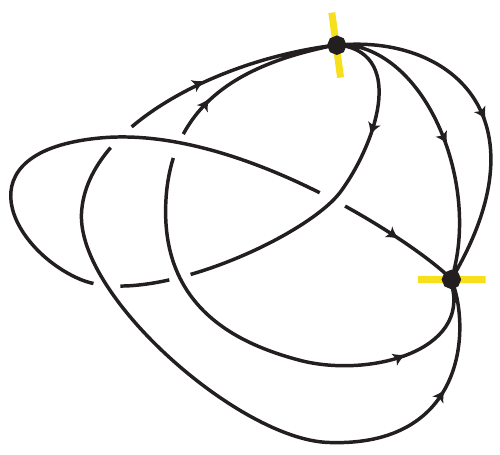}
\caption{A diagram of a transverse spatial graph, where the projection of the disks, transverse to the plane of projection, are shown in gray (yellow).  }\label{tgraph_diagram_fig}
\end{center}
\end{figure}

The Reidemeister moves for transverse spatial graphs are the same as the Reidemeister moves for graphs shown in Figure \ref{Rmoves}, with the restriction that RV may \emph{only} be made between pairs of incoming edges or pairs of outgoing edges.  
For clarity we will call this restriction of  RV, \textbf{R$\overline{\textrm{V}}$}.

\begin{figure}[h]
\begin{center}
\begin{picture}(250, 210)
\put(0,0){\includegraphics[scale=0.5]{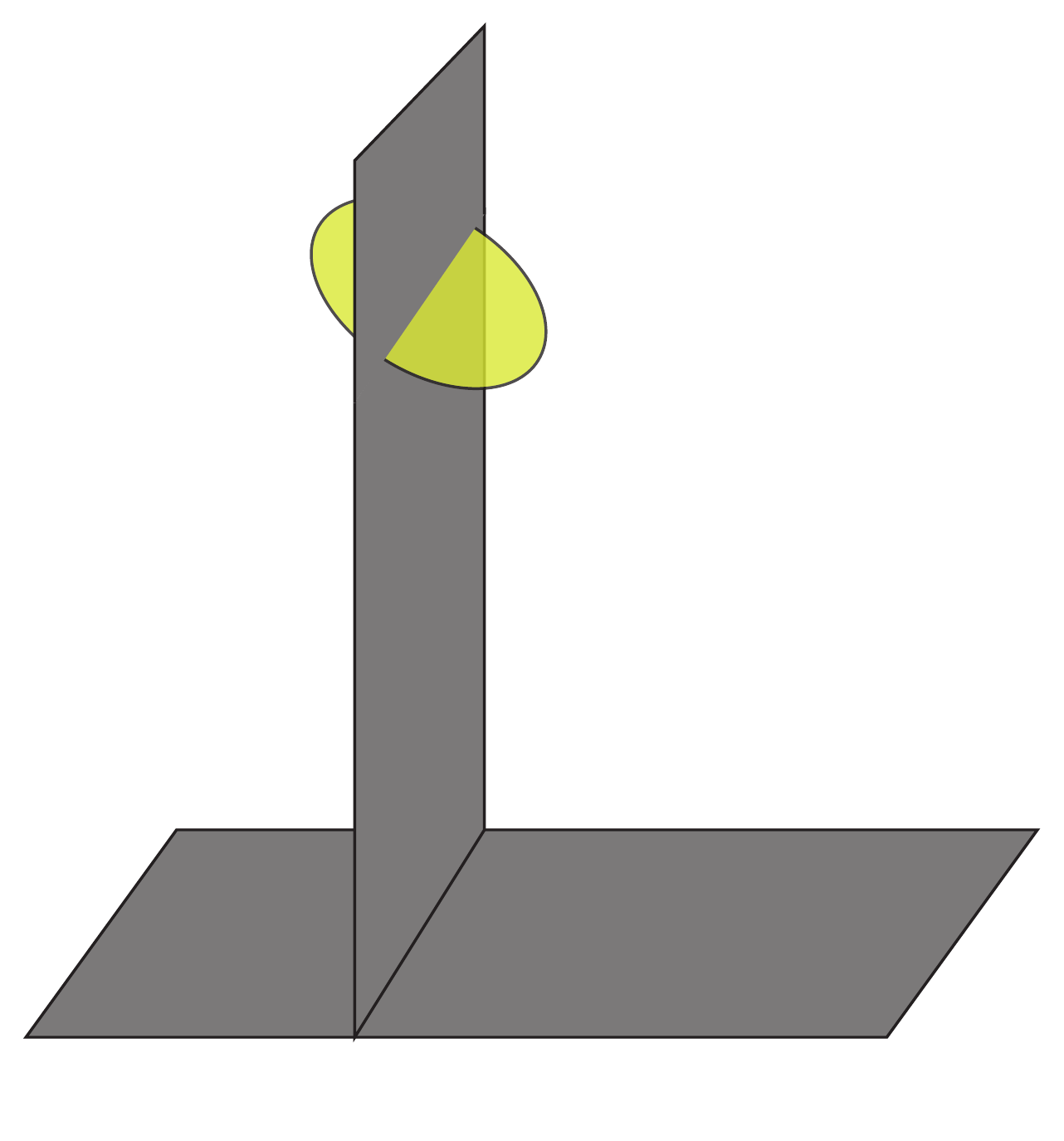}}
\put(90,190){$P_v$}
\put(100,135){$D_v$}
\put(180,60){$P_{proj}$}
\end{picture}
\caption{This shows the plane of projection $P_{proj}$ together with a disk $D_v$ of the vertex $v$ and $P_v$.  
The plane $P_v$ is the plane perpendicular to $P_{proj}$ and meeting $D_v$ in the line parallel to $P_{proj}.$}\label{fig_PvDv}
\end{center}
\end{figure}
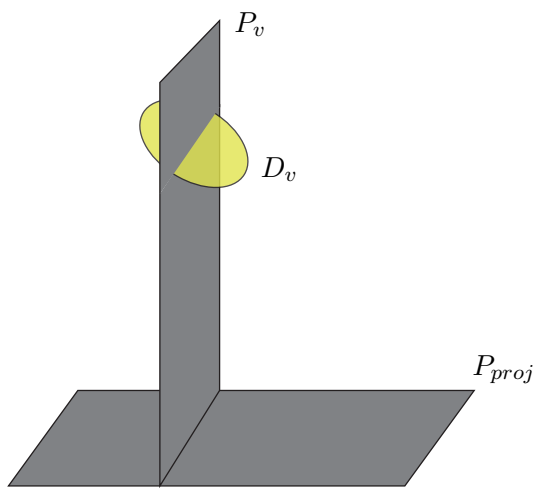

\begin{thm}  Every transverse spatial graph has a diagram.  If two transverse spatial graphs are ambient isotopic, then any two diagrams of them are related by a finite sequence of the Reidemeister moves RI-RIV, R$\overline{\textrm{V}}$, and planar isotopy.  
\end{thm}

\begin{proof}
We first show that every transverse spatial graph has a diagram.  Let $f:G \rightarrow S^3$ be a transverse spatial graph.  A regular projection for the transverse spatial graph is a projection of $f(G)$ which is a regular projection of the underlying spatial graph of $f$.  That is, the projection only has transverse double points and these are away from the vertices.  
We will first obtain a regular projection for the transverse spatial graph, and next find a representative in the ambient isotopy class where the disks are transverse to $P_{proj}$ where $P_{proj}$ denotes the plane of projection for the regular projection.  
A regular projection for a spatial graph is obtained in the usual way; a point projection of a representative of the ambient isotopy class is obtained via $\epsilon$-perturbations of the graph.  
To have the disks perpendicular to $P_{proj}$ a similar process is used. 
If the disk is transverse to $P_{proj}$ we will see that there is a unique way to move it via an ambient isotopy of $f$ to a position where it is perpendicular to $P_{proj}$.  
So we need only have all of the disks transverse to $P_{proj}$.  
For an arbitrary vertex $v$ with disk $D_v$, let $\bf{x}$ be the vector that is perpendicular to $D$ and pointing in the direction of the outgoing edges.  
If $v$ remains in the same place and the neighborhood around it is allowed to rotate, there is a full sphere of directions in which $\bf{x}$ can be pointing.  
Only two of these directions will result in $D_v$ being parallel to $P_{proj}$.  
By dimensionality arguments having the disks transverse to $P_{proj}$ is generic.  If any of the disks are not transverse, an $\epsilon$-perturbation is done.  
For each vertex $v$,  let $P_v$ be the plane that is perpendicular to $P_{proj}$ and meets $D_v$ in the line parallel to $P_{proj}$.  
For each disk transverse to $P_{proj}$ there is a unique map via rotation through the acute angle between $D_v$ and $P_v$ so that it is perpendicular to $P_{proj}$ (see Figure \ref{fig_PvDv}.)  

\begin{figure}[h]
\begin{center}
\begin{picture}(250, 120)
\put(0,10){\includegraphics[scale=0.5]{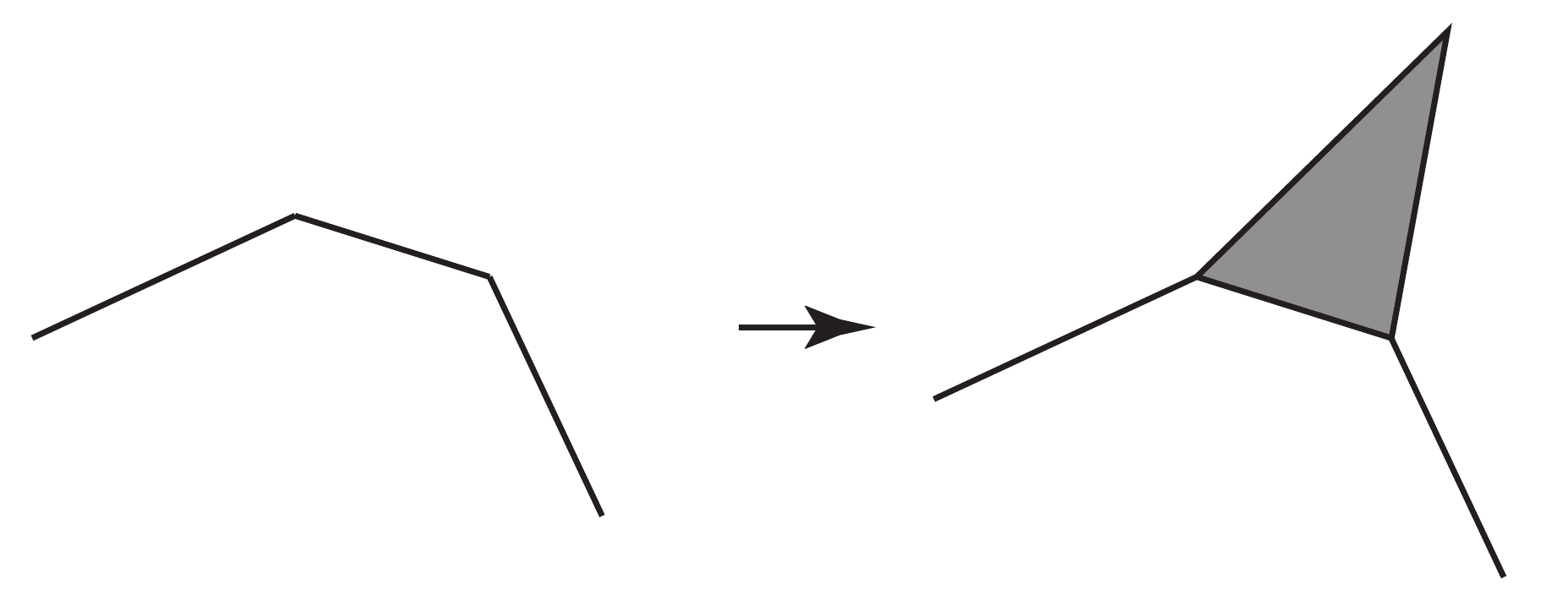}}
\put(50,85){$e_i$}
\put(87,70){$e_j$}
\put(200,75){$e_i$}
\put(245, 55){$e_j$}
\put(250,115){$e_k$}
\end{picture}
\caption{An elementary move.  }\label{fig_ElemMove}
\end{center}
\end{figure}
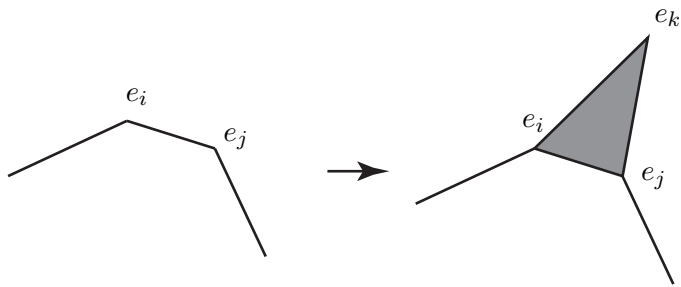
For (topological) spatial graphs, any ambient isotopy is made up of elementary moves.  
Recall that an elementary move of a spatial graph is where a linear segment of an edge $[e_i,e_j]$ is replaced with two new linear segments $[e_i, e_k]$ and $[e_k, e_j]$ if the three segments together bound a $2$-simplex which intersects the original spatial graph only in $[e_i,e_j]$, or the reverse of this move (see Figure \ref{fig_ElemMove}).  
In \cite{Kauf}, Kauffman showed that any elementary move for a spatial graph can be obtained by a sequence of the Reidemeister moves shown in Figure \ref{Rmoves}.  
Now we consider transverse spatial graphs.  An elementary move of a transverse spatial graph $f:G \rightarrow S^3$ is where a linear segment of an edge of the underlying graph $[e_i,e_j]$ is replaced with two new linear segments $[e_i, e_k]$ and $[e_k, e_j]$ if the three segments together bound a $2$-simplex $T$ which intersects $f(G)$ in $[e_i,e_j]$, or the reverse of this move (see Figure \ref{fig_ElemMove}).  We note that $T$ must miss all the transverse disks.  We will show that any ambient isotopy of transverse spatial graphs is made of elementary moves of a transverse spatial graph. 
First note that Reidemeister moves RI -- RIV preserve the isotopy class of a transverse spatial graph.  
In addition, one can still interchange a pair of neighboring incoming and incoming edges or a pair of neighboring outgoing and outgoing edges in RV.  
However if one tried to interchange a neighboring incoming and outgoing edge at the vertex $v$, the disk from the elementary move that would result in RV would intersect the transverse disk $D_v$.  
So Reidemeister move RV is restricted to move R$\overline{\textrm{V}}$. Recall that R$\overline{\textrm{V}}$ is the move RV where only neighboring incoming (respectively outgoing) edges are interchanged.

\begin{figure}[h]
\begin{center}
\includegraphics[scale=0.5]{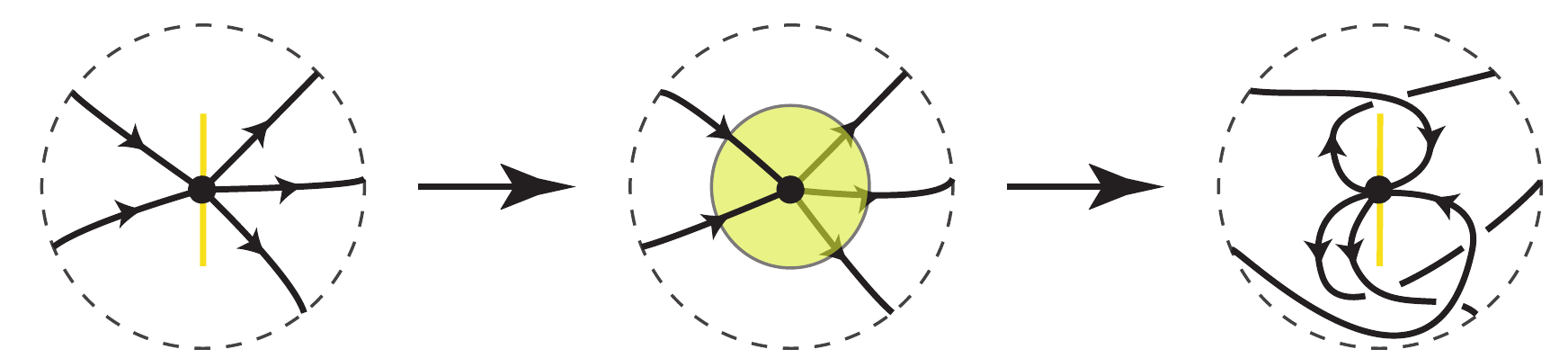}
\caption{An example of what can happen when a disk flips over, moving through a position parallel to $P_{proj}$.  }\label{fig_DiskFlip}
\end{center}
\end{figure}

We claim that one needs only Reidemeister moves RI -- RIV, and R$\overline{\textrm{V}}$ to get all ambient isotopies.  
One might be concerned that this is incomplete because of the danger of a vertex flipping over (i.e. moving through a position where the disk is parallel to the plane of projection) resulting in a change in the diagram like that shown in Figure \ref{fig_DiskFlip}.  
However, this move and any move like it can be obtained with the set of Reidemeister moves RI, RII, RIII, RIV, and R$\overline{\textrm{V}}$.  
To discuss this we will introduce another type of graph.  
A \textbf{flat vertex graph} or \textbf{rigid vertex graph} is a spatial graph where the vertices are flat disks or polygons with edges attached along the boundary of the vertex at fixed places.  
The set of Reidemeister moves for flat vertex graphs is RI -- RIV as before, and the move RV* \cite{Kauf}.  
Reidemeister move RV* is where the flat vertex is flipped over $180^{\circ}$ (see Figure \ref{fig_RV*}).  
\begin{figure}[h]
\begin{center}
\includegraphics{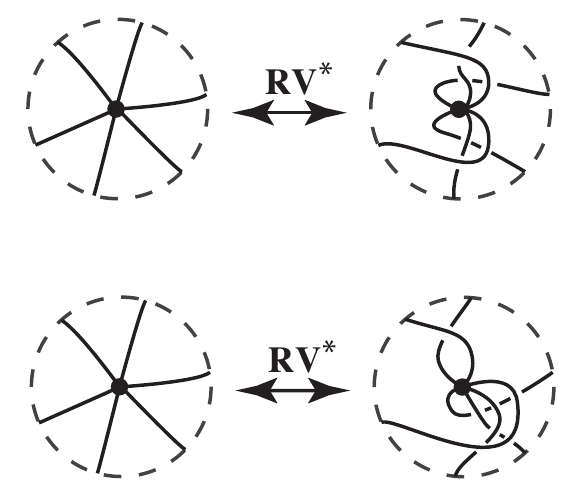}
\caption{Two different RV* moves for flat vertex graphs.  }\label{fig_RV*}
\end{center}
\end{figure}
For the move RV* a choice is made of how many edges are on each side when the vertex is flipped, but only one of these moves is need together with Reidemeister moves RI -- RIV to do any of the other ones \cite{Kauf} (see Figure \ref{fig_betweenRV*}).
In the case of transverse spatial graphs repeated use of R$\overline{\textrm{V}}$ will result in what looks like a RV* move, shown in Figures \ref{fig_RV*InMany} and \ref{fig_RV*WeHave}.  
Thus the  vertex flipping over can be accomplished by Reidemeister moves RI -- RIV, and R$\overline{\textrm{V}}$.
\end{proof}

\begin{figure}[h]
\begin{center}
\includegraphics[scale=0.5]{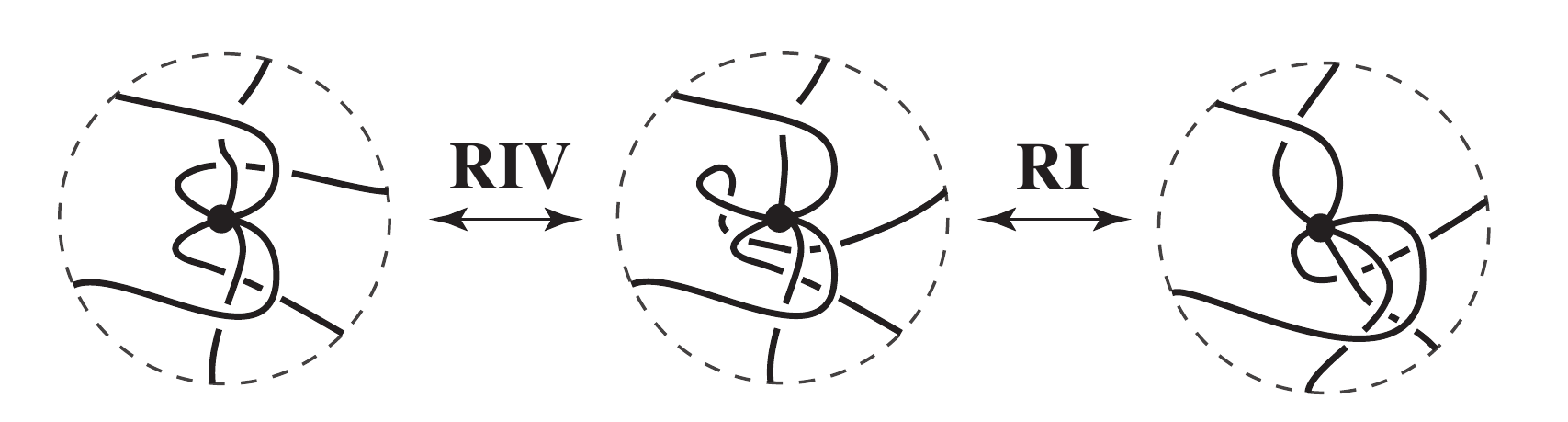}
\caption{This shows how to move between the two different RV* moves shown in Figure \ref{fig_RV*}.  
So only one RV* move is needed for each valence of vertices.  }\label{fig_betweenRV*}
\end{center}
\end{figure}

\begin{figure}[h]
\begin{center}
\includegraphics[scale=0.5]{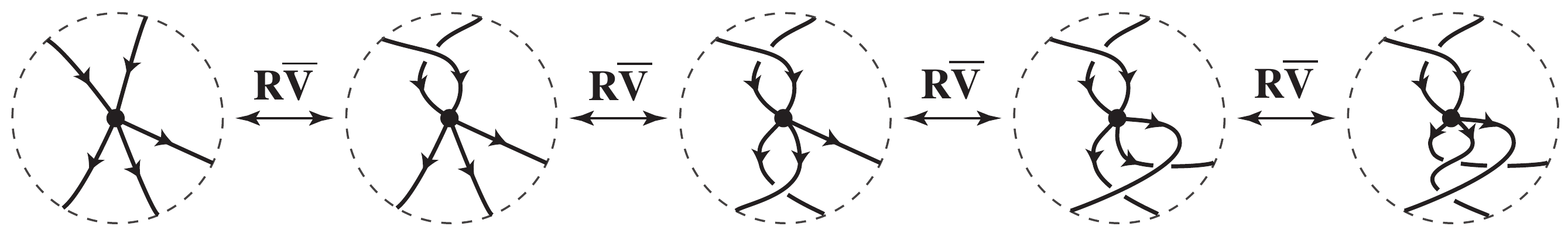}
\caption{This shows how many R$\overline{\textrm{V}}$ moves will give a RV* move.  }\label{fig_RV*InMany}
\end{center}
\end{figure}

\begin{figure}[h]
\begin{center}
\includegraphics[scale=0.5]{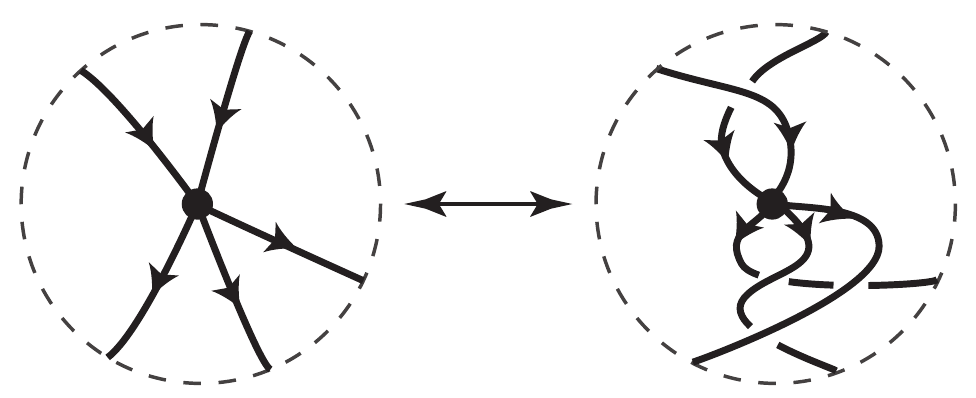}
\caption{The RV* move that is a result of many R$\overline{\textrm{V}}$ moves.  }\label{fig_RV*WeHave}
\end{center}
\end{figure}

\section{Graph Grid Diagrams}\label{sec:grid}

In this section, we define the notion of a graph grid diagrams and explain their relationship to transverse spatial graphs.  To each graph grid diagram we associate a unique transverse spatial graph.  On the other hand, we show that every transverse spatial graph can be represented by a non-unique graph grid diagram.  As with grid diagrams for knots and links, we define a set of moves on graph grid diagrams (cyclic permutation, commutation$'$, and stabilization$'$) that we call graph grid moves.  
Finally, we prove that any two graph grid diagrams representing the same transverse spatial graph are related by a sequence of graph grid moves.

\subsection{Graph Grid Diagrams}\label{subsec:graphgrid}

We will assume that the reader is familiar with grid diagrams for knots and links, see \cite{MOS, MOST}.  
Recall that a (planar) grid diagram for a link is an $n \times n$ grid of squares in the plane where each square is decorated with an $X$, and $O$, or nothing and such that every row (respectively column) contains exactly one $X$ and exactly one $O$.  
Here, we are using the notation in \cite{MOST}.  
To each grid diagram, one can associate a planar link diagram by drawing horizontal line segments from the $O$'s to the $X$'s in each row, and vertical line segments from the $X$'s to the $O$'s in each column with the convention at the crossings that a vertical segment always goes over a horizontal segment.  
We will define a more general class of grid diagrams that will represent transverse spatial graphs.  
Before defining a grid diagram we need a technical definition.

\begin{definition} Suppose $D$ is a $n$ by $n$ grid where each square is decorated with an $X$, $O$ or is empty.  We let $\mathbb{X}$ be the set of $X$'s and $\mathbb{O}$ be the set of $O$'s. We say that two elements $p, q \in \mathbb{X} \cup \mathbb{O}$ are related if $p$ and $q$ share a row or column.  Let $\sim$ be the equivalence relation generated by this relation.  We define the \textbf{connected components} of $D$ to be the equivalence classes of $\sim$.
\end{definition}

\begin{definition} A \textbf{graph grid diagram}, $g$, is an $n$ by $n$ grid where each square is decorated with an $X$, $O$, or is empty, a subset of the $O$'s are decorated with an $\ast$, and that satisfies the following conditions.   
There is exactly one $O$ in each row and column.
Each connected component contains at least one $O$ decorated with an $\ast$.
If a row (respectively column) does not contain exactly one $X$ then the $O$ in that row (respectively column) must be decorated with an $\ast$.  
The total number of rows (or the total number of columns), $n$, is called the \textbf{grid number} of $g$.  
The $O$'s decorated with an $\ast$ are called \textbf{vertex $O$'s}. We will say that an $O$ is \textbf{standard}, if the $O$ has exactly one $X$ in its row and exactly one $X$ in its column; otherwise we say it is \textbf{nonstandard}. Often, it will be convenient to number the $O$'s and $X$'s by $\{O_i\}_{i=1}^n$ and $\{X_i\}_{i=1}^m$.  When numbering, we always assume that $O_1, \dots, O_V$ correspond to vertex $O$'s.  
\end{definition}

For convenience, we may sometimes omit the $\ast$ from a figure when it is clear which $O$'s should have an $\ast$ (the nonstandard ones).  It will also be convenient to think of the grid as the set $[0,n] \times [0,n]$ in the plane, with vertical and horizontal grid lines of the form $\{i\} \times [0,n]$ and $[0,n] \times \{i\}$ where $i$ is an integer from $0$ to $n$, and the $X$'s and $O$'s are at half-integer coordinates. 

As in \cite{MOS,MOST}, our chain complex is obtained from a graph grid diagram, with the main difference being the definition of the Alexander grading.   To define this, it is sometimes necessary to consider toroidal graph grid diagrams instead of (planar) graph grid diagrams.  A \textbf{toroidal graph grid diagram} is a graph grid diagram that is considered as being on a torus by identifying the top and bottom edges of the grid and identifying the left and right edges of the grid.  We denote the toroidal graph grid diagram by $\T$.   We view the torus as being oriented and the orientation being inherited from the plane.  When the context is clear, we will just call it a graph grid diagram.   In a toroidal graph grid diagram, the horizontal and vertical gridlines, 
become circles.  We denote the horizontal circles by $\alpha_1, \dots, \alpha_n$, the vertical circles $\beta_1, \dots, \beta_n$ and we let $\alphas=\{\alpha_1, \dots, \alpha_n\}$ and $\betas=\{\beta_1, \dots, \beta_n\}$. When the grid is drawn on a plane, by convention, we will order the horizontal (respectively vertical) circles from bottom to top (respectively left to right) so that the leftmost circle is $\beta_1$ and the bottommost circle is $\alpha_1$.  
Note, to get a (planar) diagram from a toroidal diagram, one takes a \textbf{fundamental domain} for the torus and cuts along a horizontal and vertical grid circle
and identifies it with $[0,n)\times [0,n)$.

\subsection{Graph Grid Diagrams to Transverse Spatial Graphs and Their Diagrams}\label{subsec:gridtograph} Let $g$ be a graph grid diagram.  We can associate a transverse spatial $f$ to $g$ as follows.  First put a vertex at each of the $O$'s that are decorated with an $\ast$.  Let $O_i$ be an $O$ in $g$ lying in row $r_i$ and column $c_i$.  For each $X_j$ in row $r_i$, connect $O_i$ to $X_j$ with an arc inside of the row (oriented from $O_i$ to $X_j$) so that it is disjoint from all the $X$'s and $O$'s and so that all the arcs in row $r_i$ are disjoint from one another.  We will call these horizontal arcs.  Now push the interior of the arcs in row $r_i$ slightly upwards, above the plane.  For each $X_j$ in column $c_i$, connect $X_j$ to $O_i$ with an arc inside of the column (oriented from $X_j$ to $O_i$) so that it is disjoint from all the $X$'s and $O$'s and so that all the arcs in row $c_i$ are disjoint from one another.  We will call these vertical arcs. Now push the interior of the arcs in column $c_i$ slightly downward, below the plane.  Put a disk in the square containing $O_i$ with $O_i$ at its center. In this case we say that the \textbf{graph grid diagram $g$ represents the spatial graph $f$}.  Any choice of arcs gives the same transverse spatial graph.  

Note that the aforementioned procedure will actually give us a (non-unique) projection of the transverse spatial graph.  However, this will not be a diagram of $f$ since the transverse disks will be parallel to the plane of projection.  It will be convenient for us to define a class of grid diagram that give a well define diagram of a transverse spatial graph when following this procedure.  

Consider a graph grid diagram. For a nonstandard $O$, let the set of $X$'s that appear in a row or column with this $O$ be  called its \textbf{flock}.  
If the $X$'s in the flock of an $O$ are all adjacent to the $O$ or adjacent to other $X$'s that are adjacent to the $O$, then the flock is said to be \textbf{clustered}.  
A flock is in \textbf{L-formation}, if the $X$'s are all to the right and above the $O$.  
It should be noted that the choice of having the $X$'s above and to the right of the $O$ is arbitrary, any pair of below and to the right, above and to the left, or below and to the left will work similarly.
A \textbf{preferred graph grid diagram} is a graph grid diagram where all nonstandard $O$'s (with more than one $X$ in it's flock) have their flocks in L-formation.  

Now suppose that $g$ is a preferred graph grid diagram.  We follow the procedure in the first paragraph of this subsection except that now we put the transverse disks perpendicular to the plane so that they divide the horizontal and vertical arcs.  Moreover, to get a unique diagram $D$, the ordering of the edges around the vertex for nonstandard $O$'s is given by the convention illustrated in Figure \ref{gridvertex}.  In this case we say that the \textbf{graph grid diagram $g$ represents the diagram of the transverse spatial graph, $D$}.
We note that the transverse spatial graph associated to this diagram is equivalent to the transverse spatial graph obtained by following the procedure in the first paragraph of this subsection. 
\begin{figure}[h]
\begin{center}
\begin{picture}(250, 90)
\put(0,10){\includegraphics[scale=0.5]{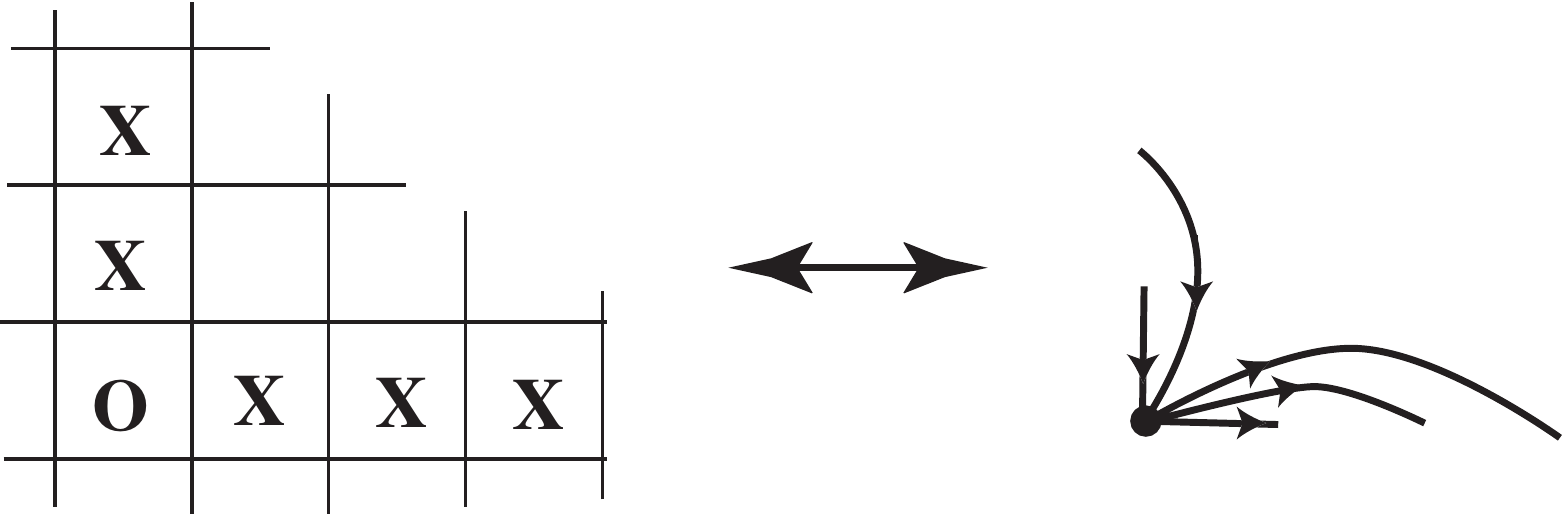}}
\put(20,25){\large{*}}
\end{picture}
\caption{On the left there is a graph grid diagram, of a nonstandard $O$ with the flock in L-formation.  On the right is the diagram of the associated vertex for this portion of the graph grid diagram.  This shows the order in which the edges appear around the vertex. }\label{gridvertex}
\end{center}
\end{figure}

\begin{figure}[htpb!]
\begin{center}
\begin{picture}(250, 110)
\put(0,10){\includegraphics[scale=0.5]{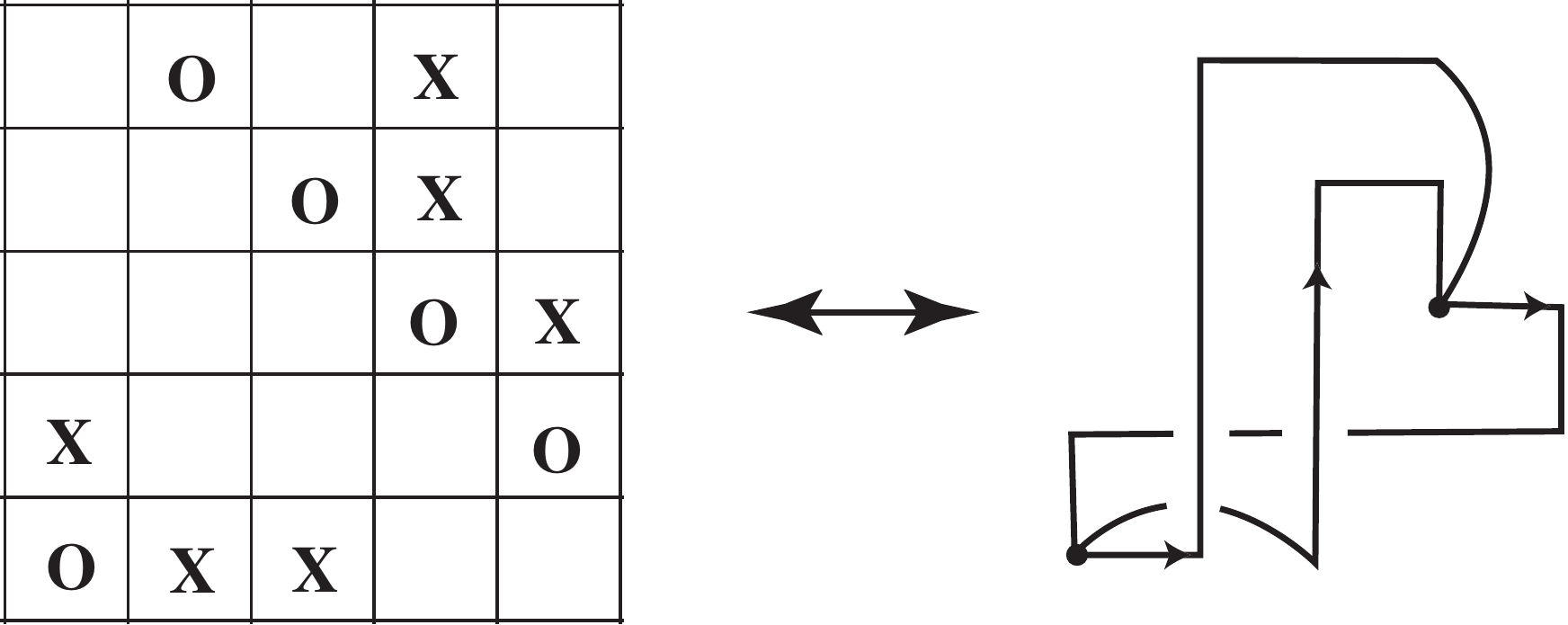}}
\put(14,19){*}
\put(72,58){*}
\end{picture}
\caption{On the left there is a preferred grid diagram.  
On the right is the diagram of the transverse spatial graph associated with this graph grid diagram.   }\label{fig:grid}
\end{center}
\end{figure}

\subsection{Graph Embedding to Preferred Graph Grid Diagram.}\label{ssec:Gtog}

We have shown that to each graph grid diagram, we can associate a transverse spatial graph.  We now show that for each transverse spatial graph, there is a (preferred) graph grid diagram representing it.  

\begin{proposition}\label{prop:grid_repr_sg}Let $f:G \rightarrow S^3$ be a transverse spatial graph.  Then there is a preferred graph grid diagram $g$ representing $f$.  Moreover, for each diagram of a transverse spatial graph, $D$, there is a preferred graph grid diagram $g$ representing $D$.   \end{proposition}

\begin{proof}Chose a diagram $D$ of the transverse spatial graph $f$.  We construct a graph grid diagram $g$, representing $D$, by the following procedure.
At the vertices the edges are partitioned into two sets, incoming edges and outgoing edges, as $D$ is a diagram of a transverse spatial graph.  
Move the edges around each vertex (and perhaps the disk) by planar isotopy so that all outgoing edges are to the right of the vertex and all incoming edges are above the vertex, as shown in Figure \ref{gridvertex}.  
Away from vertices the process is the same as that for knots or links.  
The arcs of the edges are made ``square."  
All crossings are made so that the horizontal arc goes under the vertical arc (see Figure \ref{fig:sqcrossing}).  
Then the diagram is moved via a planar isotopy so that no vertical arcs or  vertices with their incoming edge arcs are in the same vertical line, and similarly for horizontal arcs and vertices with their outgoing edge arcs.  
A vertex along with its incoming edge arcs are associated with a single column, and the vertex together with its outgoing edge arcs are associated with a single row (see Figure \ref{gridvertex}).   
Each of the vertical and horizontal arcs are also associated with a column and  row of the grid, respectively.  
This will result in an equal number of rows and columns.  
Each vertex will add a row and a column.  
Each vertical arc that are not next to a vertex will add a column.  
Each vertical arc can be paired with the following horizontal arc that is not next to a vertex, which will add a row.  
A graph grid representation is then given by placing $X$'s and $O$'s on the $n$ by $n$ grid.  
At vertices a single $O^\ast$ is placed at the vertex, then an $X$ is placed in the same row at the corner of each of the outgoing edges and an $X$ is placed in the same column at the corner for each of the incoming edge, this is done as shown in Figure \ref{gridvertex}.  
Next $X$'s and $O$'s are placed along the edges at the corners consistent with the orientation, arcs go from $X$'s to $O$'s in columns and for $O$'s to $X$'s in rows.  

\end{proof}

\begin{figure}[h]
\begin{center}
\includegraphics[scale=0.4]{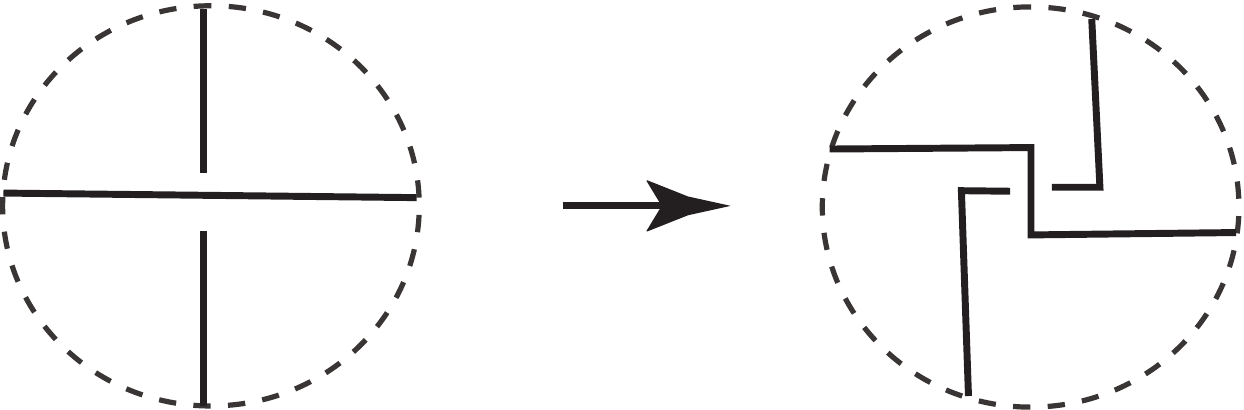}
\caption{This shows how to change to a diagram with vertical over-crossings without changing the embedded graph.  }\label{fig:sqcrossing}
\end{center}
\end{figure}

\subsection{Grid Moves}
Following Cromwell \cite{Crom} and Dynikov \cite{Dy}, any two grid diagrams of the same link are related by a finite sequence of grid moves: 
\begin{description}
\item[Cyclic permutation] The rows and columns can be cyclically permuted (see Figure \ref{permutation}).
\begin{figure}[htpb!]
\begin{center}
\includegraphics{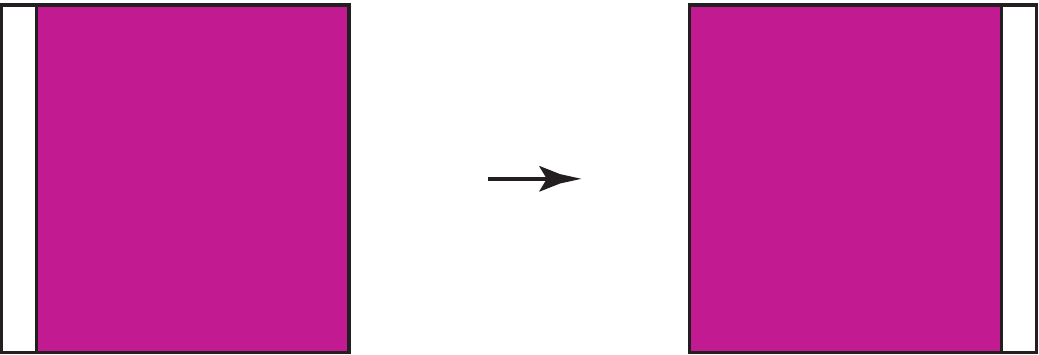}
\caption{An example of cyclic permutation of columns.  }\label{permutation}
\end{center}
\end{figure}

\item[Commutation] Pairs of adjacent columns (respectively rows) may be exchanged when the following conditions are satisfied.  For columns, the four X’s and O’s in the adjacent columns must lie in distinct rows, and the vertical line segments connecting O and X in each column must be either disjoint or nested (one contained in the other) when projected to a single vertical line.  There is an obvious analogous condition for rows (see Figure \ref{comm}).  

\begin{figure}[htpb!]
\begin{center}
\includegraphics{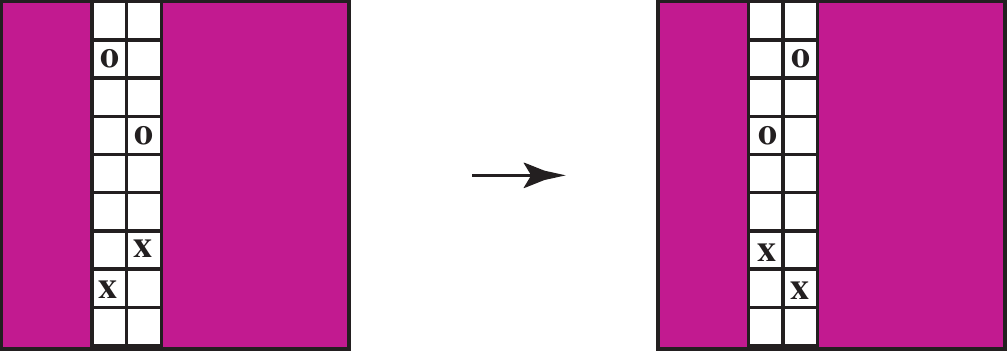}
\caption{An example of commutation of columns.  }\label{comm}
\end{center}
\end{figure}
\item[Stabilization/Destabilization]Let $g$ be an $(n-1)\times (n-1)$ graph grid diagram with decorations $\{O_i\}_{i=1}^{n-1}$ and $\{X_j\}_{j=1}^{n-1}$.  Then $\bar{g}$, an $n \times n$ graph grid diagram, is a stabilization of $g$ if it is obtained from $g$ as follows. Suppose there is a row of $g$ that contains $O_i$ and $X_j$.  In $\bar{g}$, we replace this one row with two new rows and add one new column.  We place $O_i$ into one of the new rows (and in the same column as before) and $X_j$ into the other new row (and in the same column as before).  We place decorations $O_{n}$ and $X_{n}$ into the new column so that $O_{n}$ occupies the same row as $X_j$ and $X_{n}$ occupies the same row as $O_i$. See Figure \ref{stab} for an example.  There is a similar move where the roles of columns and rows are interchanged. 
A destabilization is the reverse of a stabilization.
\end{description}

\begin{figure}[htpb!]
\begin{center}
\includegraphics{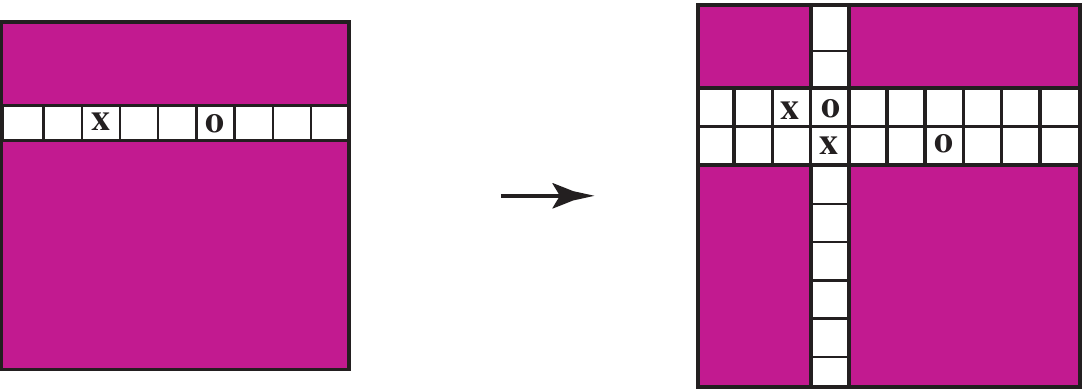}
\caption{An example of stabilization.  }\label{stab}
\end{center}
\end{figure}

For the graph grid moves there are two differences.  
We will replace the usual commutation with a slightly more general commutation$'$ to include exchanging neighboring columns which have entries in the same row (or rows with entries in the same columns) and to include exchanges of rows and columns that have more than a single $X$ in them (or no $X$'s).  
We will also restrict the stabilization/destabilization move to only occur along edges (which we explain below).

\vspace{12pt}
\noindent{\bf The Graph Grid Moves:}\\
Cyclic permutation is unchanged.
\begin{description}
\item[Cyclic permutation] The rows and columns can be cyclically permuted (see Figure \ref{permutation}).
\end{description}
Commutation will be replaced with with the more general commutation$'$.
\begin{description}
\item[Commutation$'$]  
Pairs of adjacent columns may be exchanged when the following conditions are satisfied.  There are vertical line segments $LS_1$ and $LS_2$ on the torus such that (1) $LS_1 \cup LS_2$ contain all the $X$'s and $O$'s in the two adjacent columns, (2) the projection of $LS_1 \cup LS_2$ to a single vertical circle $\beta_i$ is $\beta_i$, and (3) the projection of their endpoints, $\partial(LS_1) \cup \partial(LS_2)$, to a single $\beta_i$ is precisely two points.  Here we are thinking of $\mathbb{X}$ and $\mathbb{O}$ as a collection of points in the grid with half-integer coordinates.  There is an obvious analogous condition for rows (see Figure \ref{comm'}).
\end{description}
\begin{figure}[htpb!]
\begin{center}
\begin{picture}(400, 220)
\put(0,0){\includegraphics{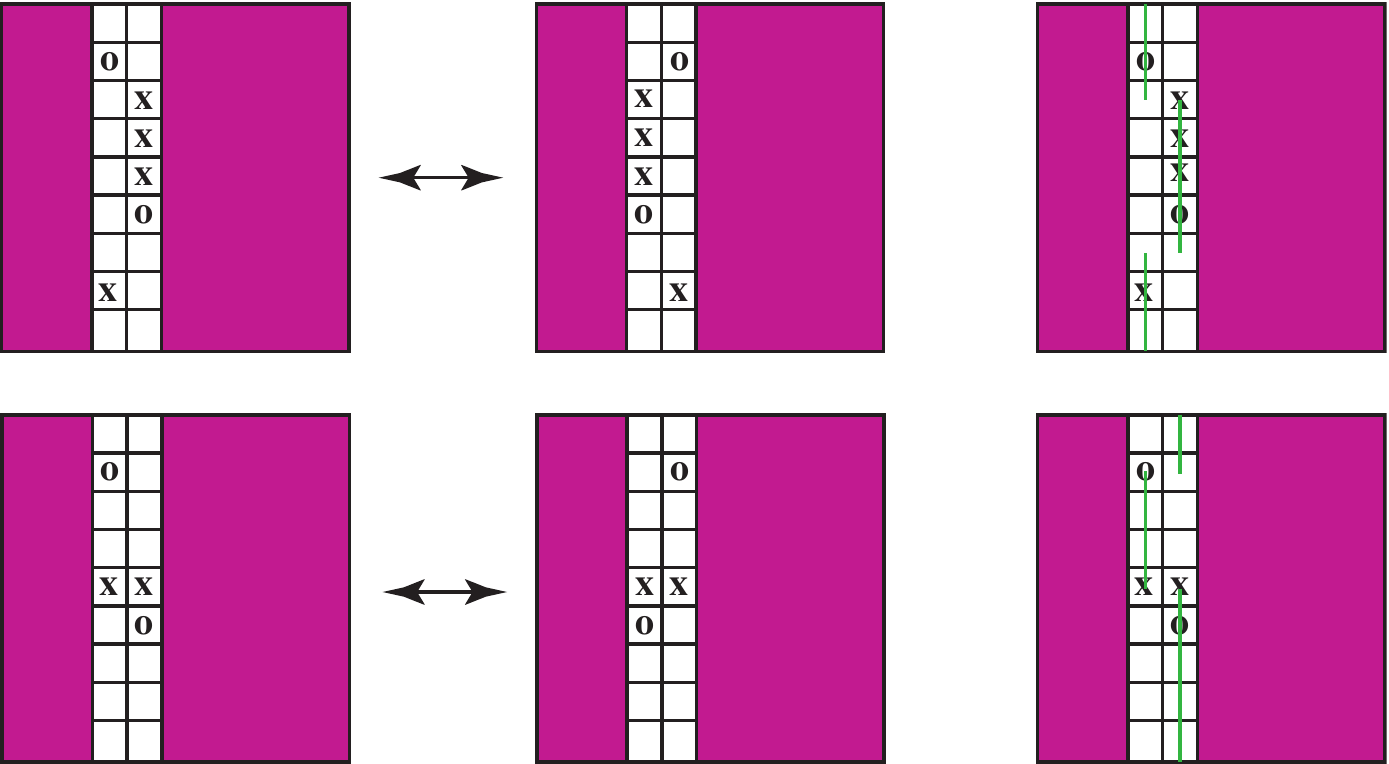}}
\put(43,158){\tiny{*}}
\put(187,158){\tiny{*}}
\end{picture}
\caption{Two examples of commutation$'$ moves of columns.  }\label{comm'}
\end{center}
\end{figure}
We define a generalization of stabilization, called stabilization$'$.  This move will add a jog or a nugatory crossing to the edge of the projection of the associated transverse spatial graph.  
\begin{description}
\item[Stabilization$'$/Destabilization$'$] 	Let $g$ be an $(n-1)\times (n-1)$ graph grid diagram with decorations $\{O_i\}_{i=1}^{n-1}$ and $\{X_j\}_{j=1}^{m-1}$.  Then $\bar{g}$, an $n\times n$ graph grid diagram, is a row stabilization$'$ of $g$ if it is obtained from $g$ as follows. Suppose there is a row of $g$ that contains the decorations $O_k, X_{j_1}, \dots X_{j_l}$ with $l\geq 1$.  In $\bar{g}$, we replace this one row with two new rows and add one new column.  We place $O_k, X_{j_2}, \dots, X_{j_l}$ into one of the new rows (and in the same column as before) and $X_{j_1}$ into the other new row (and in the same column as before).  We place decorations $O_{n}$ and $X_{m}$ into the new column so that $O_{n}$ occupies the same row as $X_{j_1}$ and $X_{m}$ occupies the same row as $O_k$. See Figure~\ref{stabnew} for an example.  
	A column stabilization$'$ is a row stabilization$'$ where one reverses the roles of rows and columns.  We say that $\bar{g}$ is obtained from $g$ by a \textbf{stabilization$'$} if it is obtained by a row or column stabilization$'$.   A destabilization$'$ is the reverse of a stabilization$'$.
\end{description}

Note that $O_{n}$ will not be associated to a vertex so will not be decorated with an $\ast$.  Also, if any $O_i$ is  decorated with an $\ast$ (including $O_k$) in $g$ then it will also be decorated with an $\ast$ in $\bar{g}$.  We do not allow stabilization$'$ of rows with no $X$'s in them.

\begin{rem}\label{rem:stab_simple}If $\bar{g}$ is obtained as row stabilization$'$ on the graph grid diagram $g$ then one can use multiple commutation$'$ moves to change $\bar{g}$ into a row stabilization$'$ obtained from $g$ where $X_{j_1}, X_{m}, O_{n}$ share a corner, $X_{j_1}$ is directly to the left of $O_{n}$, and $O_n$ is directly above $X_m$ (as in Figure~\ref{stab_prime}).  Note that by using only commutation, like in \cite{MOST}, one can only assume that $X_{j_1}, X_{m}, O_{n}$ share a corner which leaves you with four cases instead of one.  This will allow us to simplify the proof of stabilization$'$.   There is a similar statement for column stabilization$'$.   
\end{rem}

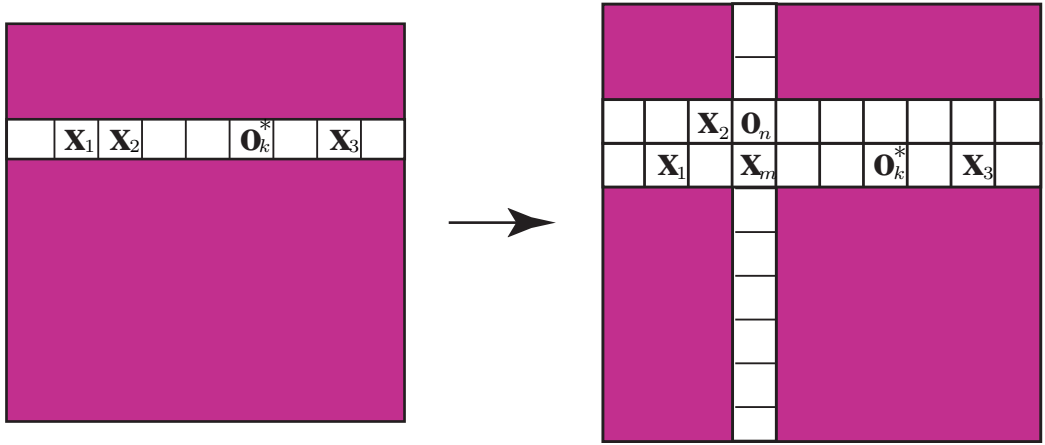
\begin{figure}[htpb!]
\begin{center}
\begin{picture}(391,167)
\put(0,0){\includegraphics{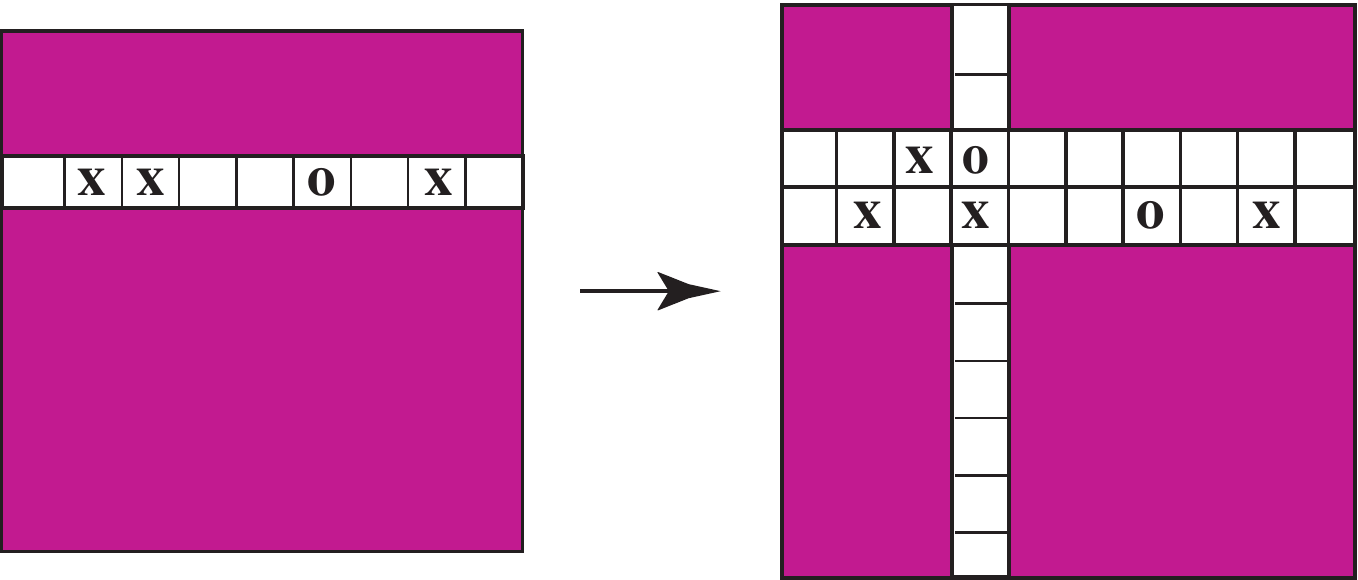}}
\put(30,110){\tiny $1$}
\put(47,110){\tiny $2$}
\put(96,110){\tiny $k$}
\put(96,118){\small $\ast$}
\put(130,110){\tiny $3$}
\put(254,100){\tiny $1$}
\put(270,115){\tiny $2$}
\put(285,115){\tiny $n$}
\put(285,100){\tiny $m$}
\put(336,100){\tiny $k$}
\put(336,108){\small $\ast$}
\put(370,100){\tiny $3$}
\end{picture}
\caption{An example of stabilization$'$.}\label{stab_prime}
\label{stabnew}
\end{center}
\end{figure}

\vspace{12pt}

\subsection{The Graph Grid Theorem}
Before the main theorem of this section we need a lemma.  
To each diagram of a transverse spatial graph $f$ there are an infinite number of different graph grid diagrams representing $f$ that can be constructed using the procedure described in the proof of Proposition~\ref{prop:grid_repr_sg}.  
This procedure produces a preferred grid diagram. However, doing a graph grid move on a preferred graph grid diagram will result in diagrams that are not necessarily in preferred form.  Moreover, if one chooses a random graph grid diagram representing a transverse spatial, it will not necessarily be in preferred form.  Indeed, in practice, one can often reduce the size of the grid number by making it not preferred.  

\begin{figure}[htpb!]
\begin{center}
\begin{picture}(500,140)
\put(0,0){\includegraphics[scale=0.45]{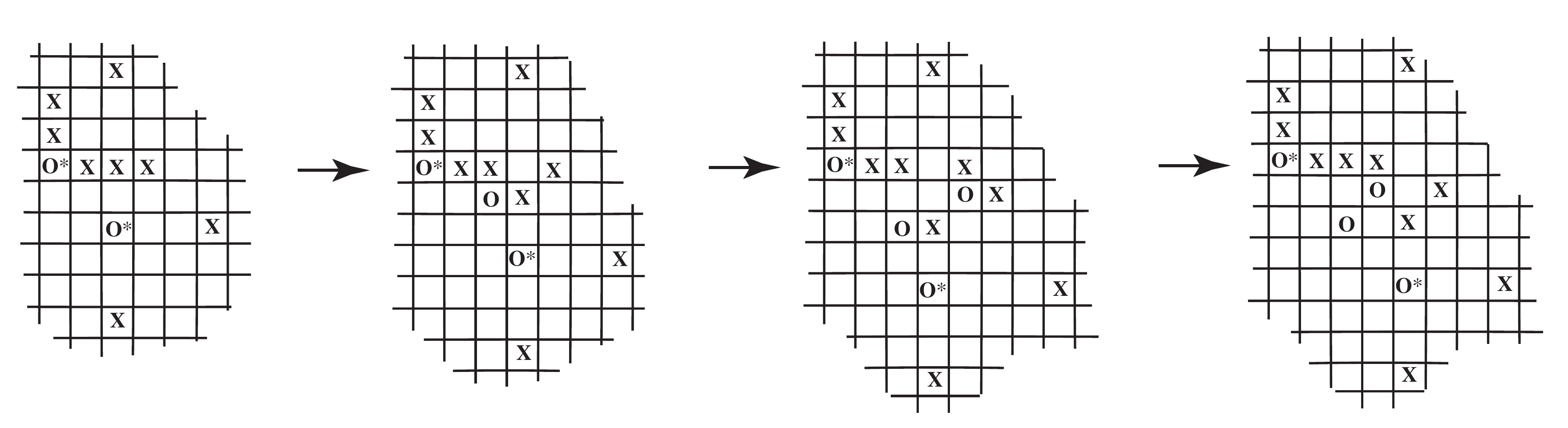}}
\put(85,82){\tiny {\bf STAB$'$}}
\put(204,82){\tiny {\bf STAB$'$}}
\put(330,82){\tiny {\bf COMM$'$}}
\end{picture}
\caption{This shows an example of the moves needed to separate the flocks of two $O^\ast$ and move the upper most flock back into L-formation.  }
\label{fig:sharedX}
\end{center}
\end{figure}
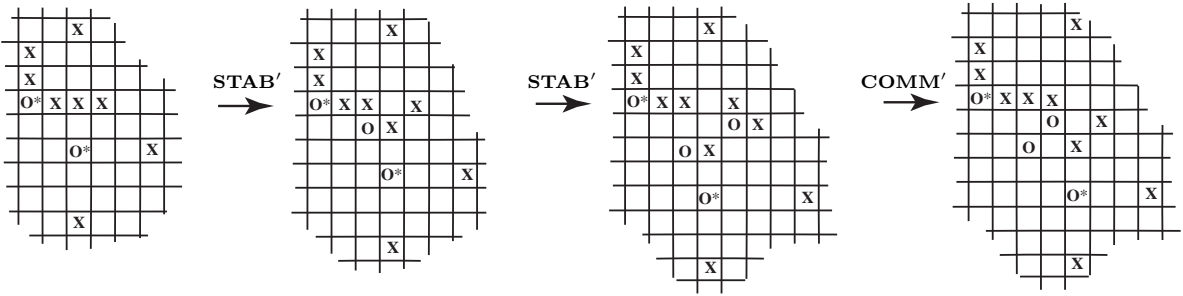

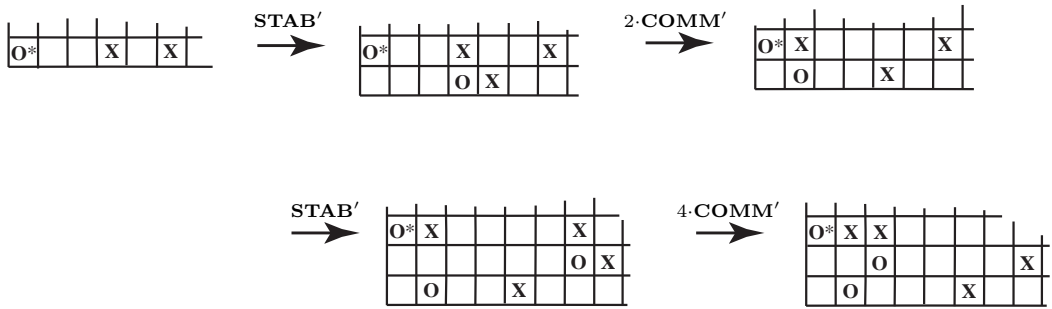
\begin{figure}[htpb!]
\begin{center}
\begin{picture}(425,140)
\put(0,0){\includegraphics[scale=0.55]{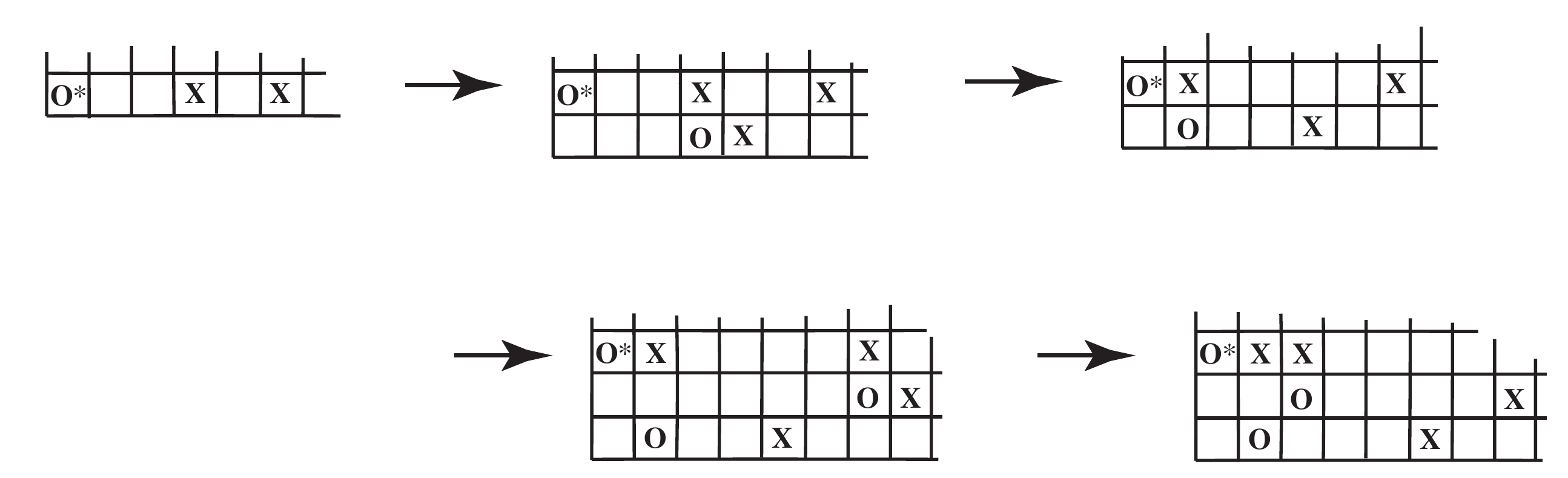}}
\put(105,110){\tiny {\bf STAB$'$}}
\put(245,110){\tiny {\bf $2\cdot$COMM$'$}}
\put(119,38){\tiny {\bf STAB$'$}}
\put(265,38){\tiny {\bf $4\cdot$COMM$'$}}
\end{picture}
\caption{This shows an example of the moves needed to move the $X$ is the row into L-formation.  }
\label{fig:cluster}
\end{center}
\end{figure}

\begin{lem} Every graph grid diagram $g$ representing a transverse spatial graph, $f$, is related to a preferred graph grid diagram for representing $f$ by a finite sequence of graph grid moves.  
\end{lem} \label{lemma_topref}

\begin{proof}
Recall that a preferred grid diagram is one where all of the nonstandard $O$'s have their flocks L-formation.  
Given a graph grid diagram, choose a nonstandard $O$ that is not in L-formation.  Use cyclic permutation to put the $O$ at the lower left corner of the grid.
We will explain an algorithm move this $O$ into L-formation.
If there are any $X$'s in the flock with this $O$ that are also in a flock of another $O$ that is in L-formation, then the other L-formation flock will need to be moved.  
If our nonstandard $O$ of interest is in a column with an $X$ that is in L-formation with another nonstandard $O$, we use the following procedure to move the flock out of the way.  
An example of this is shown in Figure \ref{fig:sharedX}.  
We do a row stabilization$'$ at said $X$, the new row is placed below the L-formation flock.  
Now the nonstandard $O$'s no longer share an $X$, but if there were any $X$'s to the right of the previously shared $X$, the flock was split by the stabilization$'$ move and so it is no longer in L-formation.  
To move the flock back into L-formation a stabilization$'$ move is done at each of the $X$'s to the right of the split (going from left to right), each time adding a row below the flock.
After the stabilization$'$ moves are done, 
the columns containing an $X$ in the flock can be moved by commutation$'$ moves to be next to the other $X$'s in the flock.  
This is done with all of the $X$'s so the flock is in L-formation again.  
If the $X$ were to share a row with our $O$ and a column with a different nonstandard $O$ that is in L-formation, a similar procedure is done with the roles of the rows and columns switched.  

Suppose there is no $X$ in the flock that is already in L-formation with a different $O$. 
Then, a row stabilization$'$ is done on the right most $X$ that is in the row with the $O$, adding a row to the bottom of the diagram and adding an $X$ and $O$ next to each other in the new column.  
Commutation$'$ can be used to move the new column next to the nonstandard $O$, or next to $X$'s that are next to the $O$ (see Figure \ref{fig:cluster}).  
This is repeated until all of the $X$'s in the row are adjacent to the $O$.  
A similar process is done with the $X$'s in the column of the $O$, bring the $O$ into L-formation.  
This process can be repeated until all of the $O$'s are in L-formation.  

This will increase the number of nonstandard $O$'s in L-formation, because no other flock is moved out of L-formation.    We continue this until all flocks are in L-formation. 
\end{proof}

For the following proof, we need a few more definitions.  
If an $O^\ast$ is associated with a vertex $v$, and is in L-formation, then all of the columns that contain an $X$ in the flock are called \textbf{v-columns}, similarly those rows containing the flock are called \textbf{v-rows}.  
We will give a name to certain sequences of the graph grid moves, they will be called \textbf{(column or row) vertex stabilization (and destabilization)}.  
A \textbf{column (or row) vertex stabilization} introduces a stabilization to the left of (or below) all of the $X$'s in the column (or in the row) with a nonstandard $O$, as shown in Figure \ref{vstab}.  
The row vertex stabilization is a combination of a number of stabilization$'$s and commutation$'$s.  
For a nonstandard $O,$ first a row stabilization$'$ is done, where the right most $X$ is placed into the lower new row by itself, the new column is placed between the set of $O$'s and $X$'s in the  upper row, and the $X$ in the lower row.  
Next, the second from the right $X$ is moved by commutation$'$ so that it is in the right most position.  
A stabilization$'$ move is done in the same way.  
Then commutation$'$ moves are done on the rows, moving the newest row directly below the flock, below the rows created in the stabilization$'$s that happened before.  
Finally, the $X$ in the flock is moved back to the original place in the flock via commutation$'$.  
Follow the same procedure for all of the $X$'s in the row. 

\begin{figure}[h]
\begin{center}
\begin{picture}(200, 210)
\put(0,0){\includegraphics[scale=0.5]{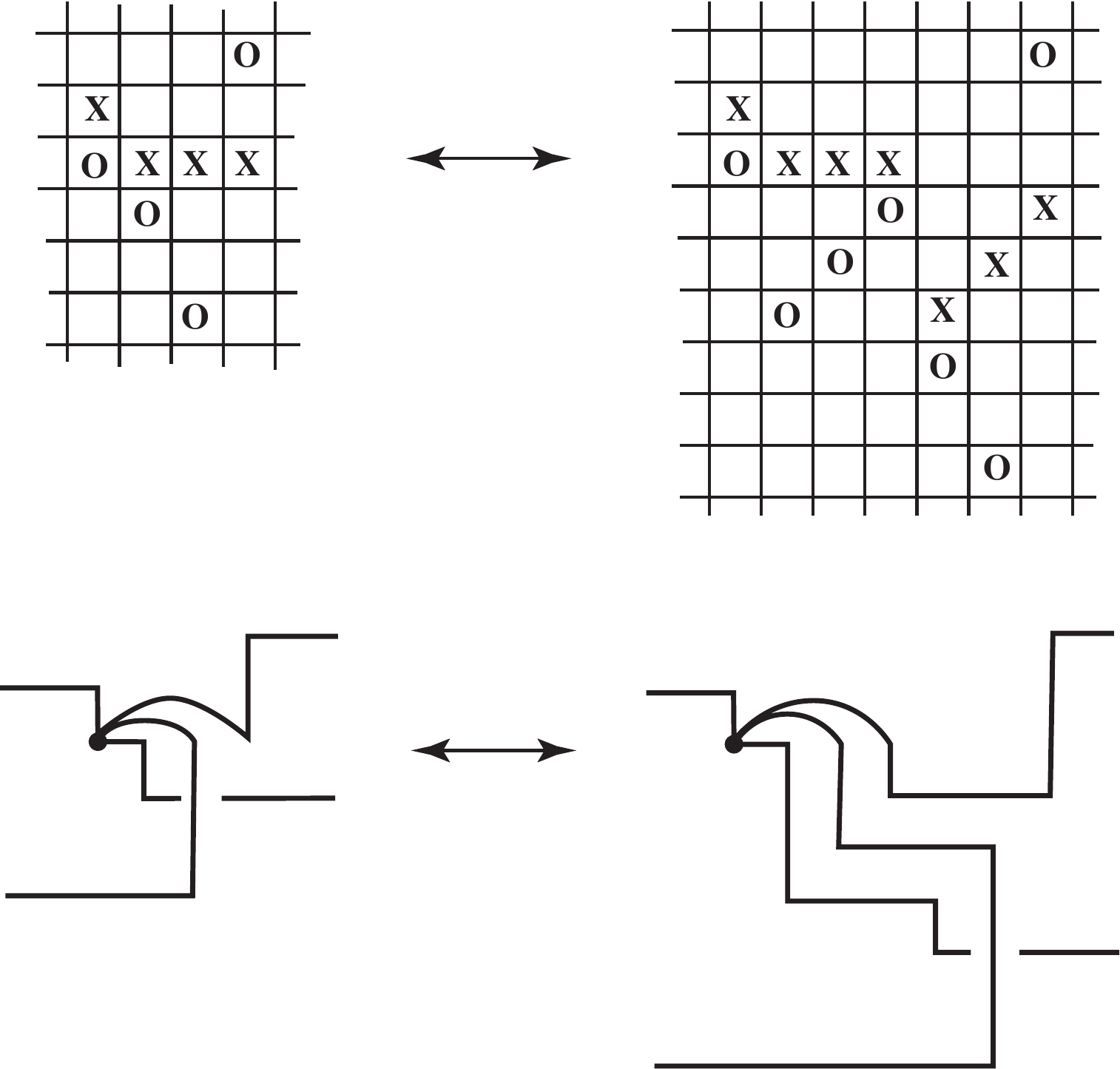}}
\put(19,176){\tiny{*}}
\put(145,176){\tiny{*}}
\end{picture}
\caption{This gives an example of a row vertex stabilization, showing what happens on a grid and the change in the associated graph.  }\label{vstab}
\end{center}
\end{figure}

\begin{thm} \label{thm:main_grid} If $g$ and $g'$ are two graph grid diagrams representing the same transverse spatial graph, then $g$ and $g'$ are related by a finite sequence of graph grid moves.  
\end{thm}

\begin{proof}  
First, using Lemma \ref{lemma_topref} we move both $g$ and $g'$ to preferred graph grid diagrams.  
We know that the diagrams of two isotopic transverse spatial graphs are related by a finite sequence of the graph Reidemeister moves, RI, RII, RIII, RIV, and R$\overline{\text{V}}$, shown in Figure \ref{Rmoves}, together with planar isotopy.  
So we need only show that preferred graph grid diagrams that result from embeddings that differ by a single Reidemeister move (or planar isotopy) can be related by a finite sequence of graph grid moves.

Due to the work of Cromwell \cite{Crom} and Dynikov \cite{Dy}, it is known that  any two grid diagrams of the same link are related by a finite sequence of grid moves, cyclic permutation, commutation, and stabilization/destabilization.  
The Reidemeister moves are local moves. In the grid diagram there is a set of columns and or rows that will be moved to accomplish any one of RI, RII or RIII.  
Because the first three Reidemeister moves do not involve vertices and we are working with preferred diagrams, the rows and columns that are moved will not contain an $O^\ast$.  
It could however contain rows or columns that that contain $X$'s that are in a flock with an $O^\ast$.  
In this case, first a vertex stabilization is done, so that the flock is not disrupted and the graph grid stays in preferred formation.  
Thus we need only show that any two preferred graph grid diagrams that come from the same embedding and differ as a results of a single Reidemeister move or planar isotopy which involves vertices can be related by a finite sequence of graph grid moves.  

\begin{figure}[h]
\begin{center}
\includegraphics[scale=0.4]{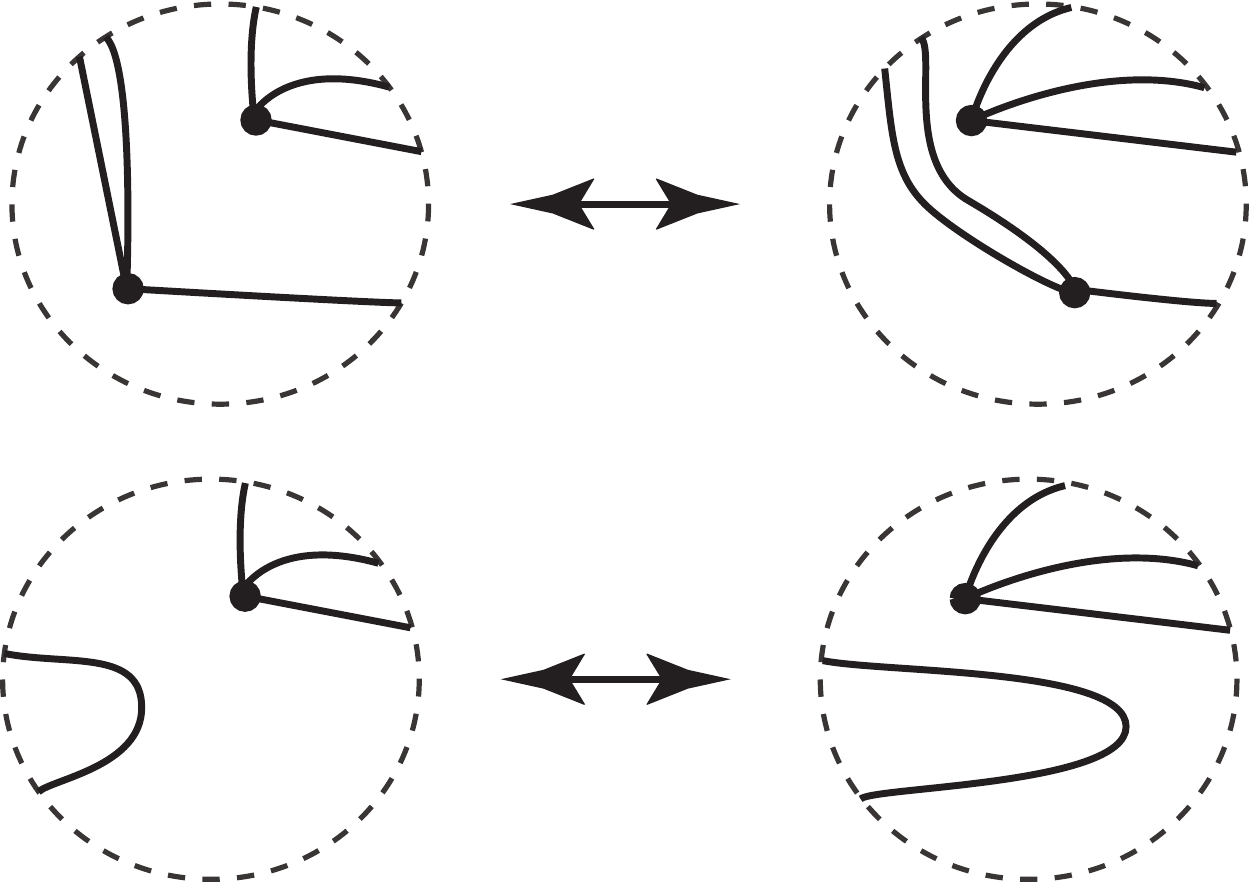}
\caption{This shows examples of the two planar isotopies that can occur which contain vertices.  The top is where two vertices are moved parallel to each other and the bottom one is where an arc and a vertex move parallel to each other.  }\label{PVVPAV}
\end{center}
\end{figure}

There are two moves and three planar isotopies with vertices to be considered: RIV, R$\overline{\text{V}}$, a planar isotopy where a valence two vertices is moved along the arc of the edges, a planar isotopy where two vertices are moved parallel to each other, and a planar isotopy where an arc and a vertex move parallel to each other (shown in Figure \ref{PVVPAV}).  
For the Figures of the graph grid diagrams in this proof we will only place an $*$ on an $O$ if it is not obvious from the grid that it is an $O^*$.  

\begin{figure}[h]
\begin{center}
\includegraphics[scale=0.5]{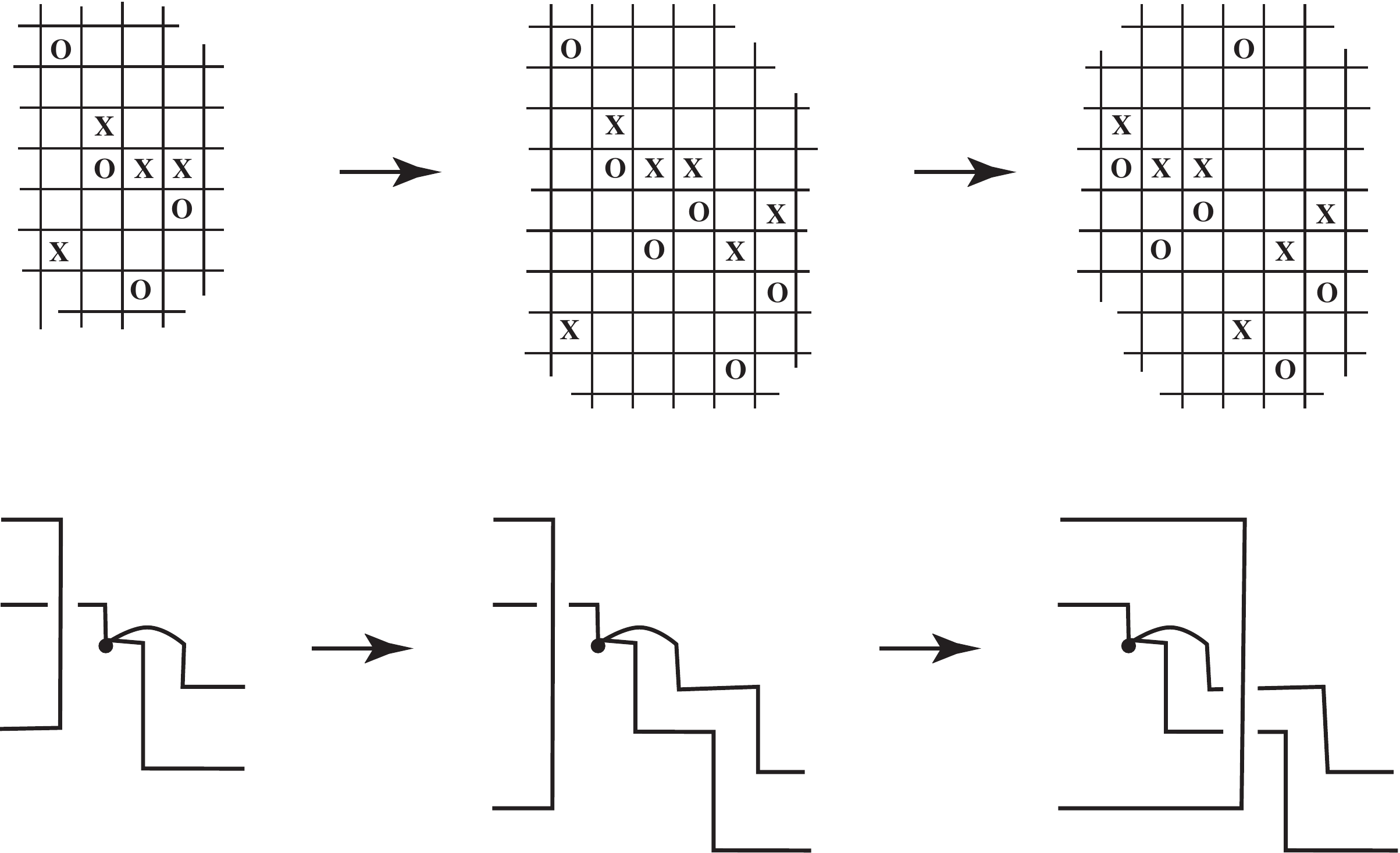}
\caption{This is an example where a vertex stabilization is first done then a commutation move to have a RIV move in the associated graph. }\label{RIVmoves}
\end{center}
\end{figure}

\textbf{RIV move:}
The RIV move is where an arc is moved from one side of a vertex to the other side by going either over or under the vertex.  
Up to planar isotopy we can assume that the edge is next to the vertex that it will pass over (or under).  
An example of RIV is shown in Figure \ref{RIVmoves}, here a row vertex stabilization is done followed by three column commutation$'$ moves.  
In general, RIV can be obtained via: first a vertex stabilization move, if needed, then a number of commutation$'$ moves between the $v$-columns (resp. $v$-rows) and other column (resp. row) that is associated with an appropriate arc.    

\textbf{R$\overline{\text{V}}$ move:}
The R$\overline{\text{V}}$ move corresponds to switching the order of the edges in the projection next to the vertex, which introduces a crossing between these edges.  See the left most move in Figure~\ref{fig_RV*InMany} for reference. 
Since we are working with transverse spatial graphs such a move can only occur between pairs of incoming edges and outgoing edges.  
We will look at the graph grid moves needed for an R$\overline{\text{V}}$ move between two outgoing edges.  The proof is similar for two incoming edges. 

In general, a commutation$'$ move between columns or rows that contain $X$'s in the same flock will result in a R$\overline{\text{V}}$ move between the two associated edges involved.   
In order to be able to iterate such moves, we present the follow processes.  In an R$\overline{\text{V}}$ move, two edges are switched next to a vertex.  
Let's call one of them the left edge and one the right edge.  
There are two possibilities with a R$\overline{\text{V}}$ move; either the right edge goes under the left or it goes over the left edge.    

In Figure \ref{RVmoves_a}, we show an example of R$\overline{\text{V}}$ where the right edge goes under the left between the two left most outgoing edges.  
In general, to have the right edge go under the left edge between two outgoing edges, first a row vertex stabilization move is done, followed by a commutation$'$ move between the columns containing the $X$'s associated with the edges involved.  

In Figure \ref{RVmoves_b}, we show an example of R$\overline{\text{V}}$ where the right edge goes over the left between the two left most outgoing edges.   
Let $X_1$ and $X_2$, from left to right, be the $X$'s in the flock that are associated with the edges that will be interchanged next to the vertex.  
In general to have the right edge go over the left edge, a row vertex stabilization move is done, if needed.
Next a row stabilization$'$ move is done on the row that contains the standard $O$ that is in the same column as $X_2$.  Call this $O$, $O_i$.    
The column that is added in the stabilization$'$ is placed immediately to the right of the flock.  
Then a commutation$'$ move is done to move the row containing $O_i$, below the row containing the standard $O$ that is in the same column as $X_1$.  
Finally a commutation$'$ move is done between the columns containing $X_1$ and $X_2$.
To do R$\overline{\text{V}}$ for the incoming edges, one needs only switch the role of the row and column.  

\begin{figure}[h]
\begin{center}
\includegraphics[scale=0.4]{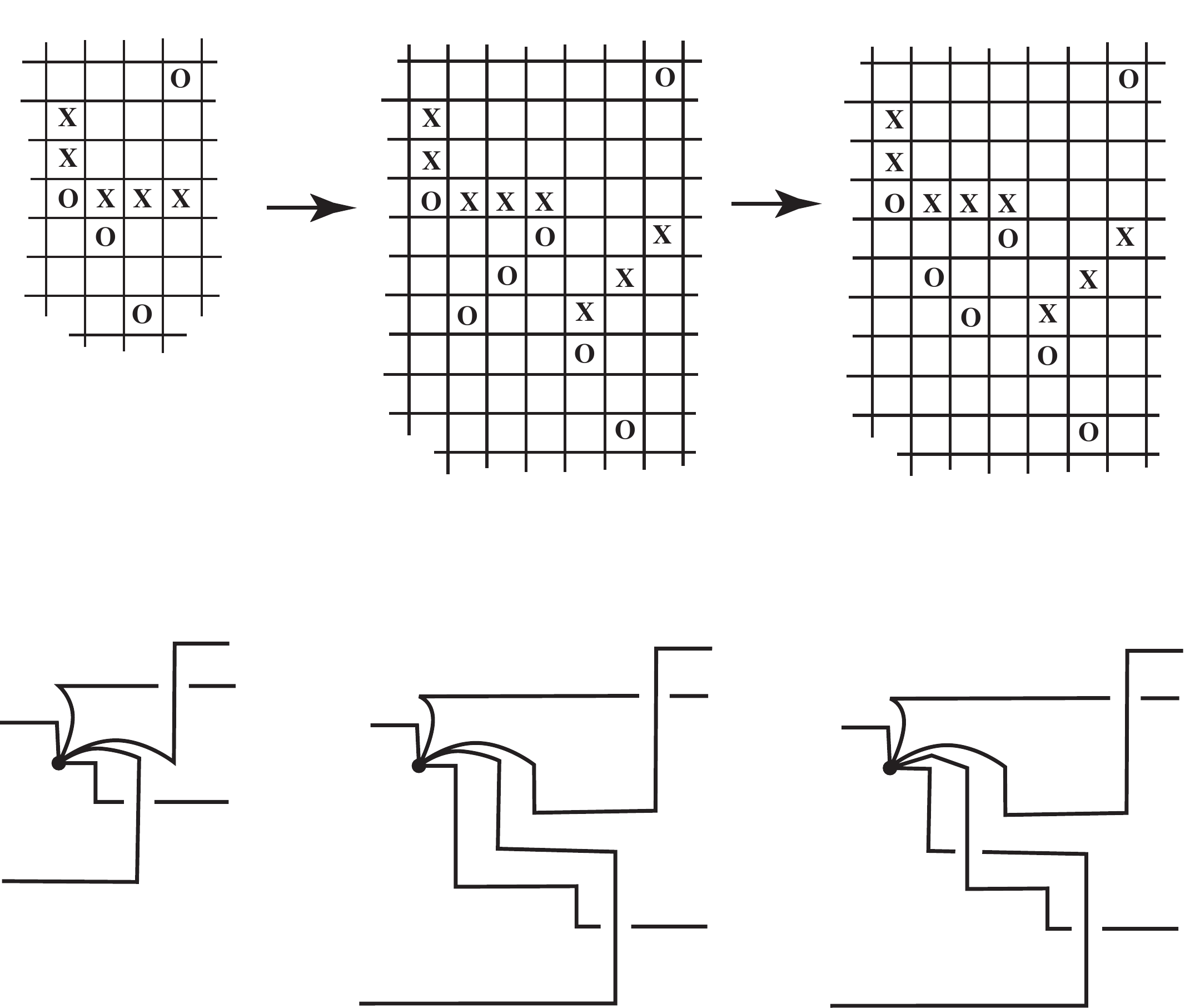}
\caption{This shows a row vertex stabilization move, followed by a commutation$'$ move in the grids, producing a R$\overline{\text{V}}$ move in the associated spatial graph.}\label{RVmoves_a}
\end{center}
\end{figure}

\begin{figure}[h]
\begin{center}
\includegraphics[scale=0.35]{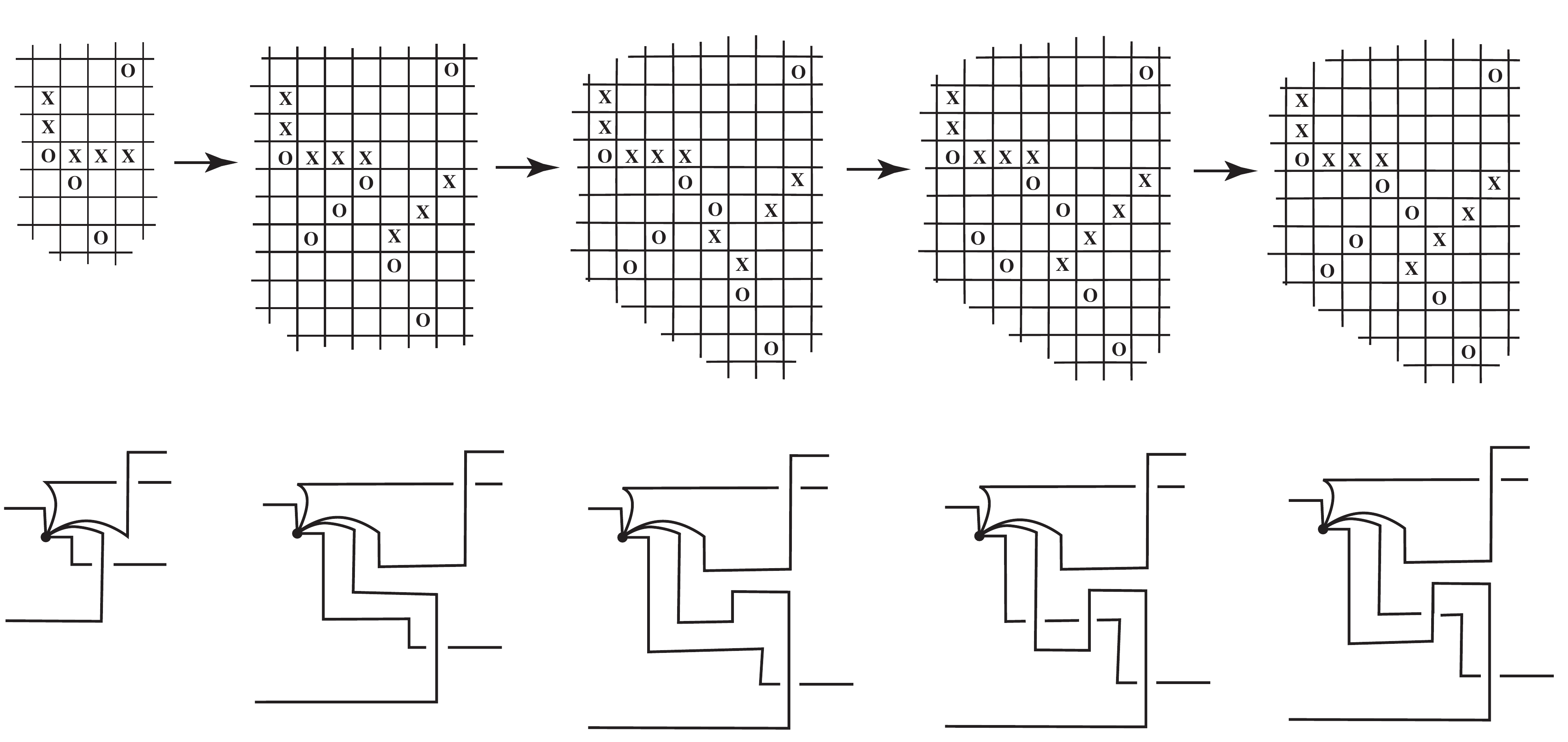}
\caption{This shows a row vertex stabilization move, followed by a row stabilization$'$move, a row commutation$'$ move and a column commutation$'$ move in the grids, producing a R$\overline{\text{V}}$ move in the associated spatial graph.  }\label{RVmoves_b}
\end{center}
\end{figure}

\textbf{Movement of a valence two vertex:}
The movement of a valence two vertex is equivalent to moving a $O^\ast$ with a single $X$ in both its row and column to the position of a standard $O$ that is on one of the incident edges.  
This could be thought of as choosing a different $O$ on the edge to be special, and was first addressed for grid diagrams in Lemma 2.12 of \cite{MOST}.  
A proof of the independence of which $O$ is special is given in Lemma 4.1 of \cite{Sa11}.  
Since this is a local change the same diagrammatic proof works in the graph case.  We outline the proof here.     

We will describe in words the moves needed to do this.  However, the reader may just choose to look at the moves done in Figure~\ref{move_vertex_fig}. To move a valance two vertex along an edge, we move the associated $O^\ast$ to the position of a standard $O$ on an adjacent edge.  
First a row stabilization$'$ is done at one of the neighboring $X$'s, between the $X$ and the $O^\ast$.  
The new column containing the new $X$ and $O$ are moved by commutation$'$ next to $O^\ast$, shown in the second image in Figure~\ref{move_vertex_fig}.
Then the row containing $O^\ast$ can be moved by commutation$'$ moves to the $X$ in the column with the $O^\ast$.  
Then the column containing $O^\ast$ can be moved by commutation$'$ to the $O$ in the column with the $X$ that is next to $O^\ast$.
Now $O^\ast$ is left and the $O$ and $X$ can be moved in their row by commutation$'$ to the $X$ that is in the $O$'s column.  
These $X$ and $O$ can then be removed by a column destabilization$'$.    

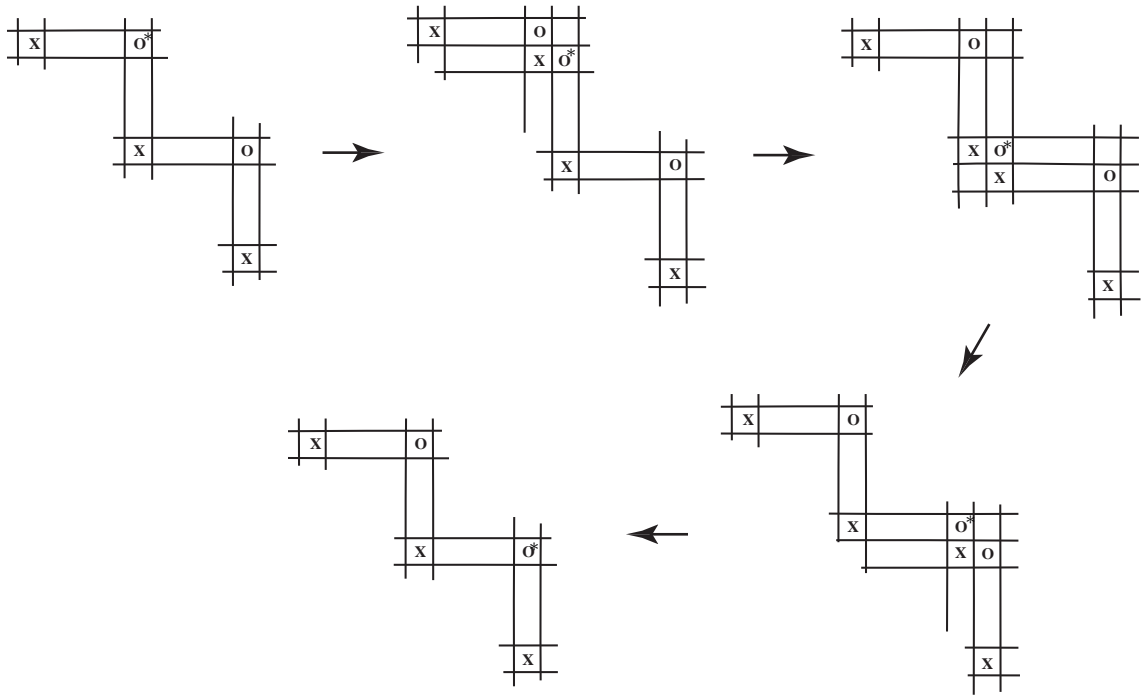
\begin{figure}[h]
\begin{center}
	\begin{picture}(456,264)
\put(0,0){\includegraphics[scale=0.5]{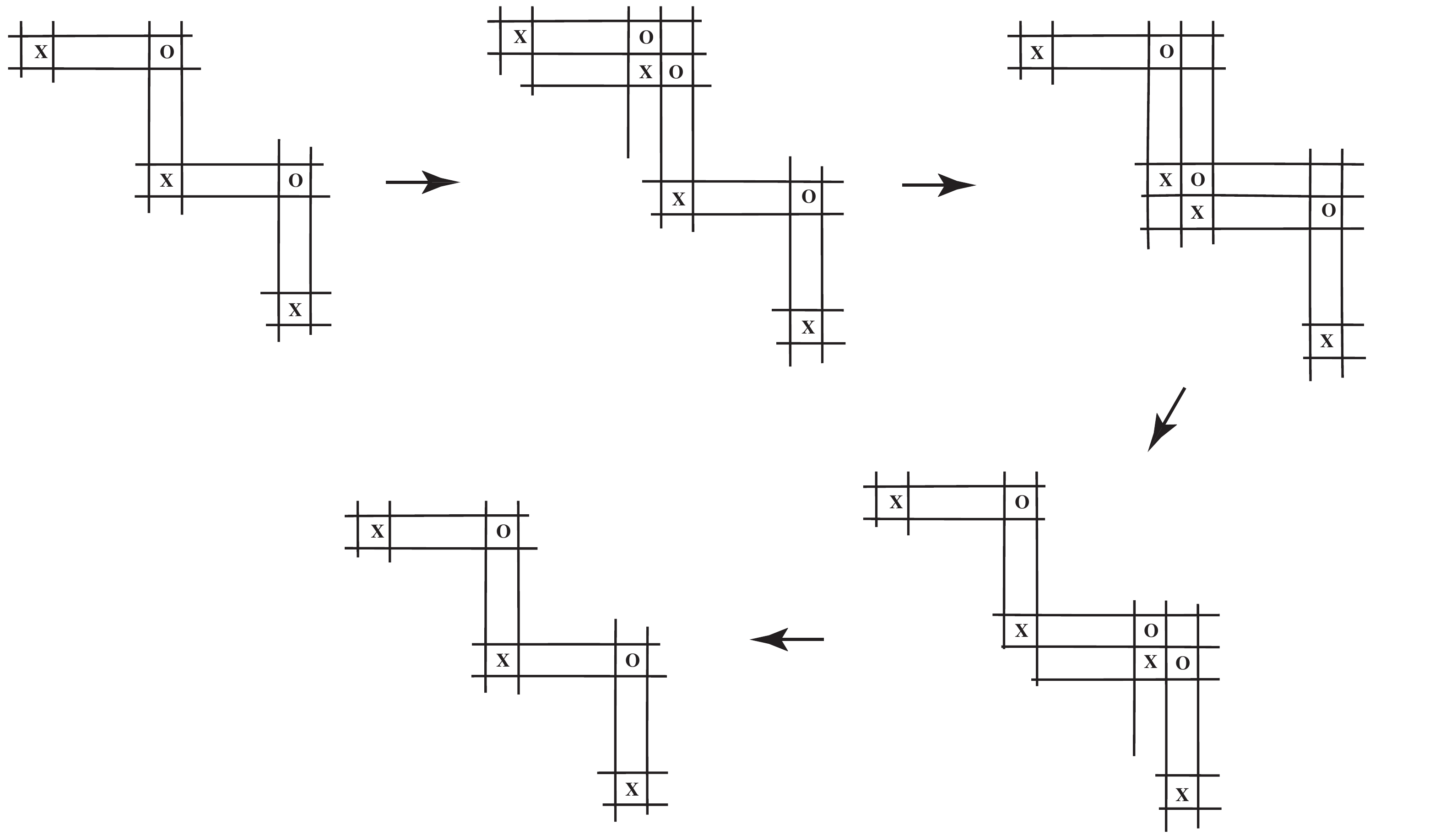}}
\put(54,250){\tiny{$\ast$}}
\put(214,244){\tiny{$\ast$}}
\put(379,209){\tiny{$\ast$}}
\put(365,67){\tiny{$\ast$}}
\put(200,57){\tiny{$\ast$}}
\end{picture}
\caption{Shows the graph grid moves needed to move a standard $O^\ast$ to the position of an $O$.  }\label{move_vertex_fig}
\end{center}
\end{figure}

\textbf{Two vertices pass each other:}
The planar isotopy where one vertex $v$ passes another vertex $w$ can be obtained via first vertex stabilization moves if needed, and then a number of commutation$'$ moves between the $v$-columns and the $w$-columns.  
In Figure \ref{VVmoves}, we show an example where only a single stabilization$'$ move is needed before the commutation$'$ moves, switching the order of the $v$-columns and the $w$-columns.  
To have the vertices move passed each other vertically rather than horizontally the roles of the rows and columns are interchanged.  

\begin{figure}[h]
\begin{center}
\includegraphics[scale=0.4]{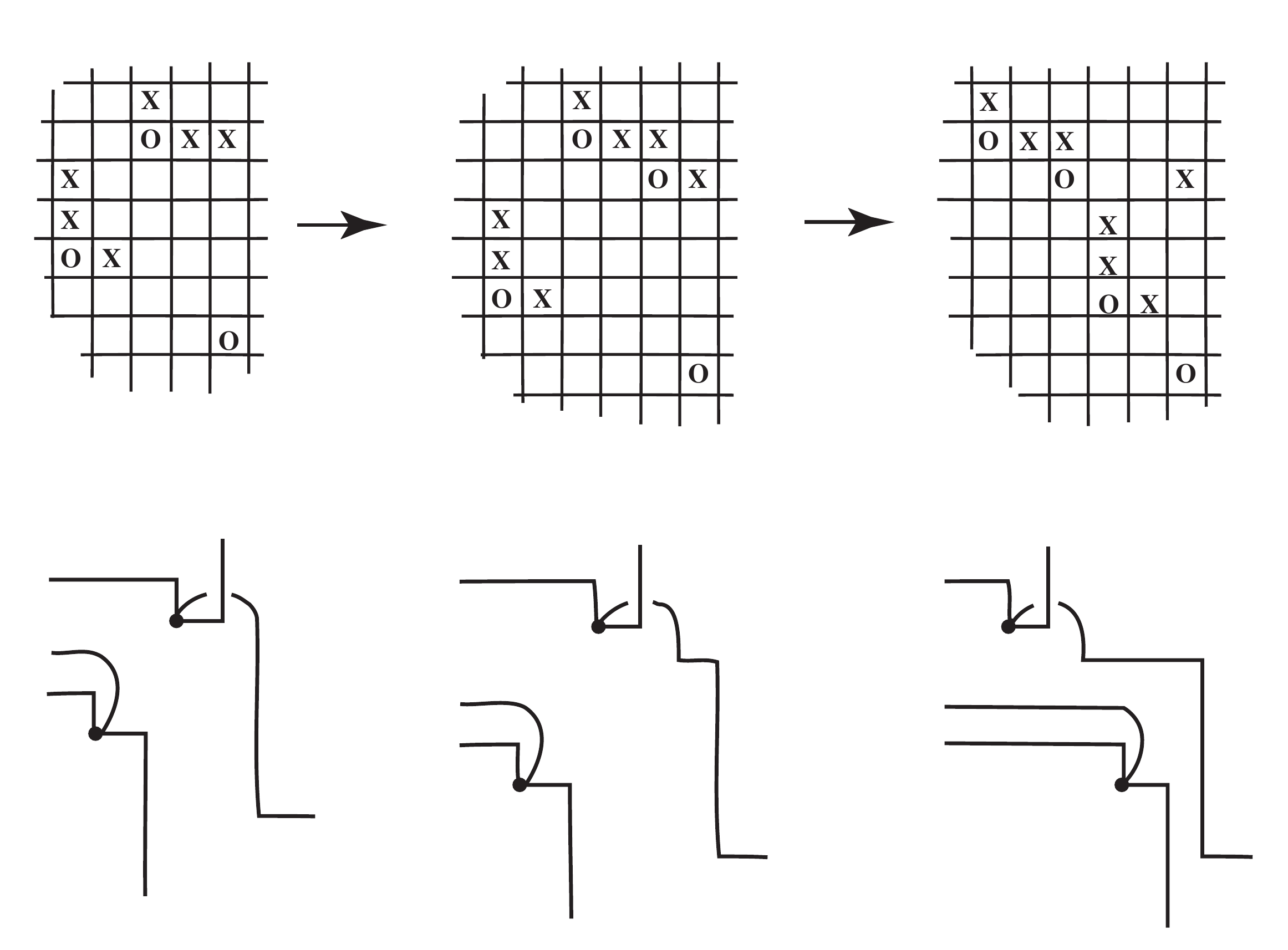}
\caption{This figure is an example illustrating the two steps to do a planar isotopy where one vertex passes another vertex in the diagram of the associated transverse spatial graph.  First a stabilization move was done then a number of the commutation$'$ moves were done.  }\label{VVmoves}
\end{center}
\end{figure}

\textbf{A vertex and arc pass each other:}
The planar isotopy where an arc and a vertex move passed each other can be obtained via first vertex stabilization moves if needed and then commutation$'$ moves between a set of $v$-columns (resp. $v$-rows) and another column (resp. row) that is associated with an appropriate arc (see Figure \ref{VEmoves}).   

\begin{figure}[h]
\begin{center}
\includegraphics[scale=0.4]{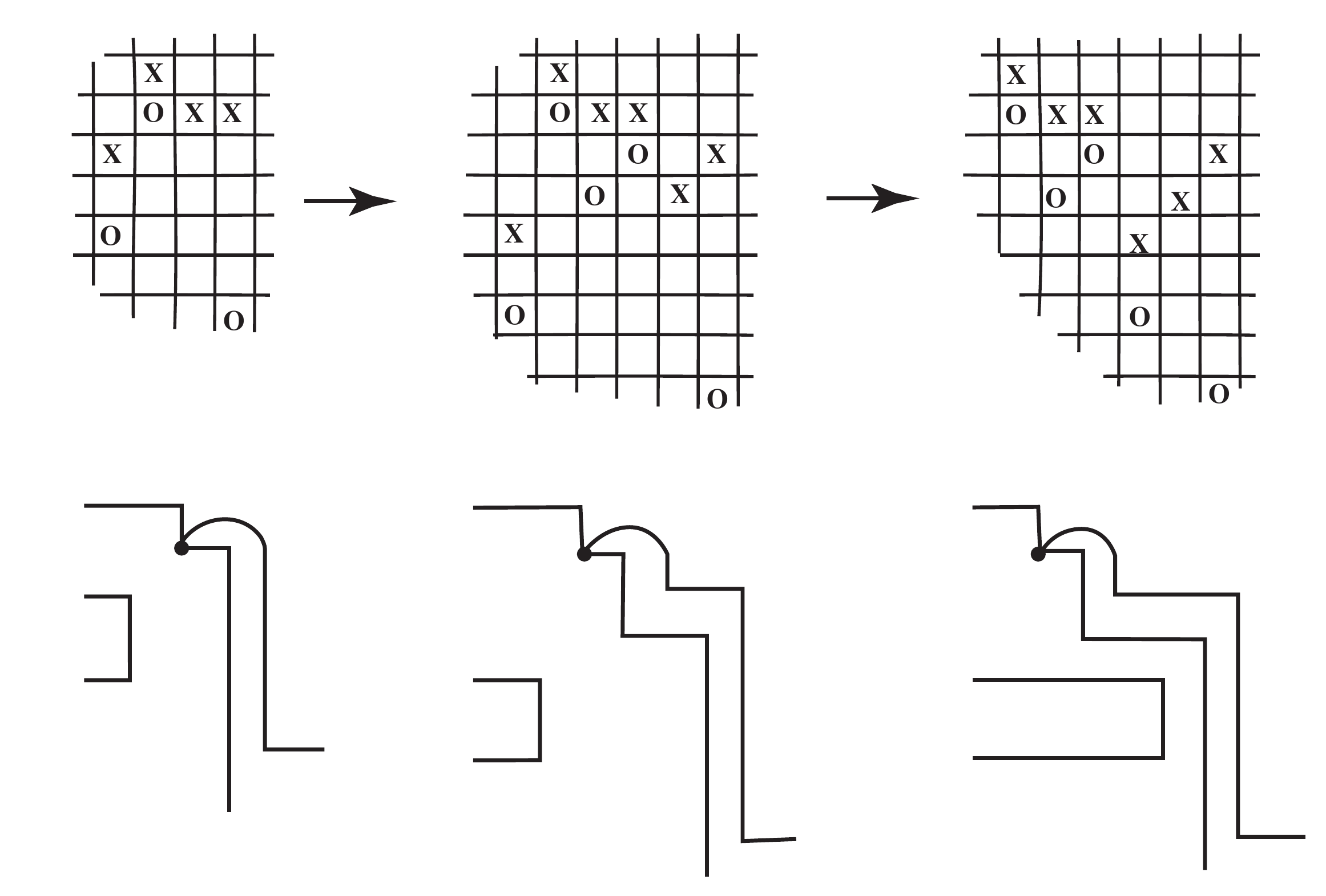}
\caption{This figure is an example illustrating the two steps to do a planar isotopy where a vertex passes an arc in the diagram of the associated transverse spatial graph.  First a row vertex stabilization move was done then a number of the commutation$'$ moves were done.  }\label{VEmoves}
\end{center}
\end{figure}

\end{proof}

\noindent This shows that even though there are numerous different graph grid diagrams that will represent the same transverse spatial graph, all such grids are related by a sequence of the graph grid moves.


\section{Graph Floer homology}\label{sec:def}

In this section, we will define the main invariant of this paper, which we call the graph Floer homology of a spatial graph.  This will take the form of the homology of a bigraded chain complex that is a module over a polynomial ring (or more generally, the quasi-isomorphism type of the chain complex).  One of the gradings is the homological grading (also called the Maslov grading) and the other grading is called the Alexander grading and will take values in the first homology of the exterior of the transverse spatial graph.  Our definitions will generalize those given in \cite{MOS} and \cite{MOST} except that we only get a (relatively) bigraded object instead of a filtered object.  In particular, when the spatial graph is a knot or link, we recover the associated graded objects from \cite{MOST} (but with a relative Alexander grading). In this section and throughout the rest of the paper, we assume that the reader is familiar with the material in Sections 1--3 of \cite{MOST}.   

\subsection{Algebraic Terminology}

We start with some algebraic preliminaries. The reader can skip this subsection upon first reading and refer back to it as needed.  Many of these definitions are similar to those in Section~2.1 of \cite{MOST}.

Let $C$ be a vector space over $\mathbb{F}$ where $\mathbb{F}$ is the field with $2$ elements and let $\G$, $\G_1$, and $\G_2$ be abelian groups.  Recall that a \textbf{$\G$ grading on $C$} (also called an \textbf{absolute $\G$ grading}) is a decomposition $C = \oplus_{g \in \G} C_g$ where $C_g \subseteq C$ is a vector subspace of $C$ for each $g$.  In this case we say that $C$ is \textbf{graded over $\G$} or is \textbf{graded}.  A linear map $\phi: C \rightarrow C^\prime$ between two graded vector spaces is  \textbf{a graded map of degree $h$} if $\phi(C_{g}) \subseteq C^\prime_{g+h}$ for all $g \in \G$.  A \textbf{relative $\G$ grading on $C$} is a $\G$ grading that is well defined up to a shift in $\G$.  That is, $C = \oplus_{g \in \G} C_g$ and $C = \oplus_{g \in \G} C'_g$ give the same relative $\G$ gradings if there exists an $a \in \G$ such that $C_g = C'_{g+a}$ for all $g\in \G$.  Thus if $C = \oplus_{g \in \G} C_g$ has a well-defined relative grading and $x \in C_{g_1}$ and $y \in C_{g_2}$ then the difference between their gradings, $g_1 - g_2$ is well-defined and independent of the choice of direct sum decomposition.   In this case, we say that $C$ is \textbf{relatively graded over $\G$} or is \textbf{relatively graded}.  A linear map $\phi: C \rightarrow C^\prime$ between two relatively graded vector spaces is  a \textbf{graded map} if there exists an $h\in \G$ such that $\phi(C_{g}) \subseteq C^\prime_{g+h}$ for all $g \in \G$. Note that it does not make sense to talk about the degree of this map since we can shift the subgroups and get a different value for $h$.  In this paper, we will be interested in relatively \emph{bigraded} vector spaces over $H_1(E(f))$ and $\mathbb{Z}$ where $E(f)$ is the complement of a transverse spatial graph in $S^3$. 

\begin{definition} A \textbf{$(\G_1,\G_2)$ bigrading on $C$} is an $\G_1 \oplus \G_2$ grading on $C$.  In this case we say that $C$ is \textbf{bigraded over $\G_1$ and $\G_2$} or is \textbf{graded}.   We may also refer to a bigrading as an \textbf{absolute bigrading} when convenient.  A linear map $\phi: C \rightarrow C^\prime$ between two bigraded vector spaces is a \textbf{bigraded map of degree} $(h_1,h_2)$ if $\phi(C_{(g_1,g_2)}) \subseteq C^\prime_{(g_1+h_1,g_2+h_2)}$ for all $(g_1,g_2) \in \G_1 \oplus \G_2$.  A \textbf{relative bigrading} on $C$ over $\G_1$ and $\G_2$ is a relative $\G_1 \oplus \G_2$ grading on $C$.   In this case we say that $C$ is \textbf{relatively bigraded over $\G_1$ and $\G_2$} or is \textbf{relatively graded}.  A linear map $\phi: C \rightarrow C^\prime$ between two relatively graded vector spaces is a \textbf{bigraded map} if there exists an $(h_1,h_2)\in \G_1 \oplus \G_2$ such that $\phi(C_{(g_1,g_2)}) \subseteq C^\prime_{(g_1+h_1,g_2+h_2)}$ for all $(g_1,g_2) \in \G_1 \oplus \G_2$.  
\end{definition}

Note that a (relative) bigrading of $C$ over $\G_1$ and $\G_2$ gives a well-defined (relative) grading over $\G_i$ for $i=1,2$ in the obvious way: $C= \oplus_{g_1 \in \G_1} (\oplus_{g_2 \in \G_2} C_{(g_1,g_2)})$ and $C= \oplus_{g_2 \in \G_2} (\oplus_{g_1 \in \G_1} C_{(g_1,g_2)})$.
Our main invariant will turn out to be graded over two groups, one of which will be relatively graded and one which will be (absolutely) graded. In the following definition, RA stands for relative-absolute or relatively-absolultely. 

\begin{definition}Let $\G_1$ and $\G_2$ be abelian groups and $C$ be a vector space. An \textbf{RA $(\G_1,\G_2)$ bigrading on $C$} is a $\G_1 \oplus \G_2$ grading that is well defined up to a shift in $\G_1 \oplus \{0\}$.  That is, $C = \oplus_{g \in \G} C_g$ and $C = \oplus_{g \in \G} C'_g$ give the same RA $(\G_1,\G_2)$ bigradings if there exists an $(g_1,0) \in \G$ such that $C_g = C'_{g+(g_1,0)}$ for all $g\in \G$.  In this case, we say that $C$ is \textbf{RA bigraded over $\G_1$ and $\G_2$} or is \textbf{RA bigraded}. A linear map $\phi: C \rightarrow C^\prime$ between two RA bigraded vector spaces is a \textbf{bigraded map of degree $(*,h_2)$} if there exists an $h_1 \in \G_1$ such that  $\phi(C_{(g_1,g_2)}) \subseteq C^\prime_{(g_1+h_1,g_2 + h_2)}$ for all $(g_1,g_2) \in \G_1 \oplus \G_2$.  A linear map $\phi: C \rightarrow C^\prime$ between two RA bigraded vector spaces is a \textbf{bigraded map} if it is a bigraded map of some degree. 
\end{definition}
Note that an RA $(\G_1,\G_2)$ bigrading of $C$ gives a well-defined relative grading over $\G_1$ and a well-defined (absolute) grading over $\G_2$.
We will also need to define bigraded chain complexes and their equivalences.  

\begin{definition} A \textbf{$(\G,\Z)$ bigraded chain complex} is a $(\G,\Z)$ bigraded vector space $C$ and bigraded map $\partial: C \rightarrow C$ of degree $(0,-1)$ such that $\partial^2=0$. For $g\in \G$ and $i\in \mathbb{Z}$, a linear map $\phi: C \rightarrow C'$ between $\G \oplus \Z$ bigraded chain complexes is a \textbf{bigraded chain map of degree $(g,i)$} if it is a chain map ($\partial \circ \phi = \phi \circ \partial$) and it a bigraded map of degree $(g,i)$.  We say that $\phi$ is a \textbf{bigraded chain map} if it is a bigraded chain map of some degree. 
\end{definition}

Note that if $C =  \oplus_{g\in \G} C_g$ has a relative $\G$ grading then it makes sense to talk about a \textbf{graded map $\phi:C \rightarrow C$ of degree $h$} as one that satisfies $\phi(C_g) \subset C_{g+h}$ for all $g\in \G$.  For, suppose $a \in \G$ and $C = \oplus_{g\in \G} C'_g$ such for all $g\in \G$, $C_g = C'_{g+a}$.   Then $\phi(C'_g) = \phi(C_{g-a}) \subset C_{g-a+h} = C'_{g+h}$ for all $g\in \G$.  Similarly, we can define a bigraded map of degree $(g_1,g_2)$ between a relatively (or RA) bigraded vector space and itself.   We say a linear map (between absolutely, relatively, or RA bigraded vector spaces) is a  \textbf{bigraded} map if it is a bigraded map of some degree.

\begin{definition} A \textbf{relative (respectively RA) $(\G,\Z)$ bigraded chain complex} is a relative (respectively RA) $(\G,\Z)$ bigraded vector space $C$ and bigraded map $\partial: C \rightarrow C$ of degree $(0,-1)$ such that $\partial^2=0$. A linear map $\phi: C \rightarrow C'$ between relative (respectively RA) $\G \oplus \Z$ bigraded chain complexes is a \textbf{bigraded chain map} if (it is a chain map) $\partial \circ \phi = \phi \circ \partial$ and it a bigraded map.  
\end{definition}

Similarly, we can define a \textbf{$(\G,\Z_2)$ bigraded, relatively bigraded, and RA bigraded chain complex}.  This will be used in the last section of the paper when we compare our invariant to sutured Floer homology. 

The primary invariant will be a chain complex that also takes the form of a module over a multivariable polynomial ring.  
Recall, a \textbf{$\G$ graded ring} is a commutative ring $R$
with a direct sum decomposition as abelian groups $R = \oplus_{g \in \G} R_g$ where $R_g$ is a subgroup of $R$ and $R_{g_1} R_{g_2} \subset R_{g_1+g_2}$ for all $g_i \in \G$.  If $R$ is a $\G$ graded ring, a \textbf{$\G$ graded (left) $R$-module} is a (left) $R$-module with a direct sum decomposition as abelian groups $M = \oplus_{g \in \G} M_g$ where $M_g$ is a subgroup of $M$ and $R_{g_1} M_{g_2} \subset M_{g_1+g_2}$ for all $g_i \in \G$.  A \textbf{graded $R$-module homomorphism (of degree $h$)} of is a graded map $\phi: C \rightarrow C^\prime$ (of degree $h$) between $\G$ graded $R$-modules that is also an $R$-module homomorphism.  We can similarly define \textbf{relatively graded, bigraded, relatively bigraded, and relatively-absolutely bigraded $R$-modules} and graded module homomorphisms in these cases.

\begin{definition} A \textbf{$(\G,\Z)$ bigraded (left) $R$-module chain complex} is a $(\G,\Z)$ bigraded (left) $R$-module $C$ and bigraded $R$-module homomorphism $\partial: C \rightarrow C$ of degree $(0,-1)$ such that $\partial^2=0$.  For $g\in \G$ and $i\in \mathbb{Z}$, an $R$-module homomorphism $\phi: C \rightarrow C'$ between $\G \oplus \Z$ bigraded $R$-module chain complexes is a \textbf{bigraded $R$-module chain map of degree $(g,i)$} if $\partial \circ \phi = \phi \circ \partial$ and it a bigraded map of degree $(g,i)$.  We say that $\phi: C \rightarrow C'$ is a  \textbf{bigraded $R$-module chain map} if it is a bigraded $R$-module map of some degree. We can similarly define a relative (respectively RA) $(\G,\mathbb{Z})$ bigraded $R$-module chain complex and a relative bigraded $R$-module chain map \end{definition}
	
	We remark that if $C$ is a $(\G,\Z)$ bigraded chain complex of (left) $R$-modules then it not necessarily the case that any of $C_{(g,m)}$, $\oplus_{h\in \G} C_{(g,m)}$ or $\oplus_{m \in \mathbb{Z}} C_{(g,m)}$ is an $R$-module. 
	
The primary invariant of this paper associates to each graph grid diagram, a bigraded $R$-module chain complex.  However, choosing different graph grid diagrams representing the same transverse spatial graph will lead to different chain complexes.  We will show that they are all quasi-isomorphic. 

\begin{definition} A chain map $\phi: C \rightarrow C'$ of chain complexes is a \textbf{quasi-isomorphism} if it induces an isomorphism on homology.  We say that two chain complexes $C$ and $D$ are \textbf{quasi-isomorphic} if there is a sequence of chain complexes $C_0, \dots, C_r$ and quasi-isomorphisms 
	
\begin{diagram}
	C_0 &&&& C_2 &&&&&& C_{r-2} &&&& C_r\\
	    & \rdTo^{\phi_1} && \ldTo^{\phi_2} && \rdTo^{\phi_3} &&\cdots && \ldTo^{\phi_{r-2}} && \rdTo^{\phi_{r-1}} && \ldTo^{\phi_r} & \\
	    &&   C_1 &&&& C_3 && C_{r-3} &&&& C_{r-1} &&&	
\end{diagram}
\noindent 	such that $C_0=C$, $C_r=D$. Suppose $C$ and $D$ are $(\G,\mathbb{Z})$ bigraded $R$-module chain complexes.  We say that $\phi: C \rightarrow C'$ is a bigraded $R$-module quasi-isomorphism if $\phi$ is a quasi-isomorphism and a bigraded $R$-module homomorphism.   
We say that $C$ and $D$ are \textbf{quasi-isomorphic (as $(\G,\mathbb{Z})$ bigraded $R$-module chain complexes)} if they are quasi-isomorphic and where all the quasi-isomorphisms are bigraded $R$-module chain maps.  
We can similarly define when two relative (respectively RA) $(\G,\mathbb{Z})$ bigraded $R$-module chain complexes are quasi-isomorphic.
\end{definition}

\subsection{The Chain Complex}\label{subsec:ch_complex}

For technical reasons, we need to restrict our definition to those graphs that are both sinkless and sourceless.  The graph grid diagrams representing these transverse spatial graphs have at least one $X$ per column and row.  We will need this condition to ensure that $\partial^2 =0.$

\begin{definition} A graph grid diagram is \textbf{saturated} if there is at least one $X$ in each row and each column.  A transverse spatial graph $f: G \rightarrow S^3$ is called \textbf{sinkless and sourceless} if its underlying graph $G$ is sinkless and sourceless (i.e. has no vertices with only incoming edges or only outgoing edges).
\end{definition}

\begin{remark}
(1) Suppose that $g$ is a graph grid diagram representing the transverse spatial graph $f$.  Then $g$ is saturated if and only if $f$ is sinkless and sourceless.   (2) If one performs a graph grid move on a saturated graph grid diagram, then the resulting graph grid diagram is saturated.  
\end{remark}

\noindent  \textbf{\emph{For the rest of this paper, unless otherwise mentioned, we will assume that all transverse spatial graphs are sinkless and sourceless and all graph grid diagrams are saturated}}. 

\vspace{10pt}
Let $f: G \rightarrow S^3$ be a sinkless and sourceless transverse spatial graph, $E(f) = S^3 \smallsetminus N(f(G))$ where $N(f(G))$ is a regular neighborhood of $f(G)$ in $S^3$, $g$ be an $n\times n$ saturated graph grid diagram representing $f$, and $\T$ its corresponding toroidal diagram.  Now, let 
$$\mathbf{S} = \{ \{x_i\}_{i=1}^n | x_i \in \alpha_i \cap \beta_{\sigma(i)}, \sigma \in S_n\},$$ where $S_n$ is the symmetric group on $n$ element, and define $C^-(g)$ to be the free (left) $R_n $-module generated by $\mathbf{S}$ where $$R_n = \mathbb{F}[U_1, \dots U_n] \text{ and }\mathbb{F}=\mathbb{Z}/2\mathbb{Z}$$ denotes the field with two elements.   When working with generators on a planar grid diagram, we use the convention that places the intersection point on the bottom gridline (not the top gridline) and the leftmost gridline (not the rightmost gridline).  When we want to specify the grid, we may write $\S(g)$ instead of $\S$.

Using the toroidal grid diagram, we can view the torus $\T$ as a two-dimensional CW complex with $n^2$ $0$-cells (intersections of $\alpha_i$ and $\beta_j$), 
$2n^2$ $1$-cells (consisting of line segments on $\alpha_i$ and $\beta_j$), and $n^2$ $2$-cells (squares cut out by $\alpha_i$ and $\beta_j$).   
Note that a generator $\x \in \mathbf{S}$ can be viewed as a $0$-chain.  
Let $U_{\alphas}$ be the $1$-dimensional subcomplex of $\T$ consisting of the union of the horizontal circles.  
We define paths, domains, and rectangles in the same way as \cite{MOST}.  
Given two generators $\x$ and $\y$ in $\mathbf{S}$, a \textbf{path} from $\x$ to $\y$ is a $1$-cycle $\gamma$ such that the boundary of the $1$-chain obtained by intersecting $\gamma$ with $U_{\alphas}$ is $\y - \x$.  A \textbf{domain} $D$ from $\x$ to $\y$ is an $2$-chain in $\T$ whose boundary $\partial D$ is a path from $\x$ to $\y$.  The support of $D$ is the union of the closures of the $2$-cells appearing in $D$ (with non-zero multiplicity).  We denote the set of domains from $\x$ to $\y$ by $\pi (\x, \y)$ and note that there is a composition of domains $\ast: \pi (\x, \y) \times \pi (\y, \z) \rightarrow \pi (\x, \z)$.  A domain from $\x$ to $\y$ that  is an embedded rectangle  $r$ is called a \textbf{rectangle} that connects $\x$ to $\y$.  Let $\mathrm{Rect}(\x,\y)$ be the set of rectangles that connect $\x$ to $\y$.  Notice if $\x$ and $\y$ agree in all but two intersection points then there are exactly two rectangles in $\mathrm{Rect}(\x,\y)$, otherwise $\mathrm{Rect}(\x,\y)=\emptyset$.  A rectangle $r\in\mathrm{Rect}(\x,\y)$ is \textbf{empty} if $Int(r)\cap\x=\emptyset$ where $Int(r)$ is the interior of the rectangle in $\T$.  Let $\mathrm{Rect}^o(\x,\y)$ be the set of empty rectangles that connect $\x$ to $\y$.  

We now make $C^-(g)$ into a chain complex $(C^-(g),\partial^-)$ in the usual way, by counting empty rectangles.  Note that in \cite{MOST,MOS}, the authors consider rectangles that contain both $X$'s and $O$'s.  However, because there is no natural filtration of $H_1(E(f))$, we must restrict to rectangles without any $X$'s and thus we get a graded object instead of a filtered object.  Let $\mathbb{X}$ and $\mathbb{O}$ be the set of $X$'s and $O$'s in the grid.  Put an ordering on each of these sets, $\mathbb{O} = \{O_i\}_{i=1}^n$ and $\mathbb{X}=\{X_i\}_{i=1}^m$ so that $O_1, \dots, O_V$ are associated to the $V$ vertices of the graph.  For a domain $D\in \pi(\x,\y)$, let $O_i(D)$ (respectively $X_i(D)$) denote the multiplicity with which $O_i$ (respectively $X_i$) appear in $D$.  More precisely, $D$ is a domain so $D=\sum a_j r_j$ where $r_j$ is a rectangle in $\mathcal{T}$.  Thus $O_i(D) = \sum a_j O_i(r_j)$ where $O_i(r_j)$ is $1$ if $O_i \in r_j$ and $0$ otherwise (similarly for $X_i(D)$).  We note that if $r$ is a rectangle then $O_i(r) \geq 0$.  
Define $\partial^-:C^-(g)\longrightarrow C^-(g)$ as follows.  For $\x \in \S$, let
$$\partial^-(\x)=\sum_{\y\in\s}\, \sum_{\substack{r\in\mathrm{Rect}^o(\x,\y)\\Int(r)\cap\mathbb{X}=\emptyset}}U_1^{O_1(r)}\cdots U_m^{O_m(r)}\cdot \y.$$  Extend $\partial^-$ to all of $C^-(g)$ so that it is an $R_n$-module homomorphism. When we want to specify the grid, we will write $\partial_g^-$ instead of $\partial^-$.

\begin{prop}  If $g$ is a saturated graph grid diagram then $\partial_g^-\circ\partial_g^-=0$.
\end{prop}\label{Prop_BoundaryMap}

\begin{proof}  This proof follows the proof of Proposition 2.10, page 2349 of \cite{MOST} almost verbatim.  The only change is that we only consider regions that do not contain any $X$'s.  Briefly, let $\x\in\s$, 
then $$\partial^-\circ\partial^-(\x)=\sum_{\mathbf{z}\in\s}\sum_{\substack{D\in\pi(\x,\z)}} N(D)\cdot U_1^{O_1(D)}\cdots U_n^{O_n(D)}\cdot\mathbf{z},$$ 
where $N(D)$ is the number of ways one can decompose $D$ as $D=r_1*r_2$ where $r_1\in\mathrm{Rect}^o(\x,\y)$, 
$r_2\in\mathrm{Rect}^o(\y,\mathbf{z})$ and $Int(r_i) \cap \mathbb{X} = \emptyset$.  When $\x\neq\mathbf{z}$ there are three general cases.  The rectangles are either disjoint, overlapping, or share a common edge.  In each of these cases there are two ways that the region can be decomposed as empty rectangles.  Thus the resulting $\mathbf{z}$ occurs in the sum twice.   
The case where $\x=\mathbf{z}$  is the result of a domains $D\in\pi(\x,\x),$ which are width one annuli.  Such domains do not occur in the image of $\partial^-\circ\partial^-(\x)$ because $\partial^-$ only counts rectangles that do not contain $X$'s and we have assumed that our graph grid diagram is saturated.  Thus we see that $\partial^-\circ\partial^-(\x)$ vanishes.  
\end{proof}

\subsection{Gradings} 
We put two gradings on the $(C^-(g),\partial^-)$.  The first is the homological grading, also called the Maslov grading.   This will be defined exactly the same as in \cite{MOST}.  We quickly review the definition for completeness. 

Given two finite sets of points $A$ and $B$ in the plane and $q=(q_1,q_2)$ a point in the plane, define $\mathcal{I}(q,B) =\# \{(b_1,b_2)\in B|b_1>q_1, b_2>q_2\}$.  That is, $\mathcal{I}(q,B)$ is the number of points in $B$ above and to the right of $q$.  Let 
  $\mathcal{I}(A,B)=\sum_{q\in A} \mathcal{I}(q,B)$ and $\mathcal{J}(A,B)=(\mathcal{I}(A,B)+\mathcal{I}(B,A))/2$. 
It will be useful to note that an equivalent definition of $\mathcal{I}(A,q)$ is the number of points in the set $\{a\in A|q_1>a_1, q_2>a_2\},$ i.e.~the number of points in $A$ to the left and below $q$.  So $\mathcal{J}(q,A)$ counts with weight one half all the points in $A$ up and to the right of $q$ and down and to the left of $q$.  Slightly abusing notation, we view $\mathbb{O}$ as a set of points in the grid with half-integer coordinates (the points where the $O$'s occur).  Similarly, we also view $\mathbb{X}$ as a set of points in the grid with half-integer coordinates. We extend $\mathcal{J}$ bilinearly over formal sums and differences of subsets in the plane and define the Maslov grading of $\x \in \S$ to be
$$M(\x)=\mathcal{J}(\x-\mathbb{O}, \x-\mathbb{O})+1.$$
This is consistent with the definition given in \cite{MOST} and the definition only depends on the set $\mathbb{O}$.  Since, like in \cite{MOST}, we also have exactly one $O$ per column and row, the Lemmas 2.4 and 2.5 of \cite{MOST} also hold for grid diagrams of transverse spatial graphs.  Thus, it follows that $M$ is a well-defined function on the toroidal grid diagram (Lemma 2.4 of \cite{MOST}).  In addition, for $x,y \in \S$, we can define the relative Maslov grading by  $M(\x,\y):=M(\x)-M(\y)$.  By Lemma 2.5 of \cite{MOST}, 
\begin{equation}M(\x,\y)=M(\x)-M(\y)=1-2n_{\mathbb{O}}(r),\label{maslov_eq}\end{equation}  for any \emph{empty}\footnote{If the rectangle $r$ contains $m$ points of $\x$ in its interior then $M(\x) = M(\y) + 1 + 2(m - n_{\mathbb{O}(r)})$} rectangle $r$ connecting $\x$  to $\y$, where 
where $n_{\mathbb{O}}(r)$ is the number of $O$'s in $r$.

\begin{figure}[h]
\begin{center}
\begin{picture}(99,35)
\put(0,0){\includegraphics{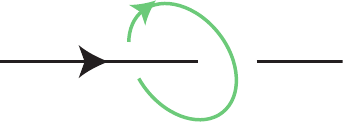}}
\put(50,40){$w(e)$}
\put(90, 22){$e$}
\end{picture}
\caption{This figure shows the weight $w(e)$ assigned to an edge $e \in E(G)$.  }\label{weight}
\end{center}
\end{figure}

Before defining the Alexander grading we first need to establish a weight system on the edges of $G$.  Let $E(G)$ be the set of edges of the graph $G$.  We define  
$w:E(G)\longrightarrow\ H_1(E(f))$ by sending each edge to the meridian of the edge, with orientation given by the right hand rule, seen as an element of $H_1(E(f))$ (see Figure~\ref{weight}). 
We say that $w(e)$ is the weight of the edge $e$.
We will denote the weight of $X$, $w(X)$ and the weight of $O$, $w(O)$, and define them based on the weight of the associated edge.  If $X$ or $O$ appear on the interior of an edge $e\in E(G)$ in the associated transverse spatial graph then $w(X):=w(e)$ and $w(O):=w(e)$.  If the $O$ is associated the vertex $v$ in the associated transverse spatial graph, then $$w(O):=\sum_{e\in In(v)} w(e)=\sum_{e\in Out(v)}w(e)$$
where $In(v)$ is the set of incoming edges of $v$ and $Out(v)$ is the set of outgoing edges of $v$.

Let the function $\epsilon:{\mathbb{O}\cup\mathbb{X}}\to \{1,-1\}$ defined by: 
\[\epsilon(p) = \left\{ 
\begin{array}{r l}
  1 & \quad \mbox{if $p\in\mathbb{X}$}\\
  -1 & \quad \mbox{if $p\in\mathbb{O}$.}\\ \end{array} \right. \]

For a point $q$ in the grid, define $$A^g(q) = \sum_{p\in\mathbb{O}\cup\mathbb{X}}\mathcal{J}(q,p)w(p)\epsilon(p).$$
We define the \textbf{Alexander grading of $\x$} with respect to the grid $g$ to be
$$A^g(\x)=\sum_{p\in\mathbb{O}\cup\mathbb{X}}\mathcal{J}(\x,p)w(p)\epsilon(p) = \sum_{x_i \in \x} A^g(x_i).$$  This value a priori lives in $\frac{1}{2}H_1(E(f))$ however, by Lemma~\ref{lem:h}, $A^g(\x) \in H_1(E(f))$. 
Note that this definition depends on the choice of \emph{planar} grid $g$ and is not a well-defined function on the toroidal grid.\footnote{One can slightly change this definition to make it well-defined on the toroidal grid diagram.  Our invariant will still only be relatively graded in the end however.}  However, the relative Alexander grading is a well-defined function on the toroidal grid diagram. 

\begin{definition}For $\x,\y \in S$, let $A^{rel}(\x,\y) :=A^g(\x)-A^g(\y)$ be the \textbf{relative Alexander grading} of $\x$ and $\y$.
\end{definition}

When it is clear, we may drop the ``rel'' or $g$ on $A$.  By the following Lemma, $A^{rel}$ does not depend on how you cut open the toroidal diagram $\T$ to give a planar grid diagram $g$.  We first need to define some notation. For a rectangle $r$, we define $\mathbf{n}_\mathbb{O}(r) = \sum_{q \in \mathbb{O}\cap r} w(q).$ Similarly, we define $\mathbf{n}_\mathbb{X}(r) = \sum_{q \in \mathbb{X}\cap r} w(q)$.  If $D$ is a domain then $D = \sum a_i D_i$ where $D_i$ is a rectangle.  We extend $\mathbf{n}_\mathbb{O}$ and $\mathbf{n}_\mathbb{X}$ linearly to domains, so that $\mathbf{n}_\mathbb{O}(D) = \sum a_i \mathbf{n}_\mathbb{O}(D_i)$ (similarly for $\mathbb{X}$).  

Note that a path from $\x$ to $\y$, on the toroidal grid, gives a $1$-cycle in $E(f)$.  To see this, recall that the transverse spatial graph associated to the graph grid diagram is constructed from 
vertical arcs going from an $X$ to an $O$ outside the torus (above the plane) and horizontal arc going from $O$ to $X$ inside the torus (below the plane).  Thus, the intersection of the transverse spatial graph and the torus is $\mathbb{X} \cup \mathbb{O}$.  Since a path is a $1$-cycle on the torus missing $\mathbb{X} \cup \mathbb{O}$, we get an element of $H_1(E(f))$.  Since the $\alpha_i$ and $\beta_i$ bound disks in $E(f)$, this is a well-defined element of $H_1(E(f))$, independent of the choice of path.

\begin{lem} \label{lem:alex_count}
Let $\mathbf{x,y}\in \mathbf{S}$.  If $D \in \pi(\x,\y)$ is a domain connecting $\x$ to $\y$ then 
\begin{equation}\label{eq:Ano} A^g(\x)-A^g(\y) = \mathbf{n}_\mathbb{X}(D)-\mathbf{n}_\mathbb{O}(D).\end{equation}
If $\gamma$ is a path connecting $\x$ to $\y$ then \begin{equation}\label{eq:Ahom}A^g(\x)-A^g(\y) = [\gamma]\end{equation} where $[\gamma] \in H_1(E(f))$ is the homology class of $\gamma$.  
 \end{lem}
We note that the domain (or rectangle) in this Lemma does not have to be empty. 

\begin{proof} Let $r$ be a rectangle connecting $\x$ to $\y$.  We will first show that Equation~(\ref{eq:Ano}) holds for $r$.  Consider 
 $$A^g(\x)-A^g(\y)= \sum_{q\in\mathbb{O}\cup\mathbb{X}}\mathcal{J}(\x,q)w(q)\epsilon(q) - \sum_{q\in\mathbb{O}\cup\mathbb{X}}\mathcal{J}(\y,q)w(q)\epsilon(q).$$

\begin{figure}[htpb!]
\begin{center}

\begin{picture}(100, 100)
\put(0,0){\includegraphics[scale=0.9]{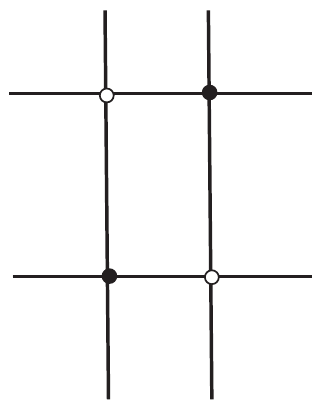}}
\put(40, 55){\small $r$} 
\put(56, 75){\tiny $x_1$}
\put(30, 27){\tiny $x_2$}
\put(30, 75){\tiny $y_1$}
\put(56, 27){\tiny $y_2$}

\end{picture}
\end{center}
\caption {\small Shows the rectangle $r$ with the intersections labeled.}
\label{square}
\end{figure}
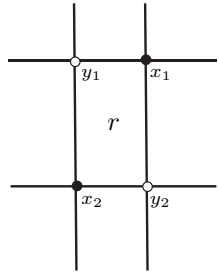

Let $x_1, x_2, y_1,$ and $y_2$ be the intersection points at the corners of $r,$ as shown in Figure \ref{square}.  Since the intersection points of $\x$ and $\y$ only differ at the corners of the rectangle $r$ this difference reduces to
$$\sum_{q\in\mathbb{O}\cup\mathbb{X}}[\mathcal{J}(x_1,q)+\mathcal{J}(x_2,q)-\mathcal{J}(y_1,q)-\mathcal{J}(y_2,q)]w(q)\epsilon(q).$$
By definition $\mathcal{J}(x_1,\mathbb{O}\cup\mathbb{X})$, counts with weight a half all those $X$'s and $O$'s up and to the right of $x_1$ and down and to the left of $x_1$.  In Figure \ref{diffsq} (a), we show which regions will have points counted in $\mathcal{J}(x_1, -)$ and $\mathcal{J}(x_2, -).$  The shading indicates whether if it will be counted with a weight of a half or one, this depends on whether it is counted in one or both of $\mathcal{J}(x_1, -)$ and $\mathcal{J}(x_2, -).$  Similarly,  Figure \ref{diffsq} (b), we show which regions will have points counted in $\mathcal{J}(y_1, -)$ and $\mathcal{J}(y_2, -).$ These counts differ by the points in $r$ counted with weight one.  

\begin{figure}[htpb!]
\begin{center}

\begin{picture}(175, 100)
\put(0,10){\includegraphics[scale=0.8]{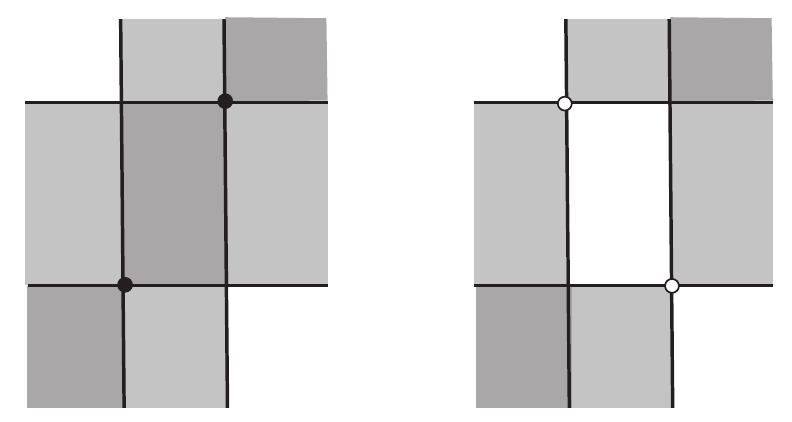}}
\put(35,0){(a)}
\put(140,0){(b)}
\put(40, 60){\small $r$} 
\put(53, 80){\tiny $x_1$}
\put(30, 37){\tiny $x_2$}
\put(143, 60){\small $r$} 
\put(133, 80){\tiny $y_1$}
\put(156, 37){\tiny $y_2$}

\end{picture}
\end{center}
\caption {\small (a) Shows with 50\% shading those regions that will be counted with weight one half in $\mathcal{J}(x_1, -)$ and $\mathcal{J}(x_2, -).$  (b) Shows with 50\% shading those regions that will be counted with weight one half in $\mathcal{J}(y_1, -)$ and $\mathcal{J}(y_2, -).$}
\label{diffsq}
\end{figure}
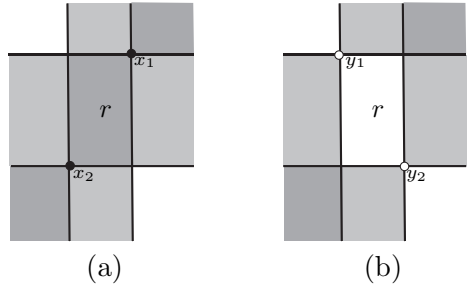
\noindent So we see that this difference is exactly $$\sum_{q\in r\cap[\mathbb{O}\cup\mathbb{X}]}w(q)\epsilon(q)=\sum_{q \in \mathbb{X}\cap r}w(q)-\sum_{q \in \mathbb{O}\cap r} w(q).$$  Thus, Equation~(\ref{eq:Ano}) holds for rectangles. 

Now we show that Equation~(\ref{eq:Ahom}) holds.  Let $\gamma$ a path connecting $\x$ to $\y$.  Using the fact that $S_n$ is generated by transposition, it follows that $\x,\y \in \S$, are related by a finite sequences of rectangles in $\mathcal{T}$.  That is, there is a sequence $\x = \x_1, \dots, \x_l =\y$ of points in $S$ and rectangles $r_j$ connecting $\x_j$ to $\x_{j+1}$ for $1 \leq j \leq l-1$.  Thus, 
$$ A^g(\x) - A^g(\y) = \sum_j (A^g (\x_j) - A^g(\x_{j+1})) = \sum_j (\mathbf{n}_\mathbb{X}(r_j)-\mathbf{n}_\mathbb{O}(r_j))  = \sum_j [\partial r_j] = [\sum_j \partial r_j].$$  Since $\sum_j \partial r_j$ is also a path connecting $\x$ to $\y$, $[\sum_j \partial r_j] =[\gamma]$.  Thus, we have proved, Equation~(\ref{eq:Ahom}).

Finally, we prove Equation~(\ref{eq:Ano}) for a general domain.  Suppose $D$ is a domain connecting $\x$ to $\y$.  Then $D = \sum a_i D_i$ for some rectangles $D_i$ and $\partial D$ is a path connecting $\x$ to $\y$.  Thus $A^g(\x) - A^g(\y) = [\partial D] = \sum_i a_i [\partial D_i] = \sum_i a_i (\mathbf{n}_\mathbb{X}(D_i)-\mathbf{n}_\mathbb{O}(D_i)) = \mathbf{n}_\mathbb{X}(D)-\mathbf{n}_\mathbb{O}(D)$.
\end{proof}

\begin{cor}\label{cor:torus_grading}The relative grading $A^{rel}: \S \times \S \rightarrow H_1(E(f))$ is a well-defined function on the toroidal graph grid diagram. 
\end{cor}

\begin{proof}Any two $\x,\y \in \S$ are related by a sequence of rectangles and hence there is always a path $\gamma$ connecting $\x$ to $\y$.  Since the homology class of the path is independent of the choice of path and $A^{rel}(\x,\y) =[\gamma]$, we see that $A^{rel}$ is independent of how you cut open the toroidal graph grid diagram to get a planar graph grid diagram. 
\end{proof}

We now provide an easy way to compute $A^g$ for a planar graph grid diagram $g$.  Let $\mathcal{L}$ be the lattice points in the grid, that is the $n^2$ intersections between the horizontal and vertical gridlines (i.e. the set of points that on the torus become $\alpha_i \cap \beta_j$).  Define $h:\mathcal{L}\rightarrow H_1(E(f))$, the \textbf{generalized winding number} of a point $q \in \mathcal{L}$ as follows.  Place the planar graph grid diagram $g$ on the Euclidean plane with the lower left corner at the origin (and the upper right corner is at the point $(n,n)$).  Now consider the following projection of the associated transverse spatial graph $pr(f)$.  Like in the last section, this is obtained by connecting 
the $X$'s to $O$'s by arcs in the columns and $O$'s to the $X$'s in the rows.  However, we now require that the arcs do not leave the $n\times n$ planar grid (they cannot go around the torus).  We also project this to the plane and ignore crossings.  For $q \in \mathcal{L}$, let $c$ be any path along the horizontal and vertical gridlines starting at the origin and ending at $q$.  We also require that $c$ meets $pr(g)$ transversely.  Then $c$ intersects $pr(g)$ in a finite number of points $b_1,\dots, b_{k}$ where $b_i$ lives on the interior  some edge of the spatial graph.  Suppose $b_i$ lies on edge $e$.  Define $\sigma(b_i)$ to be $\pm w(e)$ where the sign is given by the sign of the intersection of $e$ with $c(q)$ with the usual orientation of the plane (see Figure~\ref{fig:crossings}).  Using this, we
set 
$$h(q) = \sum_{i=1}^k \sigma(b_i).$$ 
When it is useful, we may also write $h^g$ to specify that we are computing $h$ in the graph grid diagram $g$.

\begin{figure}[htpb]
\begin{center}

\begin{picture}(122,60)
\put(0,10){\includegraphics{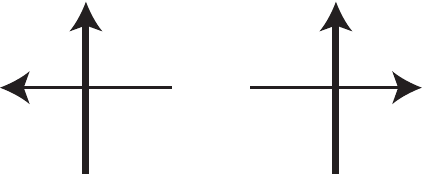}}
\put(12,0){$+w(e)$}
\put(85,0){$-w(e)$}
\put(33,52){$e$}
\put(87,52){$e$}
\put(68,24){$c(q)$}
\put(36,24){$c(q)$}
\end{picture}
\end{center}
\caption{$\sigma(b_i) = \pm w(e)$ where the sign is given by the sign of $e \cdot c(q)$ }
\label{fig:crossings}
\end{figure}
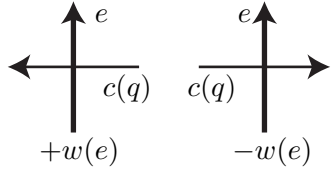

\begin{lem}The map $h$ is well-defined.
\end{lem}

\begin{proof}Consider the transverse spatial graph associated to $g$ whose projection is $pr(g)$ but which is pushed slightly above the plane.  Fix a base point at infinity  in $S^3$.  It is easy to see that $h(q)$ is the homology of a loop made up of the path from infinity to the origin, then a path in the plane, from the origin to $q$, and finally the path from $q$ to the point at infinity.  This is independent of the choice of path from the origin to $q$.  Therefore $h$ is well-defined.  
\end{proof}

\begin{lem}\label{lem:h} For any point $q$ on the lattice
\begin{equation}\label{eq:h} A^g(q) = -h(q).\end{equation}
\end{lem}

\begin{proof} Let $g$ be a graph grid diagram.  First we see that $h(q)=0$ for any point on the boundary of $g$ by definition.  
Notice that the sum $$\sum_{p\in col_i}w(p)\epsilon(p)=0,$$ for each $i$ where $col_i$ is the set of $O$'s and $X$'s in the $i$th column.  Similarly, $$\sum_{p\in row_i}w(p)\epsilon(p)=0,$$ for each $i$ where $row_i$ is the set of $O$'s and $X$'s in the $i$th row.  Given these observations, it is immediate that $A^g(q) = \sum_{p\in\mathbb{O}\cup\mathbb{X}}\mathcal{J}(q,p)w(p)\epsilon(p)$ is also zero for any point on the boundary of $g$.  

We will proceed by induction on the vertical grid line on which our lattice point occurs.  Suppose the equality in (\ref{eq:h}) holds for $q_i$ a lattice point on the $i$th vertical arc and consider $q_{i+1}$ the point immediately to the right of $q_i$ on the grid.  Let
$$RH:=\sum_{p\in\mathbb{O}\cup\mathbb{X}}\mathcal{J}(q_{i+1},p)w(p)\epsilon(p)-\sum_{p\in\mathbb{O}\cup\mathbb{X}}\mathcal{J}(q_i,p)w(p)\epsilon(p),$$
and 
$$LH:=-h(q_{i+1})-(-h(q_i)).$$

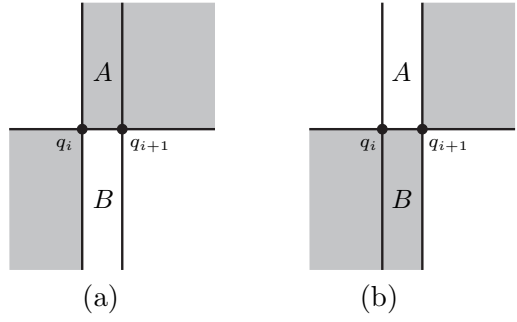
\begin{figure}[htpb!]
\begin{center}

\begin{picture}(175, 100)
\put(0,10){\includegraphics[scale=0.9]{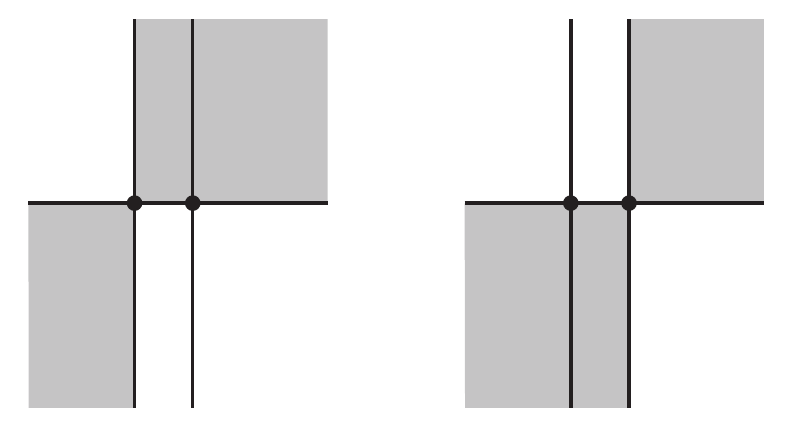}}
\put(35,0){(a)}
\put(140,0){(b)}
\put(25, 60){\tiny $q_i$} 
\put(53, 60){\tiny $q_{i+1}$}
\put(39, 85){\small $A$}
\put(39, 37){\small $B$}
\put(140, 60){\tiny $q_i$} 
\put(166, 60){\tiny $q_{i+1}$}
\put(152, 85){\small $A$}
\put(152, 37){\small $B$}

\end{picture}
\end{center}
\caption {\small (a) Shows those regions shaded that will be counted with weight one half in $\mathcal{J}(q_i, -)$.  (b) Shows those regions shaded that will be counted with weight one half in $\mathcal{J}(q_{i+1}, -)$.  }
\label{h_grid_step}
\end{figure}

Figure \ref{h_grid_step} shows the regions in which the points of $\mathbb{O}\cup\mathbb{X}$ will counted with weight a half in $\mathcal{J}(q_i,\mathbb{O}\cup\mathbb{X})$ and $\mathcal{J}(q_{i+1},\mathbb{O}\cup\mathbb{X}).$  These differ only in what is counted in the $i$th column.  So we see that
$$RH=\frac{1}{2}[\sum_{p\in B\cap(\mathbb{O}\cup\mathbb{X})}w(p)\epsilon(p)-\sum_{p\in A\cap(\mathbb{O}\cup\mathbb{X})}w(p)\epsilon(p)].$$
Using the fact  $\sum_{p\in col_i}w(p)\epsilon(p)=0$ again, we can simplify this to,
$$RH=-\sum_{p\in A\cap(\mathbb{O}\cup\mathbb{X})}w(p)\epsilon(p) = \sum_{p\in B\cap(\mathbb{O}\cup\mathbb{X})}w(p)\epsilon(p).$$
We now consider $LH=-h(q_i)-[-h(q_{i+1})] = h(q_{i+1}) - h(q_i)$.
Suppose the $O$ in the $i$th column is in $A$.  Then all the vertical arcs of $pr(g)$ in the $i$th column that intersect the arc from $q_i$ to $q_{i+1}$ are oriented upwards.  Thus $LH = \sum_{p \in \mathbb{X} \cap B}w(p) = \sum_{p\in B\cap(\mathbb{O}\cup\mathbb{X})}w(p)\epsilon(p)$.  If the $O$ in the $i$th column is in $B$, the all the vertical arcs of $pr(g)$ in the $i$th column that intersect the arc from $q_i$ to $q_{i+1}$ are oriented downwards.  So $LH = \sum_{p \in \mathbb{X} \cap A}-w(p) = -\sum_{p\in A\cap(\mathbb{O}\cup\mathbb{X})}w(p)\epsilon(p)$.
\end{proof}

\begin{cor}For all $\x \in \S(g)$, $$A^g(\x) = - \sum_{x_i \in \x} h(x_i) \in H_1(E(f)). $$
\end{cor}

For a given saturated graph grid diagram $g$, the functions $A^g$ and $M$ make $C^-(g)$ into a well-defined $(H_1(E(f)), \mathbb{Z})$ bigraded $R_n$-module chain complex once we say how the grading changes when we multiply a generator by $U_i$.  We set \begin{equation}A^g(U_i) = -w(O_i), M(U_i)=-2\label{eq:ring_grading}\end{equation}
 and define
$$A^g(U_1^{a_1} \cdots U_n^{a_n} \x) = A^g(\x) + \sum_{i=1}^n a_i A^g(U_i)$$
and 
$$M(U_1^{a_1} \cdots U_n^{a_n} \x) = M(\x) + \sum_{i=1}^n a_i M(U_i).$$
\noindent We remark that since $g$ is saturated, $A^g(U_i) \neq 0$.  

For each $a \in H_1(E(f))$ and $m \in \mathbb{Z}$, let $C^-(g)_{(a,m)}$ be the subspace (of the vector space) of $C^-(g)$ with basis $\{U_1^{a_1} \cdots U_n^{a_n} \x\, | \, A^g(U_1^{a_1} \cdots U_n^{a_n} \x) = a, \, M(U_1^{a_1} \cdots U_n^{a_n} \x)=m \}$.  This gives a bigrading on $C^-(g)=\sum_{(a,m)} C^-(g)_{(a,m)}$.


\begin{prop}  The differential $\partial^-$ drops the Maslov grading by one and respects the Alexander grading.  That is, $\partial^-: C^-(g)_{(a,m)} \rightarrow C^-(g)_{(a,m-1)}$ for all $a \in H_1(E(f))$ and $m\in \mathbb{Z}$.
\end{prop}

\begin{proof} Consider $U_1^{O_1(r)}\cdots U_m^{O_m(r)}\cdot \y$ appearing in the boundary of $\x$.  Since $\x$ and $\y$ are connected by an empty rectangle $r$, we see that $M(\x)=M(\y)+1-2n_{\mathbb{O}(r)}$ by Equation~(\ref{maslov_eq}).  Since each $U_i$ drops the Maslov grading by two, we have $M(\x)=M(U_1^{O_1(r)}\cdots U_m^{O_m(r)}\cdot \y)+1.$  Thus the differential $\partial^-$ drops the Maslov grading by one.

Next $A^g(\x) = A^g(\y) + \sum_{\mathbb{X}\cap r}w(X)-\sum_{\mathbb{O}\cap r} w(O),$ but by definition of the differential $\mathbb{X}\cap r=\emptyset.$  So $A^g(\x) = A^g(\y) -\sum_{\mathbb{O}\cap r} w(O)=A^g(U_1^{O_1(r)}\cdots U_m^{O_m(r)}\cdot \y).$  Thus the differential $\partial^-$ respects the Alexander grading.
\end{proof}

We are interested in viewing $(C^-(g),\partial^-)$ as a module instead of just a vector space.   Using the definition in (\ref{eq:ring_grading}), $\mathbb{F}[U_1, \dots, U_n]$ becomes an $(H_1(E(f)),\mathbb{Z})$ bigraded ring (i.e. $H_1(E(f)) \oplus \mathbb{Z}$ graded) and with this grading, $(C^-(g),\partial^-)$ is a $(H_1(E(f)),\mathbb{Z})$ bigraded $R_n$-module chain complex.   We would like to define an invariant of the graph grid diagram that is unchanged under any graph grid moves, giving an invariant of the transverse spatial graph.  Since $\mathbb{F}[U_1, \dots, U_n]$ depends on the size of the grid, we need to view $C^-(g)$ as a module over a smaller ring.  One choice would be to view $C^-(g)$ as a module over $\mathbb{F}[U_1, \dots, U_{V+E}]$, where $U_1, \dots, U_V$ correspond to the vertices of the graph and $U_{V+1}, \dots U_{V+E}$ each correspond to a choice of $O$ on a distinct edge of the graph.  However, by Proposition~\ref{prop:no_edges} below, multiplication by $U_i$ which is corresponds to an edge is chain homotopic to $0$ or multiplication by $U_j$ where $U_j$ corresponds to a vertex ($1 \leq j \leq V$).  Thus it makes sense to view $C^-(g)$ as a module over $\mathbb{F}[U_1, \dots, U_V]$ (recall that we ordered $\mathbb{O}$ such that $O_1,\dots,O_V$ are vertex $Os$).

Let $I :\mathbb{F}[U_1, \dots, U_V] \rightarrow \mathbb{F}[U_1, \dots, U_n]$ be the natural inclusion of rings defined by setting  $I(U_i)=U_i$ for $1\leq i \leq V$.  Using $I$, any module over $\mathbb{F}[U_1, \dots, U_n]$ naturally becomes an $\mathbb{F}[U_1, \dots, U_V]$-module.  Thus, we will view $C^-(g)$ as a $R_V$-module where $R_V=\mathbb{F}[U_1, \dots, U_V]$.  Note that $\partial^-$ preserves the homology and hence the homology of $(C^-(g),\partial^-)$ inherits the structure of a $(H_1(E(f)), \Z)$ bigraded $R_V$-module.

\begin{definition} Let $g$ be a saturated graph grid diagram representing the transverse spatial graph $f:G \rightarrow S^3$.  The \textbf{graph Floer chain complex} of $g$ is the RA $(H_1(E(f)), \Z)$ bigraded $R_V$-module chain complex $(C^-(g),\partial^-)$.  The \textbf{graph Floer homology of $g$}, denoted $HFG^-(g)$, is the homology of $(C^-(g),\partial^-)$ viewed as a RA $(H_1(E(f)), \Z)$ bigraded $R_V$-module.
\end{definition}

Before stating Proposition~\ref{prop:no_edges}, we need some terminology. 

\begin{definition} We say an edge of a graph is \textbf{singular} if at each of its endpoints it is the only outgoing edge or the only incoming edge. See Figure~\ref{fig:sing_edge} for an example. \end{definition}
	
We note that if a component of the graph is a simple closed curve, then every edge in that component is singular. 
	
\begin{center}
\begin{figure}[h]
	\begin{picture}(135,65)
	\put(0,0){\includegraphics{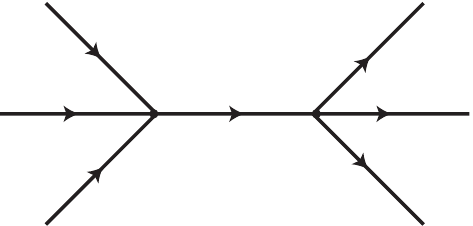}}
	\put(65,40){$e$}
	\end{picture}
	\label{fig:sing_edge}
	\caption{A singular edge $e$.}
\end{figure}
\end{center}

\begin{prop}\label{prop:no_edges}(1) If $O_i$ and $O_j$ are on the interior of the same edge then multiplication by $U_i$ is chain homotopic to multiplication by $U_j$.   (2) If $O_i$ is associated to a vertex with a single outgoing or incoming edge $e$ and $O_j$ is on the interior of $e$ then multiplication by $U_i$ is chain homotopic to multiplication by $U_j$. (3) If $O_i$ is on the interior of an edge that is not singular then multiplication by $U_i$ is null homotopic.  Moreover, each of the chain homotopies is bigraded $R_n$-module homomorphism of degree $(-w(O_i),-1)$.
\end{prop}

\noindent Note, this implies that the chain homotopies are also bigraded $R_V$-module homomorphisms.  

\begin{proof} Let $X_k \in \mathbb{X}$.  We define $H_k:C^-(g) \rightarrow C^-(g)$ by counting rectangles that contain $X_k$ but do not contain any other $X_s$. Specifically, 
$$H_k(\x):=\sum_{\y\in\s}\,\sum_{\substack{r\in \mathrm{Rect}^o(\x,\y)\\X_k\in r \\ X_s \notin r \, \forall \, X_s \in \mathbb{X}\smallsetminus \{X_k\}}} U_1^{O_1(r)}\cdots U_n^{O_n(r)}\cdot\y.$$  Note that $H_k:C^-(g)_{(a,m)} \rightarrow C^-(g)_{(a-w(O_i), m-1)}$. Suppose that $X_k$ shares a row with $O_i$ and shares a column with $O_j$.  There are three cases we need to consider.  (i) If $X_k$ is the only element of $\mathbb{X}$ in its row and column then, like in the proof of Lemma~2.8 of \cite{MOST}, we have that 
$$ \partial^- \circ H_k + H_k \circ \partial^- = U_i + U_j.$$ (ii) If $X_k$ is the only element of $\mathbb{X}$ in its row but it shares its column with other elements of $\mathbb{X}$, then $$ \partial^- \circ H_k + H_k \circ \partial^- = U_i.$$  The difference here is that the vertical annulus containing $X_k$ also includes another $X_s$ for $s \neq k$.  Thus it does not contribute to $\partial^- \circ H_k + H_k \circ \partial^-$.  (iii) Similarly, if $X_k$ is the only element of $\mathbb{X}$ in its column but it shares its row with other elements of $\mathbb{X}$, then $$ \partial^- \circ H_k + H_k \circ \partial^- = U_j.$$  Note that we need not consider the case when $X_k$ shares both its row and column with other elements of $\mathbb{X}$ since in this case $H_k$ is the zero map.  We use the fact that $w(O_i)=w(O_j)$ in parts (1) and (2) and the fact that you can add two chain homotopies to get another chain homotopy to complete the proof.  
\end{proof}

In Section~\ref{sec:invariance}, we will show that $(C^-(g),\partial^-)$ viewed as a relatively-absolutely $(H_1(E(f)),\Z)$ bigraded $R_V$-module chain complex changes by a quasi-isomorphism under graph grid moves.  Thus, its homology is an invariant of the spatial graph and not just the grid representative.

\begin{theorem}\label{thm:invariance} If $g_1$ and $g_2$ are saturated graph grid diagrams representing the same transverse spatial graph $f:G \rightarrow S^3$ then $(C^-(g_1),\partial^-)$ is quasi-isomorphic to $(C^-(g_2),\partial^-)$ as RA $(H_1(E(f)),\Z)$ bigraded $R_V$-modules.  In particular, 	$HFG^-(g_1)$ is isomorphic to $HFG^-(g_2)$ as RA $(H_1(E(f)),\Z)$ bigraded $R_V$-modules. 
\end{theorem}

\begin{proof}Suppose $g_1$ and $g_2$ are saturated graph grid diagrams that are related by a cyclic permutation move.  Then they have the same toroidal grid $\mathcal{T}$.  Note that there is an natural identification of $\S$ for both graph grid diagrams so that $C(g_1)=C(g_2)$ as abelian groups.  Let $\x,\y \in \S$.  By Corollary~\ref{cor:torus_grading}, $A^{g_1}(\x) - A^{g_1}(\y) = A^{g_2}(\x) - A^{g_2}(\y)$.  Hence $A^{g_1}(\x) - A^{g_2}(\x) = A^{g_1}(\y) - A^{g_2}(\y)$ is a constant $a$ that is independent of element of $\S$. By Lemma~2.4 of \cite{MOST}, $M(\x)$ gives the same value for both $g_1$ and $g_2$.  
	Thus the identity map $id: C^-(g_1) \rightarrow C^-(g_2)$ is an $(H_1(E(f)),\Z)$ bigraded $R_V$-module chain map of degree $(a,0)$.  Thus $(C^-(g_1),\partial^-)$ is quasi-isomorphic to $(C^-(g_1),\partial^-)$ as RA $(H_1(E(f)),\Z)$ bigraded $R_V$-modules.  If $g_1$ and $g_2$ are saturated graph grid diagrams that are related by a commutation$'$ or stabilization$'$ moves then by Propositions~\ref{prop:commprime} and \ref{prop:stab}, $(C^-(g_1),\partial^-)$ is quasi-isomorphic to $(C^-(g_1),\partial^-)$ as RA $(H_1(E(f)),\Z)$ bigraded $R_V$-modules.  By Theorem~\ref{thm:main_grid}, $g_1$ and $g_2$ are related by a finite sequence of graph grid moves which completes the proof. 
\end{proof}

By Theorem~\ref{thm:invariance}, the following definitions of $QI^-(f)$ and $HFG^-(f)$ are well-defined and independent of of choice of grid diagram.

\begin{definition} Let $f: G \rightarrow S^3$ be a sinkless and sourceless transverse spatial graph.  We define $QI^-(f)$ to be the quasi-isomorphism class of the RA $(H_1(E(f)), \Z)$ bigraded $R_V$-module chain complex $(C^-(g),\partial^-)$, for any saturated graph grid diagram $g$ representing $f$.  The \textbf{graph Floer homology of $f$}, denoted $HFG^-(f)$, is the homology of $(C^-(g),\partial^-)$ viewed as a RA $(H_1(E(f)), \Z)$ bigraded $R_V$-module, for any saturated graph grid diagram $g$ representing $f$.  
\end{definition}

We note that $C^-(g)$ is a finitely generated $R_n$-module.  However, as a $R_V$-module, it is not finitely generated.  Using Proposition~\ref{prop:no_edges}, we can show that $HFG^-(f)$ is a finitely generated $R_V$-module. 

\begin{proposition}\label{prop:fg}$HFG^-(f)$ is a finitely generated $R_V$-module for any sinkless and sourceless transverse spatial graph $f: G \rightarrow S^3$.
\end{proposition}
\begin{proof}The proof is similar to Lemma~2.13 in \cite{MOST}. Let $g$ be a saturated graph grid diagram representing $f$. We first note that $C^-(g)$ is a finitely generated $R_n$-module so $H_\ast(C^-(g))$ is finitely generated as a $R_n$-module.  Let $[z_1], \dots, [z_r]$ be the generators of $H_\ast(C^-(g))$ as a  finitely generated $R_n$-module.  Let $[c] \in H_\ast(C^-(g))$. Then we can write $[c] = \sum p_i [z_i]$ for some $p_i \in R_n$.  Let $V+1 \leq j \leq n$ and $[b] \in HFG^-(f)$, then by Proposition~\ref{prop:no_edges},  $U_j [b]$ is either $0$ or equal to $U_i [b]$ for some $1\leq i \leq V$.   Using this repeatedly, it follows that $p_i [z_i] = q_i [z_i]$ for some $q_i \in R_V$.  
	\end{proof}

\subsection{Tilde and Hat Variants} \label{sec:tilde_hat} For a saturated graph grid diagram $g$, we can define two other variants of $(C^-(g),\partial^-)$.
First we define the hat theory.  Let $\mathcal{U}_V$ be the $\mathbb{F}$-vector subspace of $C^-(g)$ spanned by $U_1 C^-(g) \cup \dots \cup U_V C^-(g)$.  Define $\widehat{C}(g)$ to be the quotient $C^-(g)/\mathcal{U}_V$.  Since $\partial^-(\mathcal{U}_V) \subset \mathcal{U}_V$, it follows that $\partial^-$ descends to an $R_V$-module homomorphism $$\widehat{\partial}: \widehat{C}(g) \rightarrow \widehat{C}(g).$$  Since $C^-(g)$ has a basis of homogeneous elements $\{b_i\}_{i\in I}$ as an $\mathbb{F}$ vector space, with respect to the $(H_1(E(f)),\mathbb{Z})$ grading on $C^-(g)$, and $\mathcal{U}_V$ has a basis that is a subbasis of $\{b_i\}_{i\in I}$, the $(H_1(E(f)),\mathbb{Z})$ bigrading on $C^-(g)$ descends to a well-defined $(H_1(E(f)),\mathbb{Z})$ bigrading on $\widehat{C}(g)$.

		\begin{definition} Let $g$ be saturated graph grid diagram representing the sinkless and sourceless transverse spatial graph $f:G \rightarrow S^3$.  The \textbf{graph Floer hat chain complex} of $g$ is the $(H_1(E(f)), \Z)$ bigraded chain complex $(\widehat{C}(g),\widehat{\partial})$.  The \textbf{graph Floer hat homology of $g$}, denoted $\widehat{HFG}(g)$, is the homology of $(\widehat{C}(g),\widehat{\partial})$ viewed as an $(H_1(E(f)), \Z)$ bigraded vector space over $\mathbb{F}$.
		\end{definition}

For a given sinkless and sourceless transverse spatial graph $f$, we can use Theorem~\ref{thm:invariance} to show that the quasi-isomorphism class (and hence homology) of $(\widehat{C}(g),\widehat{\partial})$ does not depend on the choice of graph grid diagram representing $f$. The following lemma is well-known but we include it for completeness.

\begin{lemma}Let $C$ and $D$ be $\mathbb{F}[U_1,\dots,U_V]$-module chain complexes and $\phi:C \rightarrow D$, an $\mathbb{F}[U_1,\dots,U_V]$-module quasi-isomorphism.  Then $\phi$ descends to a quasi-isomorphism $\widehat{\phi}: \widehat{C} \rightarrow \widehat{D}$ of $\mathbb{F}$-vector spaces where $\widehat{C}=C/\mathcal{U}_V$ and $\widehat{D}=D/\mathcal{U}_V$.  
\end{lemma}

\begin{proof}We prove this by induction on $V$.  Suppose $V=1$.  Then $\phi$ is an $\mathbb{F}[U_1]$-module homomorphism, and $\phi$ descends to a well-defined chain map $\widehat{\phi}: \widehat{C} \rightarrow \widehat{D}$.  The following diagram commutes and the horizontal sequences are exact.  
	\begin{diagram} 0 & \rTo & C & \rTo^{U_1} & C & \rTo^{q} & \widehat{C} & \rTo & 0 \\
		  &      & \dTo^{\phi} &  & \dTo^{\phi} & & \dTo^{\widehat{\phi}} \\
		0 & \rTo & D & \rTo^{U_1} & D & \rTo^{q} & \widehat{D} & \rTo & 0 
		\end{diagram} 
		Here $U_1$ indicates that map that is multiplication by $U_1$ and $q$ is the quotient map.   Thus on homology, we get the following commutative diagram with horizontal long exact sequences.  
		\begin{diagram}  H_\ast(C) & \rTo^{(U_1)_\ast} & H_\ast(C) & \rTo^{q_\ast} & H_\ast(\widehat{C}) & \rTo & H_\ast(C) & \rTo^{(U_1)_\ast} & H_\ast(C) \\
			     \dTo^{\phi_\ast} &  & \dTo^{\phi_\ast} & & \dTo^{\widehat{\phi}_\ast} & &\dTo^{\phi_\ast} &  & \dTo^{\phi_\ast} \\
			  H_\ast(D) & \rTo^{(U_1)_\ast} & H_\ast(D) & \rTo^{q_\ast} & H_\ast(\widehat{D}) & \rTo & H_\ast(D) & \rTo^{(U_1)_\ast} & H_\ast(D)  
			\end{diagram}
			Since $\phi_\ast$ is an isomorphism, by the $5$-lemma, so is $\widehat{\phi}_\ast$.

			Now suppose the lemma is true for $V$ and let $C$ and $D$ be $\mathbb{F}[U_1,\dots,U_{V+1}]$-module chain complexes and $\phi:C \rightarrow D$, an $\mathbb{F}[U_1,\dots,U_{V+1}]$-module quasi-isomorphism.  Consider $C$ and $D$ as $\mathbb{F}[U_1,\dots,U_{V}]$-modules.  Then by the inductive hypothesis, $\phi': C/\mathcal{U}_V \rightarrow D/\mathcal{U}_V$ is a quasi-isomorphism where $\phi'$ is induced from $\phi$ (which we called $\widehat{\phi}$ before).  Moreover, note that $\phi'$ is an $\mathbb{F}[U_{V+1}]$-module homomorphism.  Let $C' = C/\mathcal{U}_V$ and $D' = D/\mathcal{U}_V$. Using the proof from the case when $V=1$, we can show that the induced map $\widehat{\phi'}: C'/U_{V+1} C' \rightarrow D'/U_{V+1} D'$ is a quasi-isomorphism.  It is straightforward to show that the natural map $C/\mathcal{U}_{V+1} \rightarrow C'/U_{V+1} C'$ is a chain isomorphism (similarly for $D$), which completes the proof. \end{proof}

	\begin{corollary}\label{cor:hatinvariance} If $g_1$ and $g_2$ are saturated graph grid diagrams representing the same transverse spatial graph $f:G \rightarrow S^3$ then $(\widehat{C}(g_1),\widehat{\partial})$ is quasi-isomorphic to $(\widehat{C}(g_2),\widehat{\partial})$ as RA $(H_1(E(f)),\Z)$ bigraded vector spaces.  In particular, 	$\widehat{HFG}(g_1)$ is isomorphic to $\widehat{HFG}(g_2)$ as RA $(H_1(E(f)),\Z)$ bigraded vector spaces. 
		\end{corollary}	
		
		As a result, the following definitions of $\widehat{QI}(f)$ and $\widehat{HFG}(f)$ are well-defined and independent of of choice of graph grid diagram.

		\begin{definition} Let $f: G \rightarrow S^3$ be a sinkless and sourceless transverse spatial graph.  We define $\widehat{QI}(f)$ to be the quasi-isomorphism class of the RA $(H_1(E(f)), \Z)$ bigraded chain complex $(\widehat{C}(g),\widehat{\partial})$, for any saturated graph grid diagram $g$ representing $f$.  The \textbf{graph Floer hat homology of $f$}, denoted $\widehat{HFG}(f)$, is the homology of $(\widehat{C}(g),\widehat{\partial})$ viewed as a RA $(H_1(E(f)), \Z)$ bigraded vector space over $\mathbb{F}$, for any saturated graph grid diagram $g$ representing $f$.  
		\end{definition}
		
		Note that $(\widehat{C}(g),\widehat{\partial})$ is an infinitely generated vector space but, but in the same way as Proposition~\ref{prop:fg}, one can show that its homology is finitely generated. 
		
		\begin{proposition}\label{prop:fghat}$\widehat{HFG}(f)$ is a finitely generated vector space over $\mathbb{F}$ for any sinkless and sourceless transverse spatial graph $f: G \rightarrow S^3$.
		\end{proposition}
		\begin{proof} Choose a saturated graph grid diagram $g$ representing $f$.  Let $j$ be such that $V+1 \leq j \leq n$.  Then by Proposition~\ref{prop:no_edges}, there is a chain homotopy $H :C^-(g) \rightarrow C^-(g)$ ($H$ is $H_k$ for some $k$) that is an $R_V$-module homomorphism that satisfies one of the following conditions: (i)  $\partial^- \circ H + H \circ \partial^- = U_j$ or (ii) there exists an $1 \leq i \leq V$ such that $\partial^- \circ H + H \circ \partial^- = U_i + U_j$.   Since $H$ is an $R_V$-module homomorphism, $H$ descends to a well-defined $\widehat{H}:\widehat{C}(g) \rightarrow \widehat{C}(g)$ satisfying $$ \widehat{\partial} \circ \widehat{H} + \widehat{H} \circ \widehat{\partial} = U_j.$$  So $U_j[b]=0$ for all $[b] \in \widehat{HFG}(g)$.  The rest of the proof is similar to the proof of Proposition~\ref{prop:fg}.
			\end{proof}

			We now define the tilde theory.  The tilde theory will be the easiest theory to compute.  However, it narrowly fails to be an invariant of the spatial graph since it will depend on the grid size.  On the other hand, one can recover the hat theory from it which makes it quite useful.  It will also be easier to compute the bigraded Euler characteristic (Alexander polynomial) of the hat theory using tilde theory; for more details, see Section~\ref{sec:Alex}.  
			
			Let $\mathcal{U}_n$ be the $\mathbb{F}$-vector subspace of $C^-(g)$ spanned by $U_1 C^-(g) \cup \dots \cup U_n C^-(g)$.  Define $\widetilde{C}(g)$ to be the quotient $C^-(g)/\mathcal{U}_n$.  Since $\partial^-(\mathcal{U}_n) \subset \mathcal{U}_n$, it follows that $\partial^-$ descends to an linear map $$\widetilde{\partial}: \widetilde{C}(g) \rightarrow \widetilde{C}(g)$$ of vector spaces over $\mathbb{F}$.  
			Since $C^-(g)$ has a basis of homogeneous elements $\{b_i\}_{i\in I}$ as an $\mathbb{F}$ vector space, with respect to the $(H_1(E(f)),\mathbb{Z})$ grading on $C^-(g)$, and $\mathcal{U}_n$ has a basis that is a subbasis of $\{b_i\}_{i\in I}$, the $(H_1(E(f)),\mathbb{Z})$ bigrading on $C^-(g)$ descends to a well-defined $(H_1(E(f)),\mathbb{Z})$ bigrading on $\widetilde{C}(g)$. 
			Thus $(\widetilde{C}(g), \widetilde{\partial})$ is a $(H_1(E(f)),\mathbb{Z})$ bigraded chain complex. 
			
			\begin{definition} Let $g$ be a saturated graph grid diagram representing the sinkless and sourceless transverse spatial graph $f:G \rightarrow S^3$.  The \textbf{graph Floer tilde chain complex} of $g$ is the $(H_1(E(f)), \Z)$ bigraded chain complex $(\widetilde{C}(g),\widetilde{\partial})$.  The \textbf{graph Floer tilde homology of $g$}, denoted $\widetilde{HFG}(g)$, is the homology of $(\widetilde{C}(g),\widetilde{\partial})$ viewed as an $(H_1(E(f)), \Z)$ bigraded vector space over $\mathbb{F}$.
			\end{definition}	  	
			
			We will relate $\widetilde{HFG}(g)$ and $\widehat{HFG}(g)$ for a given graph grid diagram $g$.  First, we recall the (bigraded) mapping cone which we will use in the next lemma.  
				
				Let $(A,\partial^A)$ and $(B,\partial^B)$ be $(\G,\mathbb{Z})$ bigraded chain complexes and let $\phi: A \rightarrow B$ be a bigraded chain map of degree $(g,m)$ for some $g\in \G$ and $m\in\mathbb{Z}$.  \label{def:mc} Define the \textbf{(bigraded)  mapping cone complex of $\phi$}, denoted $(cone(\phi),\partial)$, as follows: $$cone(\phi) = \bigoplus_{(h,n)\in \G \oplus \Z} cone(\phi)_{(h,n)}$$ where 
				$$cone(\phi)_{(h,n)} = A_{(h-g,n-m-1)} \oplus B_{(h,n)}$$ 
				and the boundary map is defined as 
				$$\partial(a,b) = (-\partial^A(a), -\phi(a) + \partial^B(b) )$$for all $a \in A$ and $b\in B$.  Checking the definitions, we see that $\partial$ is a bigraded map of degree $(0,-1)$.  We also note that if $A$ and $B$ are $(\G,\mathbb{Z})$ bigraded $R$-module chain complexes and $\phi$ is a bigraded $R$-module chain map then $(cone(\phi),\partial)$ is an $(\G,\mathbb{Z})$ bigraded $R$-module chain complex.
			
The following Lemma is similar to Lemma 2.14 of \cite{MOST} except now we have bigraded chain complexes instead of filtered, graded chain complexes.  For $g\in \G$ and $m \in \mathbb{Z}$, let $W{(-g,-m+1)}$ be the two dimensional $(\G,\Z)$ bigraded vector space over $\mathbb{F}$ spanned by one generator in degree $(0,0)$ and the other in degree $(-g,-m+1)$.  If $(C,\partial)$ is any bigraded $(\G, \Z)$ chain complex over $\mathbb{F}$, then $C \otimes W{(-g,-m+1)}$ becomes a bigraded chain complex with boundary $\partial \otimes id$ in the usual way.  That is, 
$$(C \otimes W{(-g,-m+1)})_{(h,l)} = \bigoplus_{(h,l) = (h_1+h_2,l_1+l_2)} C_{(h_1,l_1)} \otimes W{(-g,-m+1)}_{(h_2,l_2)}$$

\begin{lemma}\label{lem:homotopic} Let $(C,\partial)$ be a $(\G,\mathbb{Z})$ bigraded $\mathbb{F}[U_1, \dots, U_s]$-module chain complex and $g\in \G$ and $m\in \Z$ be fixed group elements.   Suppose that for each $i \geq 2$, multiplication by $U_i$ (which we denote by $U_i$) is a bigraded $\mathbb{F}[U_1, \dots, U_s]$-module chain map of degree $(-g,-m)$ and that 
\begin{enumerate}
\item $U_i$ is chain homotopic to $U_1$ or 
\item $U_i$ is null-homotopic (where the chain homotopy is an $\mathbb{F}[U_1, \dots, U_s]$-module homomorphism).
\end{enumerate}
Then $(C/\mathcal{U}_s,\partial)$ is quasi-isomorphic to $(C/\mathcal{U}_1 \otimes {W{(-g,-m+1)}}^{\otimes s-1}, \partial \otimes id)$ and hence
		$$H_\ast(C/\mathcal{U}_s) \cong H_\ast(C/\mathcal{U}_1) \otimes {W{(-g,-m+1)}}^{\otimes s-1}.$$ 
\end{lemma}		

Note that by $H_\ast(C/\mathcal{U}_1)$ (respectively $H_\ast(C/\mathcal{U}_s)$), we mean the homology of the chain complex whose chain group is $C/\mathcal{U}_1$ (respectively $C/\mathcal{U}_s$) and whose boundary map is induced by $\partial$.\\

\begin{proof}Let $D = C/\mathcal{U}_1 = C/U_1 C$ and $\partial^D: D \rightarrow D$ be induced by $\partial$.  Consider multiplication by $U_2$ on $D$, $\widehat{U}_2: D \rightarrow D$.  Since $U_1$ and $U_2$ commute, this is a well-defined bigraded map. There is a long exact sequence $$0 \rightarrow  D \xrightarrow{\widehat{U}_2} D \xrightarrow{pr} D/\widehat{U}_2 D \rightarrow 0$$ where $pr$ is the natural projection to the quotient.  Let $F: cone(\widehat{U}_2) \rightarrow D/\widehat{U}_2 D$ be defined by $F(c_1,c_2)= pr(c_2)$. By 1.5.8 of \cite{Weibel}, the map $F$ is a quasi-isomorphism.  Moreover, $F$ is a bigraded map of degree $(0,0)$. 
	
	If $U_2$ is chain homotopic to $U_1$ via a chain homotopy $H$ that is an $\mathbb{F}[U_1, \dots, U_s]$-module homomorphism then $H$ induces a well-defined map on $D$, $\widehat{H}:D \rightarrow D$ such that 
	$$\partial^D \circ \widehat{H} + \widehat{H} \circ \partial^D = \widehat{U}_2.$$  This also holds if $U_2$ is null-homotopic.  Since $\widehat{U}_2$ is nullhomotopic, there is a bigraded chain isomorphism from $cone(\widehat{U}_2)$ to $cone(0:D \rightarrow D)$ of degree $(0,0)$.  Hence $cone(\widehat{U}_2)$ is isomorphic to $D \oplus D[g,m-1]$ as bigraded chain complexes where $D[g,m-1]$ is the bigraded vector space defined by $D[g,m-1]_{(h,n)} = D_{h+g,n+m-1}$ and the boundary map on $D \oplus D[g,m-1]$ is $\partial^D \oplus \partial^D$.  Moreover, this is isomorphic, as bigraded chain complexes, to $D \otimes W{(-g,-m+1)}$.  The proof for $n=2$ is complete after noting that the obvious map $C/\mathcal{U}_2$ to $D/\widehat{U}_2 D$ is a bigraded chain isomorphism. 	To complete this proof, continue this type of argument, one by one for each $U_i$.
	 \end{proof}

We can use this to relate the tilde and hat chain complexes. 
 
\begin{proposition} \label{prop:hfg_tensor}
	Let $g$ be a saturated graph grid diagram representing the sinkless and sourceless transverse spatial graph $f:G \rightarrow S^3$.  Then  $$\widetilde{HFG}(g) \cong \widehat{HFG}(g) \otimes \bigotimes_{e \in E(G)} {W{(-w(e),-1)}}^{\otimes n_e} $$
as $(H_1(E(f)),\Z)$	bigraded $\mathbb{F}$ vector spaces, where $n_e$ is the number of $O$'s in $g$ associated to the interior of $e$ (not including the vertices).	
\end{proposition}

\begin{proof} We note that if $O_i$ is on the interior of edge $e$ then multiplication by $U_i$ is a graded map of degree $(-w(e),-2)$.  In addition, $U_i$ is either null-homotopic or homotopic to some $U_j$ with $j\leq V$ (where $O_j$ is a vertex).  Use the Lemma~\ref{lem:homotopic} repeatedly to complete the proof. 
\end{proof}

\section{Invariance of $HFG^-(f)$}\label{sec:invariance}

In this section we will complete the necessary steps to prove that $HFG^-(f)$ is an invariant of a sinkless and sourceless transverse spatial graph (see Theorem~\ref{thm:invariance} and its proof for more details).  That is, we will show that  $HFG^-(g)$ is invariant under each of the graph grid moves: commutation$'$ and stabilization$'$.   The proofs will be similar to the proofs found in \cite{MOST}.  There are two major differences though.  The first is that we are working with more general commutation and  stabilization moves, commutation$'$ and stabilization$'$.  
The second difference is the Alexander grading.  

\vspace{12pt}

\subsection{Commutation$'$ invariance}
\label{subsec:Commutation}

\begin{proposition}\label{prop:commprime}Suppose $g$ and $\bar{g}$ are saturated graph grid diagrams that differ by a commutation$'$ move.  Let $f:G \rightarrow S^3$ be the transverse graph associated to $g$ and $\bar{g}$ and $V$ be the number of vertices of $G$.  Then there an $(H_1(E(f)),\Z)$ bigraded $R_V$-module quasi-isomorphism $(C(g),\partial^-_g) \rightarrow (C(\bar{g}),\partial^-_{\bar{g}})$ of degree $(\delta(g,\bar{g}),0)$ for some $\delta(g,\bar{g}) \in H_1(E(f))$.
\end{proposition}

\noindent We remark that the quasi-isomorphism above will be a bigraded map for some degree (that depends on $g$ and $\bar{g}$) but will not necessarily be of degree $(0,0)$.
 
\vspace{10pt}
The proof of Proposition~\ref{prop:commprime} will take up the rest of this subsection.  We will prove the case when $\bar{g}$ is obtained from $g$ by a commutation$'$ move of columns.  The case where you exchange rows is similar. 

As in \cite{MOST}, we draw both graph grid diagrams on a single $n \times n$ grid (respectively torus when the sides are identified), which we will call the combined grid diagram, as follows.  Let the vertical line segment (respectively circle) between the columns that are exchanged be labeled $\beta$ in $g$ and  $\gamma$ in $\bar{g}$ and call the other vertical circles $\beta_1,\dots,\beta_{n-1}$ (where $n$ is the size of the grid for $g$). Let $\gamma$ be a simple closed curve on the graph grid diagram $g$ such that the following conditions are held:  (1) $\gamma$ is homotopic to $\beta$, (2) $\gamma$ hits each of the horizontal curves, $\alpha_i$, precisely once, (3) $\gamma$ does not intersect $\beta_i$ for $i\leq n-1,$ (4) after removing the $\beta$ curve, one obtains $\bar{g}$,  (5) $\gamma$ and $\beta$ intersect transversely exactly twice, and (6) the intersections of $\gamma$ and $\beta$ do not lie on the horizontal curves.  It is easy to use the line segments $LS_1$ and $LS_2$ in the definition of commutation$'$ to see that such a curve exists.  First, note that we can assume the endpoints of the line segments do not lie on $\alpha$ curves by slightly changing them.  Now, take pushoffs of the line segments $LS_1$ and $LS_2$ to the left or right as needed, and connect them up so that they satisfy the requirements above.   Let $a$ and $b$ be the intersections of $\beta$ and $\gamma$.  
See Figures ~\ref{Comm3} and \ref{samerow} for examples.  We still let $\mathcal{T}$ be the torus of the combined grid diagram obtained by gluing the top/bottom and sides.

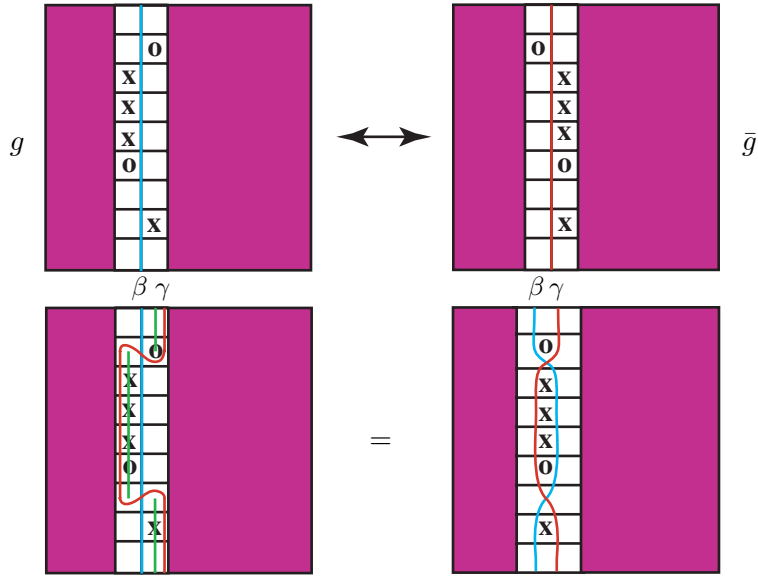
\begin{figure}[htpb!]
\begin{center}
\begin{picture}(257, 215)
\put(0,0){\includegraphics{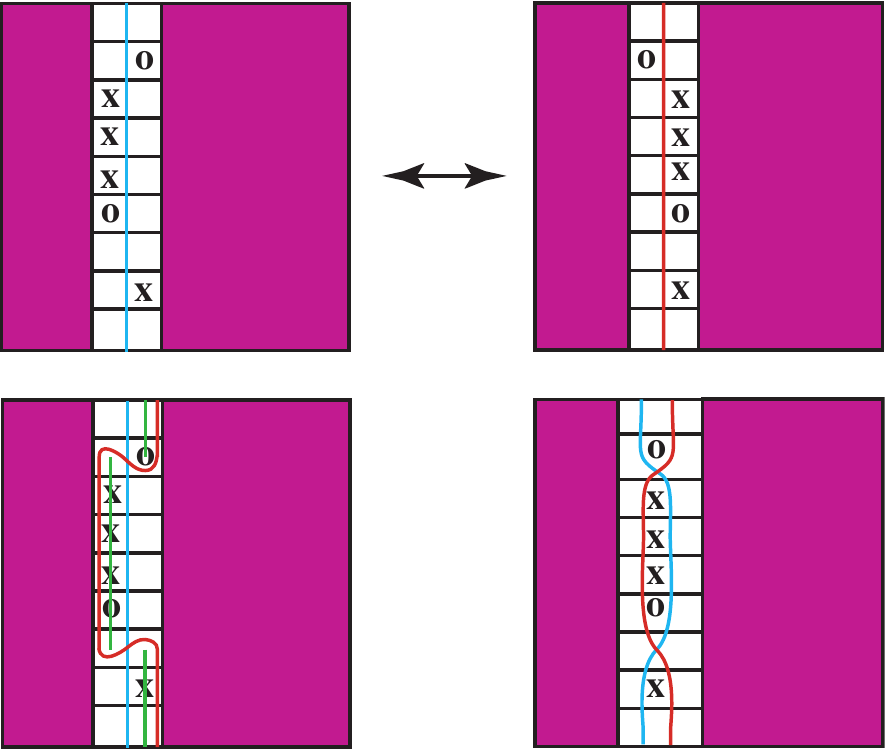}}
\put(-13,161){$g$}
\put(264,161){$\bar{g}$}
\put(33,106){\small $\beta$}
\put(42,106){\small $\gamma$}
\put(183,106){\small $\beta$}
\put(191,106){\small $\gamma$}
\put(123,48){$=$}
\end{picture}
\end{center}
\caption {{\bf Commutation$'$.}
On the top is a commutation$'$ move between $g$ and $\bar{g}$. On the bottom is the corresponding combined grid diagram where we have shown the line segments in the definition of commutation$'$ in the bottom left picture.}
\label{Comm3}
\end{figure}

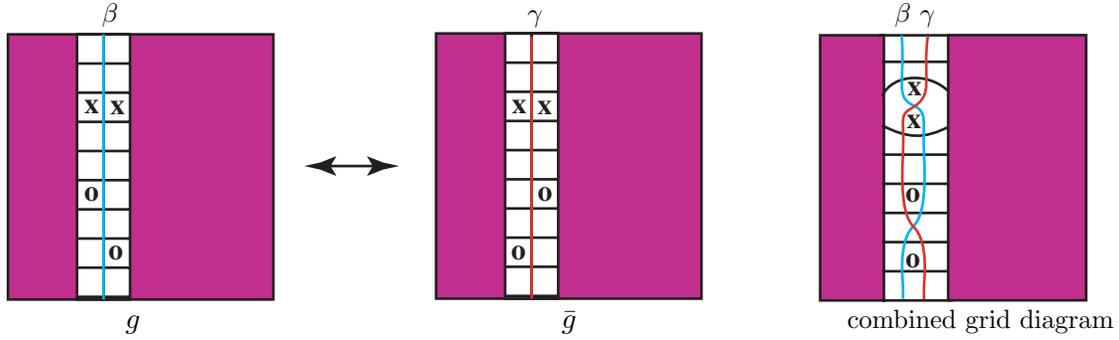
\begin{figure}[htpb!]
\begin{center}
\begin{picture}(408,122)
\put(0,10){\includegraphics{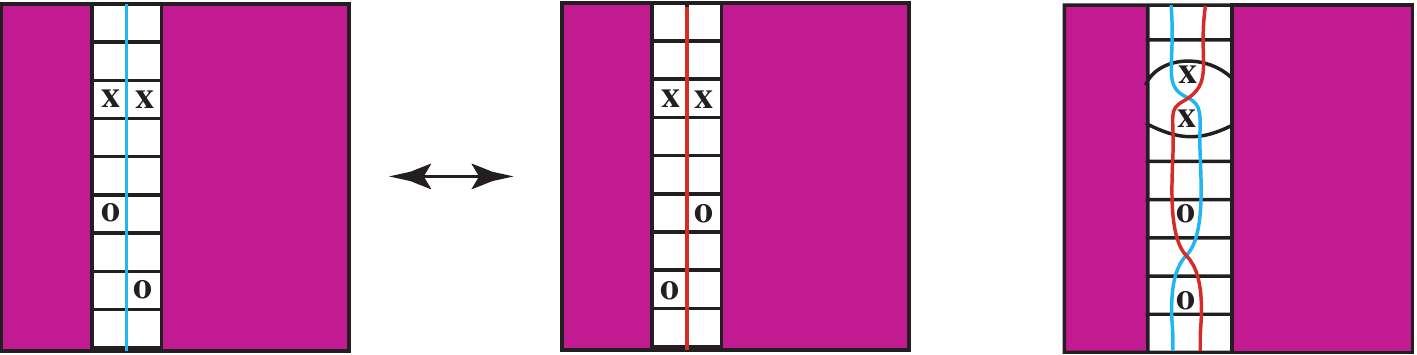}}
\put(45,0){$g$}
\put(210,0){$\bar{g}$}
\put(36,116){\small $\beta$}
\put(197,116){\small $\gamma$}
\put(336,116){\small $\beta$ $\gamma$}
\put(318,0){\small combined grid diagram}
\end{picture}
\end{center}
\caption {{\bf Commutation$'$.} Another example of a commutation$'$ move, where $X$'s appear in the same row.  }
\label{samerow}
\end{figure}

We will define a chain map $\Phi_{\beta\gamma}: C^-(g) \rightarrow C^-(\bar{g})$ and show that it is a chain homotopy equivalence.   This will show that $\Phi_{\beta\gamma}$ is a quasi-isomorphism.  
For $\x\in \S(g)$ and $\y\in\S(\bar{g})$, we let $\Pent_{\beta\gamma}(\x,\y)$ be the set of embedded pentagons with the following properties.  If $\x$ and $\y$ do not coincide at $n-2$ points, then we let $\Pent_{\beta\gamma}(\x,\y)=\emptyset$.  Suppose that $\x$ and $\y$ coincide at $n-2$ points (say $x_3=y_3, \dots, x_n=y_n$).  Without loss of generality, let $x_2 = \x \cap \beta$ and $y_2 = \y \cap \gamma$.  An element $p \in \Pent_{\beta \gamma}(\x,\y)$ is an embedded disk in $\mathcal{T}$, whose boundary consists of five arcs, each of which are contained in the circles $\beta_i$, $\alpha_i$, $\beta$, or $\gamma$ and satisfies the following conditions.  The intersections of the arcs lie on the points $x_1,x_2,y_1,y_2$ and $a.$  The point $a$ is in $\beta \cap \gamma$ and locally looks like the top intersection in Figure~\ref{Pentagons} ($b$ is the one that locally looks the bottom intersection point in $\beta \cap \gamma$).
Moreover, start at the point in $x_2$ and transverse the boundary of $p$, using the orientation given by $p$.  The condition to be in $\Pent_{\beta \gamma}(\x,\y)$ is that you will first travel along a horizontal circle, meet $y_1$, proceed along a vertical circle $\beta_i$, meet $x_1$, continue along another horizontal circle, meet $y_2$, proceed though an arc in $\gamma$ until you meet $a$, and finally traverse an arc in $\beta$ until arriving back at $x_2$. Finally, all angles are required to be less than straight.

The set of empty pentagons, $\Pent^o_{\beta\gamma}(\x,\y)$, are those pentagons $p\in \Pent_{\beta\gamma}(\x,\y)$ such that $\x \cap \Interior(p) =\emptyset$.
The map $\Phi_{\beta\gamma} \colon C^-(g) \longrightarrow C^-(\bar{g})$ is defined by counting empty pentagons, that do not contain $X$'s in the combined grid diagram as follows.  For $\x \in \S(g)$, define

$$\Phi_{\beta\gamma}(\x) = \sum_{\y\in \S(\bar{g})}\,\,
    \sum_{\substack{p\in\Pent^o_{\beta\gamma}(\x,\y)\\Int(p)\cap\mathbb{X}=\emptyset}}
U_1^{O_1(p)}\cdots U_n^{O_n(p)} \cm \y
\in C^-(\bar{g}).$$
Extend $\Phi_{\beta\gamma}$ to $C^-(g)$ so that it is an $R_n$-module homomorphism.  In particular, it is also an $R_V$-module homomorphism.

\begin{figure}[htpb!]
\begin{center}

\begin{picture}(200, 135)
\put(0,0){\includegraphics[scale=0.9]{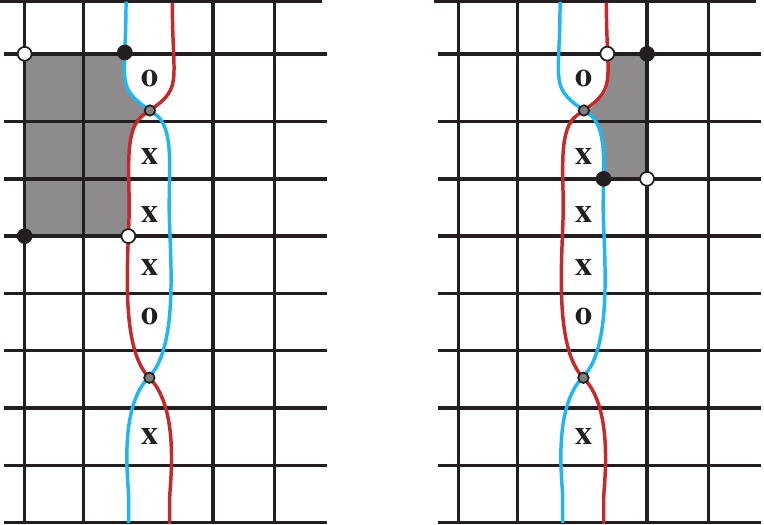}}
\put(45,107){\tiny $a$}
\put(157,107){\tiny $a$}
\put(45, 36){\tiny $b$}
\put(157,36){\tiny $b$}
\end{picture}
\end{center}
\caption {Examples of pentagons in $\Pent^o_{\beta\gamma}(\x,\y).$}
\label{Pentagons}
\end{figure}

\begin{lem}
  \label{lemma_ChainMap}
  $\Phi_{\beta\gamma}$ is $\HZ$ bigraded $R_V$-module chain map of some degree. 
\end{lem}

\begin{proof}  Since $\Phi_{\beta\gamma}$ is an $R_V$-module homomorphism, the proof is broken into three parts: checking each gradings is preserved, and showing that 
$$\partial^-\circ\Phi_{\beta\gamma}=\Phi_{\beta\gamma}\circ\partial^-.$$

\vspace{12pt}
\noindent\textit{The map $\Phi_{\beta\gamma}$ preserves the Maslov grading}: Since the definition of the Maslov grading only depends on the set $\mathbb{O}$, and we consider a subset of the pentagons considered in the proof of commutation in ~\cite{MOST} (Section 3.1), this technically follows from Lemma 3.1 in \cite{MOST}.  However, since they do not include a  proof that the Maslov grading is preserved (this is left to the reader), we will include a sketch of the proof here. 
  
We will go though the details of this computation for the case pictured in Figure \ref{regions}, other cases follow similarly.  
Consider a $U_1^{O_1(p)}\cdots U_n^{O_n(p)} \cm \y$ in the sum of $\Phi_{\beta\gamma}(\x).$  Recall that  
$$M(\x)=\mathcal{J}(\x, \x)-2\mathcal{J}(\x, \Os)+\mathcal{J}(\Os, \Os)+1.$$  
To compare the Malsov grading we interpret each of these terms for $\x$ in the grid $g$ in relation to $\y$ in the grid $\bar{g}$.  Let the intersection points of $\x$ be $x_1,\dots, x_n$ and the intersection points of $\y$ be $y_1,\dots, y_n,$ with the same subscript where they coincide.  Label the intersections points where $\x$ and $\y$ differ as $x_1,$ $x_2,$ $y_1,$ and $y_2$ and break the combined grid diagram into 14 regions labeled A,\dots,M, $p$, as shown in Figure~\ref{regions}.   

\begin{figure}[htpb!]
\begin{center}
\begin{picture}(108, 172)
\put(0,0){\includegraphics{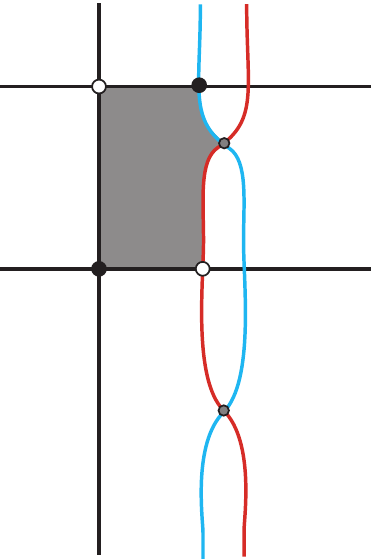}}
\put(12,145){A}\put(40,145){B}\put(61,145){C}\put(83,145){D}
\put(12,108){E}\put(40,108){$p$}\put(61,128){F}\put(61,95){G}\put(83,108){H}
\put(12,45){I}\put(40,45){J}\put(61,65){K}\put(61,20){L}\put(83,45){M}
\put(70,119){\tiny $a$}\put(70, 42){\tiny $b$}\put(20,130){\tiny $y_1$}
\put(20,78){\tiny $x_1$}\put(50,130){\tiny $x_2$}\put(50,78){\tiny $y_2$}
\end{picture}
\end{center}
\caption {The combined grid diagram with regions A,\dots,M labelled and the pentagon $p$ shaded. }
\label{regions}
\end{figure}
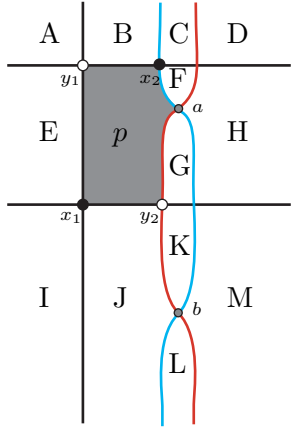

Notice that the count for $x_i$ is the same as $y_i$ is the same for $i\neq 1,2$.  The number of points in $\Os$ up and to the right, and down and to the left are not changed, since this could only be changed for an intersection between the commuted edges (i.e.~$x_2$ or $y_2$).  So $\mathcal{J}(x_i,\Os)=\mathcal{J}(y_i,\Os)$.  The number of points in $\x$ and $\y$ up and to the right and down and to the left are the same, this can be checked region by region.  If an intersection point is in region E then $x_2$ will be counted in the points up and to the right, this is replaced by the point $y_1$ which is also up and to the right.  Similarly, for all of the regions, since $x_i=y_i\in\{\text{A, B, . . . , M}\},$ thus $\mathcal{J}(x_i,\x)=\mathcal{J}(y_i,\y)$ for all $i\neq 1,2.$  

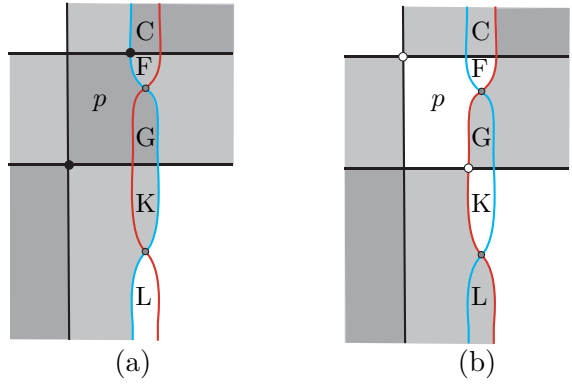
\begin{figure}[htpb!]
\begin{center}

\begin{picture}(200, 140)
\put(0,10){\includegraphics[scale=0.8]{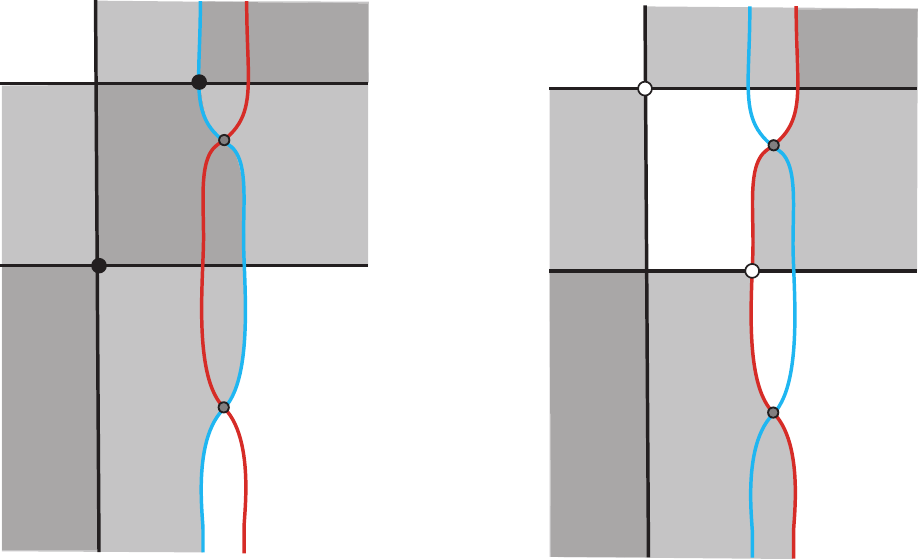}}
\put(40,0){(a)}
\put(170,0){(b)}
\put(32, 100){\small $p$} 
\put(48, 125){\small C}
\put(48, 112){\small F}
\put(48, 85){\small G}
\put(48, 60){\small K}
\put(48, 25){\small L}
\put(160, 100){\small $p$} 
\put(175, 125){\small C}
\put(175, 111){\small F}
\put(175, 85){\small G}
\put(175, 60){\small K}
\put(175, 25){\small L}
\end{picture}
\end{center}
\caption {\small (a) Shows with 50\% shading those regions that will be counted with weight one half in $\mathcal{J}(x_1, -)$ and $\mathcal{J}(x_2, -).$  (b) Shows with 50\% shading those regions that will be counted with weight one half in $\mathcal{J}(y_1, -)$ and $\mathcal{J}(y_2, -).$}
\label{count}
\end{figure}

Now for the points where $\x$ and $\y$ differ, Figure~\ref{count}(a) shows the regions that are counted for $x_1$ and $x_2$, and Figure~\ref{count}(b) shows the regions that are counted for $y_1$ and $y_2.$  So we see that the regions C, F, G, and K are counted with weight $\frac{1}{2}$ more for $x_1$ and $x_2$, and the region $p$ is counted with weight 1 more for $x_1$ and $x_2$, last region L is counted with weight $\frac{1}{2}$ less for $x_1$ and $x_2.$  
So, 
$$\mathcal{J}(\x,\x)=\mathcal{J}(\y,\y)+1,$$ because $x_1$ will count $x_2$ with weight $\frac{1}{2}$ and vice versa, but $y_1$ and $y_2$ do not count each other.   
Next, $\x$ will count all of the points in $\Os$ that $\y$ will count and additionally will count those $O$'s in the region $p$ with weight 1, those $O$'s in the regions C, F, G, and K with weight $\frac{1}{2}$ and those $O$'s in the region L with weight $-\frac{1}{2}.$  Notice that the region made up of G and K must contain exactly one $O$.  
Thus,
$$\mathcal{J}(\x, \Os)=\mathcal{J}(\y, \Os)+ O_1(p)+\cdots +O_n(p)+\frac{1}{2}(O(C)+O(F)+1)-\frac{1}{2}O(L),$$
where $O(C)$, $O(F)$, and $O(L)$ are the number of $O$'s in the respective regions.  

Lastly, if we look at what happens for the different diagrams with the sets of $\Os$ the only difference is for the $O$'s in the columns that are changed.  
Again we know that there is exactly one $O$ in the regions $G$ and $K.$  
So we have, 
$$\mathcal{J}(\Os, \Os)_g=\mathcal{J}(\Os, \Os)_{\bar{g}}+O(C)+O(F)-O(L).$$

Putting this all together we have:
\[
\begin{array}{ccl}
M(\x) & = & [\mathcal{J}(\y,\y)+1]    \\
  & & -2[\mathcal{J}(\y, \Os)+ O_1(p)+\cdots +O_n(p)+\frac{1}{2}(O(C)+O(F)+1)-\frac{1}{2}O(L)]     \\
  & &  +[\mathcal{J}(\Os, \Os)_{\bar{g}}+O(C)+O(F)-O(L)]+1  \\
  & = & \mathcal{J}(\y,\y) -2\mathcal{J}(\y, \Os) +\mathcal{J}(\Os, \Os)_{\bar{g}} -2[O_1(p)+\cdots +O_n(p)]+1\\
  & = & \mathcal{J}(\y-\Os, \y-\Os)+1-2[O_1(p)+\cdots +O_n(p)].
\end{array}
\]
So we see that the Maslov grading is unchanged.  

\vspace{12pt}
\noindent\textit{The map $\Phi_{\beta\gamma}$ preserves the Alexander grading up to a shift}: Consider the $\Hf$ grading on $C^-(g)$, $C^-(g) = \oplus_{a \in H_1(E(f))} C^-(g)_a$ (similarly for $C^-(\bar{g})$).  We will show that there is an $\delta(g,\bar{g})$ that only depends on $g$ and $\bar{g}$ (not on $\x$) such that $\Phi_{\beta\gamma}(C^-(g)_a) \subset C^-(\bar{g})_{a + \delta(g,\bar{g})}$.  We will work with the second definition of the Alexander grading, $A^g(\x)=\sum_{x_i\in\x}[-h^g(x_i)]$ to prove this. 

Let $\x \in \S(g)$ and $p \in  \Pent^o_{\beta\gamma}(\x,\y)$ such that 
$p \cap \mathbb{X}=\emptyset$ and $\Pent^o_{\beta\gamma}(\x,\y)\neq \emptyset$.  
Then $U_1^{O_1(p)}\cdots U_n^{O_n(p)} \cm \y$ is a term in $\Phi_{\beta\gamma}(\x)$.  
We use the convention as before that $x_i=y_i$ for $i\geq 3$, $x_2 \in \beta$, and $y_2 \in \gamma$.  
We note that $h^g(x_i) = h^{\bar{g}}(y_i)$ for $i\geq 3$.  So we need to show that 
\begin{eqnarray}\label{comm_eq} h^g(x_1) + h^g(x_2) - h^{\bar{g}} (y_1) - h^{\bar{g}} (y_2) - \sum_{i=1}^n w(O_i)O_i(p) = \delta(g,\bar{g})\end{eqnarray} 
for some fixed element $\delta(g,\bar{g}) \in H_1(E(f))$.  
	
We will prove the case when $a$ is the topmost intersection of $\beta$ and $\gamma$ and $p$ is a pentagon lying to the left of $a$. See two example of these pentagons in Figure~\ref{fig:pent12}.  Note that the boundary of the pentagon can contain $b$ and $p$ can contain an $O$ that lies between $\beta$ and $\gamma$. The other three cases are similar.  	
Let $\beta_{n-1}$ be the vertical line segment/circle in $g$ directly to the left of $\beta$, $\beta_1$ be the vertical line segment/circle in $g$ directly to the right of $\beta$.  We will label the $\alpha_i$ in the usual way so that $\alpha_i$ is height $i-1$.  Let 
$\alpha_l$ be the horizontal line segment/circle directly below $b$, $\alpha_{l+1}$ be the horizontal line segment/circle directly above $b$, $\alpha_k$ be the horizontal circle directly below $a$, and $\alpha_{k+1}$ be the horizontal line segment/circle directly above $a$. 
Finally, let $u_1$ be the point on $\beta_{n-1}$ that is at the same height as $x_1$ and let $u_2$ be the point on $\beta_{n-1}$ that is at the same height as $x_2$.  See Figures~\ref{fig:pent12} and \ref{fig:comm} for our conventions.  

We will say that the pentagon $p$ is narrow if $y_1 \in \beta_{n-1}$.  If $p$ is not narrow, then there is a rectangle in $g$ that is contained in $p$.  Let $r$ be the largest such rectangle.  Then $p$ decomposes into $r$ and a narrow pentagon $p'$.  Note that $r \in \mathrm{Rect}^o(\{x_1,u_2\},\{u_1,y_1\})$ and $p' \in \Pent^o(\{x_2,u_1\},\{y_2,u_2\})$.  Moreover, $r \cap \mathbb{X} = p\cap \mathbb{X} = \emptyset$.  Since the boundary map on $C^-(g)$ preserves the Alexander grading, we see that $$h^g(x_1)+h^g(u_2) = \sum_{i=1}^{n}w(O_i)O_i(r) + h^g(y_1) + h^g(u_1).$$  We consider $h^g(x_2)-h^g(u_2)$ and $h^{\bar{g}}(y_2)-h^{\bar{g}}(u_1)=h^{\bar{g}}(y_2)-h^{g}(u_1)$.
	
\begin{center}	
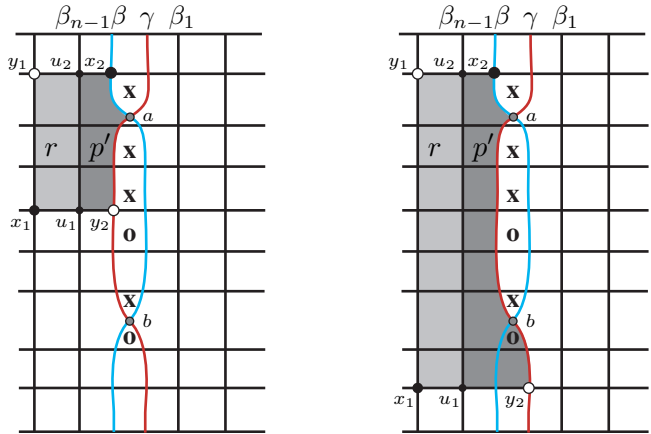
\begin{figure}[htpb]
\begin{picture}(250, 170)
\put(0,0){\includegraphics{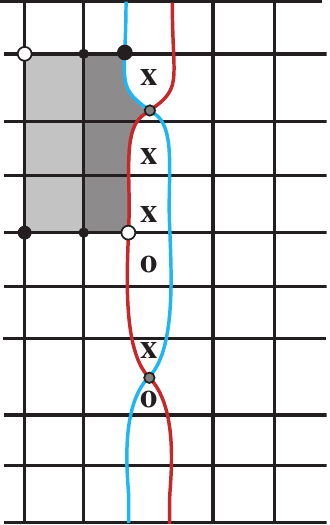}}
\put(15,155){\small $\beta_{n-1}$}
\put(36,155){\small $\beta$}
\put(47,155){\small $\gamma$}
\put(58,155){\small $\beta_1$}
\put(-3,139){\tiny $y_1$}
\put(26,139){\tiny $x_2$}
\put(13,139){\tiny $u_2$}
\put(-2,78){\tiny $x_1$}
\put(15,78){\tiny $u_1$}
\put(28,78){\tiny $y_2$}
\put(11,105){$r$}
\put(28,105){$p'$}
\put(145,0){\includegraphics{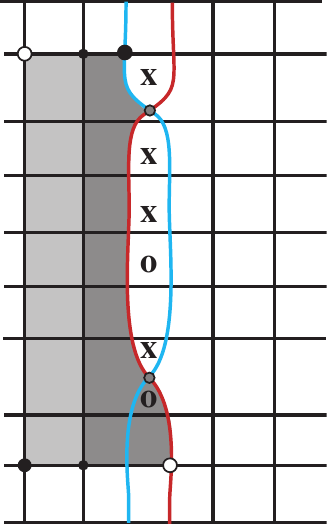}}
\put(160,155){\small $\beta_{n-1}$}
\put(181,155){\small $\beta$}
\put(192,155){\small $\gamma$}
\put(203,155){\small $\beta_1$}
\put(141,139){\tiny $y_1$}
\put(171,139){\tiny $x_2$}
\put(158,139){\tiny $u_2$}
\put(143,11){\tiny $x_1$}
\put(160,11){\tiny $u_1$}
\put(185,11){\tiny $y_2$}
\put(156,105){$r$}
\put(173,105){$p'$}
\put(48,118){\tiny $a$}
\put(48,40){\tiny $b$}
\put(193,118){\tiny $a$}
\put(193,40){\tiny $b$}
\end{picture}
\caption{Decomposing $p$ into a narrow pentagon $p'$ and a rectangle $r$}
\label{fig:pent12}
\end{figure}
\end{center}

\begin{center}
\begin{figure}[htpb]
\begin{picture}(165, 170)
\put(0,0){	\includegraphics{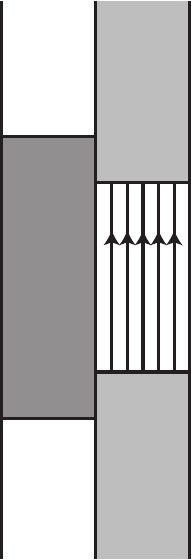}}
\put(0,164){\small $\beta_{n-1}$}
\put(27,164){\small $\beta$}
\put(55,164){\small $\beta_1$}
\put(-20,122){\small $\alpha_{k+1}$}
\put(60,107){\small $\alpha_k$}
\put(60,53){\small $\alpha_{l+1}$}
\put(-8,40){\small $\alpha_l$}
\put(13,80){$A$}
\put(39,15){$B$}
\put(113,0){\includegraphics{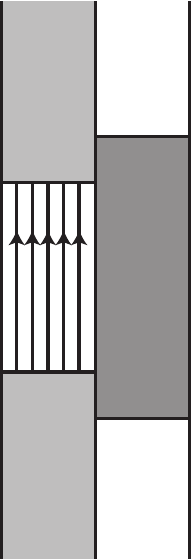}}
\put(110,164){\small $\beta_{n-1}$}
\put(137,164){\small $\gamma$}
\put(165,164){\small $\beta_1$}
\put(170,122){\small $\alpha_{k+1}$}
\put(101,107){\small $\alpha_k$}
\put(93,53){\small $\alpha_{l+1}$}
\put(170,40){\small $\alpha_l$}
\put(-8,0){\small $\alpha_1$}
\put(101,0){\small $\alpha_1$}
\put(150,80){$A$}
\put(122,15){$B$}
\put(137, -10){\small $\bar{g}$}
\put(27, -10){\small $g$}
\end{picture}
\caption{A commutation$'$ move.  Regions $A$ and $B$ contain $X$'s and $O$'s.}
\label{fig:comm}
\end{figure}
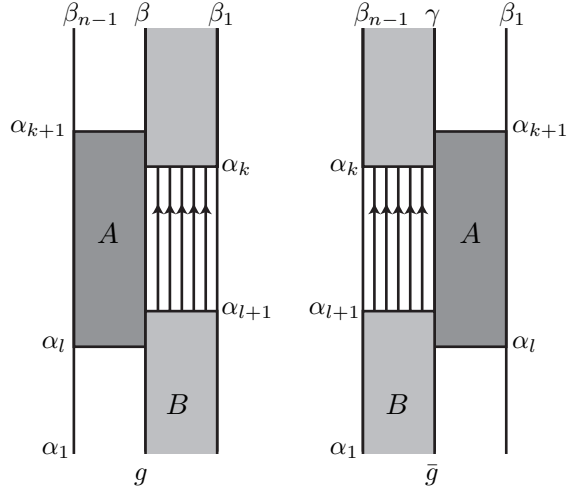
\end{center}

In order to compute $h^g$ or $h^{\bar{g}}$, draw the transverse spatial graph for $g$ and $g'$ so that the horizontal and vertical arcs connecting the $X$'s and $O$'s are inside the grid and consider their projections $pr(g)$ and $pr(g')$.  We will think of column $n-1$ as the column with $\beta_{n-1}$ on the left and column $n$ as the column with $\beta$ or $\gamma$ on the left.  Assume that $LS_1$ lies inside column $n-1$ and $LS_2$ lies in column $n$.  Let $A$ be the union of rectangles containing $LS_1$ and $B$ be the union of rectangles containing $LS_2$ (recall, we are assuming the ends of $LS_i$ do not lie on an $\alpha$ curve). Then the $X$'s and $O$'s columns $n-1$ and $n$ are contained in $A \cup B$ and projections of $A$ and $B$ intersect in the two rows that contain $a$ and $b$.  Since there are no $X$'s or $O$'s in $col_n \smallsetminus B$, the vertical arcs in $pr(g) \cap (col_n \smallsetminus B)$ form a collection of parallel arcs starting and stopping at $\alpha_k$ and $\alpha_{k+1}$ which are all oriented in the same direction (all upwards or all downwards).  In addition, $pr(g) \cap (col_{n-1} \smallsetminus A)$ is empty. See Figure~\ref{fig:comm}.  

Let $O'$ be the $O$ in column $n-1$ of $g$ and let $O''$ be the $O$ in column $n-1$ of $\bar{g}$.  Note that $p'$ either contains $O'$ or $O''$ or both or is empty (in particular, it never contains an element of $\mathbb{X}$).  
We first consider $h^g(x_2)-h^g(u_2)$. There are three cases. First suppose $x_2 \in \alpha_i$ for $i\geq k+1$ or $i\leq l$.  Then $h^g(x_2)-h^g(u_2)=0$ since the region above and below $A$ in $g$ contains no vertical arcs in $pr(g)$.  In addition, we see that $p'$ cannot contain $O'$ so that $h^g(x_2)-h^g(u_2)=0=O'(p')w(O')$ where by $O'(p')$, we mean the number of $O'$s in $p'$. Now suppose $x_2 \in \alpha_i$ for $l+1 \leq i \leq k$.  If $p'$ does not contain $O'$ then $h^g(x_2)-h^g(u_2)=0=O'(p')w(O')$.  If $p'$ contains $O'$ then the arc going from $u_2$ to $x_2$ crosses all the vertical strands emanating from this $O$, all oriented downwards, since $p'$ does not contain any $X$'s.  Thus, $h^g(x_2)-h^g(u_2)=O'(p')w(O')$.  Thus, in all cases, we see that $h^g(x_2)-h^g(u_2)=O'(p')w(O')$.

Now consider $h^{\bar{g}}(y_2)-h^{\bar{g}}(u_1)$.  We again have three cases to consider.  First suppose that $y_2 \in \alpha_j$ for $l+1 \leq j \leq k$.  
Then the arc from $u_1$ to $y_2$ crosses $m$ vertical arcs, all oriented in the same direction so $h^{\bar{g}}(y_2)-h^{\bar{g}}(u_1)=\eta$ for some $\eta \in H_1(E(f))$.  
In addition, $O''(p)=0$ so that $h^{\bar{g}}(y_2)-h^{\bar{g}}(u_1)=\eta - O''(p')w(p')$.  
Now suppose that $x_2 \in \alpha_i$ for $i\geq k+1$ or $i\leq l$.  If $p'$ does not contain $O''$ then $h^{\bar{g}}(y_2)-h^{\bar{g}}(1_2)=\eta=\eta - O''(p')w(p')$.  If $p'$ contains $O''$, $h^{\bar{g}}(y_2)-h^{\bar{g}}(u_1) = \eta-w(O'')$.  Thus, in all cases, $h^{\bar{g}}(y_2)-h^{\bar{g}}(u_1)=\eta - O''(p')w(p')$.

Putting this together and using that $h^{\bar{g}} (y_1) = h^g (y_1)$ and $h^{\bar{g}} (u_1) = h^g (u_1)$, we have 
\begin{eqnarray*} h^g(x_1) + h^g(x_2) - h^{\bar{g}} (y_1) - h^{\bar{g}} (y_2) &=& (h^g(x_1) + h^g(u_2)) - (h^g(y_1) + h^g(u_1))\\
	&\hspace{5pt}& + (h^g(x_2) - h^g(u_2)) - (h^{\bar{g}}(y_2) - h^{\bar{g}}(u_1)) \\ 
	&=& \sum_{i=1}^{n}w(O_i)O_i(r) + O'(p')w(O') -	(\eta - O''(p')w(p')) \\
	&=& \sum_{i=1}^{n}w(O_i)O_i(p) - \eta.
\end{eqnarray*} 
Thus, (\ref{comm_eq}) holds with $\delta(g,\bar{g})= -\eta$ which completes the proof that $\Phi_{\beta\gamma}$ is a graded map (with respect to the Alexander grading). 
\vspace{12pt}

\noindent\textit{$\Phi_{\beta\gamma}$ is a chain map}: 
The remaining portion of the proof that $\Phi_{\beta\gamma}$ is a chain map, follows almost immediately from the proof of Lemma 3.1 in \cite{MOST}.  However, our pentagons and rectangles cannot count $X$'s so we need to be a little more careful.  

For $\x \in \S(g)$, there is a unique element $c(\x)$ of $\S(\bar{g})$, called the canonical closest generator of $\x$, defined as follows.  Let $t$ be such that $\x \cap \beta \in \alpha_t$ and set $x'$ be the point in $\alpha_t \cap \gamma$. Define $$c(\x):= \{x_i \in \x | x_i \notin \beta\} \cup \{x'\} \in \S(\bar{g}).$$  

Suppose $D$ is a domain of the form $p\ast r$ representing a term in $\partial^- \circ \Phi_{\beta,\gamma}(\x)$ and $D$ connects $\x$ to $\y$ with $\y \neq c(\x)$.  By this, we mean to consider the juxtaposition of the pentagon $p$ connecting $\x$ to $\z$ and the rectangle $r$ connecting $\z$ to $\y$, in the combined grid diagram.  Note that the domain does not contain any element of $\mathbb{X}$.  
Then there is exactly one other empty rectangle $r'$ (in $g$ or $\bar{g}$) and empty pentagon $p'$ such that $r' \ast p'$ or $p'\ast r'$  gives a decomposition of $D$.  Note that most of the time the other decomposition of the form $r' \ast p'$.  To see this, one just needs to draw every possible domain of the form $p \ast r$ and $r \ast p$ where $r$ is an empty rectangle (in $g$ or $\bar{g}$) and $p$ is an empty pentagon.  In addition since $p'$ and $r'$ will be contained in $D$, they will not contain any element of $\mathbb{X}$ so $r' \ast p'$ or $p'\ast r'$ will represent an element of $\Phi_{\beta,\gamma} \circ \partial^-$ or $\partial^- \circ \Phi_{\beta,\gamma}$.  The same statement is true if you start with a domain $D$ of the form $r\ast p$ representing a term in $\partial^- \circ \Phi_{\beta,\gamma}(\x)$ as long as $D$ connects $\x$ to $\y$ with $\y \neq c(\x)$.

Suppose $D$ is a domain of the form $p\ast r$ representing a term in $\partial^- \circ \Phi_{\beta,\gamma}(\x)$ and $D$ connects $\x$ to $c(\x)$.  Then $D$ consists of two regions $C$ and $E$ where $C$ is the region that lies to the left of both $\gamma$ and $\beta$ and to the right of $\beta_{n-1}$ and $E$ is a subset of the region that lies to the right of $\beta$ and to the left of $\gamma$.  Note that the $C \cap \mathbb{O} = C \cap \mathbb{X} = E \cap \mathbb{X} = \emptyset$.  There is exactly one other domain $D'$ connecting $\x$ to $c(\x)$ of the form $p' \ast r'$ or $r' \ast p'$.  See Figure~\ref{fig:closest_neighbor} for an example.  
\begin{center}
\begin{figure}[htpb]
\begin{picture}(222,122)
\put(0,10){\includegraphics{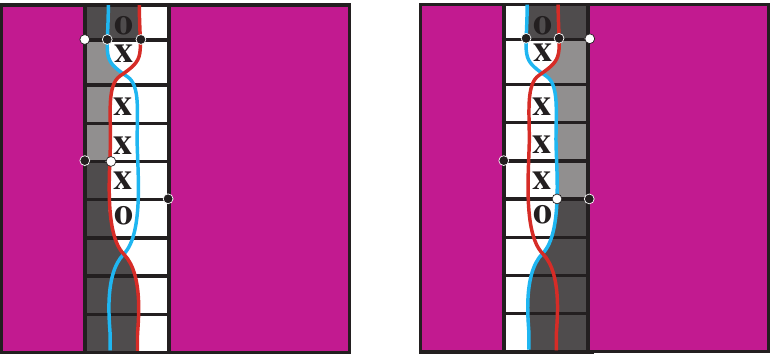}}
\put(35,0){$p \ast r$}
\put(155,0){$r' \ast p'$}
\put(28,115){\small $\beta$}
\put(38,115){\small $\gamma$}
\end{picture}
\caption{An example of two domains connecting $\x$ to $c(\x)$.  The generators $\x$ and $c(\x)$ are both in black circles; they only disagree on one row.}
\label{fig:closest_neighbor}
\end{figure}
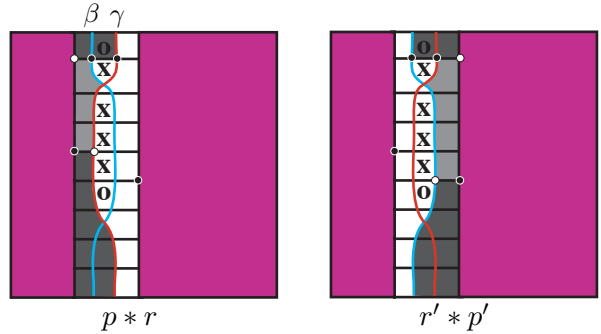
\end{center}
 This domain consists of a region $C'$ and $E$ where $C'$ is the region that lies to the right of both $\gamma$ and $\beta$ and to the left of $\beta_{1}$.   Note that $C' \cap \mathbb{X} = C' \cap \mathbb{X}=\emptyset$ and so $O_i(D)=O_i(D')$ and $D' \cap \mathbb{X}=\emptyset$.  Thus $D'$ represents an element of $\Phi_{\beta,\gamma} \circ \partial^-(\x)$ or $\partial^- \circ \Phi_{\beta,\gamma}(\x)$.  A similar statement holds for domain of the form $p\ast r$ representing a term in $\partial^- \circ \Phi_{\beta,\gamma}(\x)$ and connecting $\x$ to $c(\x)$.  Thus, every term is canceled by another.  So $\partial^-\circ\Phi_{\beta\gamma}(\x)=\Phi_{\beta\gamma}\circ\partial^-(\x)$.
\end{proof}

In order to prove that $\Phi_{\beta\gamma}$ is a chain homotopy equivalence we define a similar map $H_{\beta\gamma\beta}$ which counts hexagons in the combined grid diagram.  These hexagons are just like the ones in \cite{MOST} except they don't contain elements of $\mathbb{X}$.  We recall the definition .
For $\x,\y\in \S(g)$, let $\Hex_{\beta\gamma\beta}(\x,\y)$ be the set of embedded hexagons defined as follows.  If $\x$ and $\y$ don't coincide at $n-2$ points then $\Hex_{\beta\gamma\beta}(\x,\y)$ is the empty set.    
Suppose that $\x$ and $\y$ coincide at $n-2$ points (say $x_3=y_3, \dots, x_n=y_n$).  Without loss of generality, let $x_2 = \x \cap \beta$ and $y_2 = \y \cap \gamma$.  An element $\mathcal{H} \in \Hex_{\beta\gamma\beta}(\x,\y)$ is an embedded disk in the combined grid diagram, whose boundary consists of six arcs, each of which are contained in the circles $\beta_i$, $\alpha_i$, $\beta$, or $\gamma$ and satisfies the following conditions.  The intersections of the arcs lie on the points of $x_1,x_2,y_1,y_2,a,$ and $b$.  
Moreover, start at the point in $x_2$ and transverse the boundary of $\mathcal{H}$, using the orientation given by $\mathcal{H}$.  The condition to be in $\Hex_{\beta \gamma \beta}(\x,\y)$ is that you will first travel along a horizontal circle, meet $y_1$, proceed along a vertical circle $\beta_i$, meet $x_1$, continue along another horizontal circle, meet $y_2$, proceed though an arc in $\beta$ until you meet $b$, then travel along an arc in $\gamma$ until you hit $a$ and finally travel along an arc in $\beta$ until arriving back at $x_2$. Finally, all angles are required to be less than straight.   For $\x,\y\in \S(\bar{g})$, there is a corresponding set of hexagons $\Hex_{\gamma\beta\gamma}(\x,\y).$ 
The set of empty pentagons $\Hex^o_{\beta\gamma\beta}$ are those hexagons $q\in \Hex_{\beta\gamma\beta}(\x,\y)$ where $\x \cap \Interior(q) =\emptyset$.
Define 
$H_{\beta\gamma\beta}\colon C^-(g) \longrightarrow C^-(g)$ by
\[
  H_{\beta\gamma\beta}(\x) = \sum_{\y\in \S(g)}\,
    \sum_{\substack{q\in\Hex^o_{\beta\gamma\beta}(\x,\y)\\Int(q)\cap\mathbb{X}=\emptyset}}\!
    U_1^{O_1(q)}\cdots U_n^{O_n(q)} \cm \y.
\]

\begin{prop}
  \label{prop_HmapComm'}
  The map $\Phi_{\beta\gamma}\colon C^-(g)\longrightarrow C^-(\bar{g})$ is a chain homotopy equivalence.  
 \end{prop}

\begin{proof} To prove this, we show that  
$$
    Id+\Phi_{\gamma\beta}\circ \Phi_{\beta\gamma} +  \partial^- \circ H_{\beta\gamma\beta} + H_{\beta\gamma\beta}\circ \partial^-  = 0
$$
and 
$$    Id+\Phi_{\beta\gamma}\circ \Phi_{\gamma\beta} + \partial^- \circ H_{\gamma\beta\gamma} + H_{\gamma\beta\gamma}\circ \partial^- =0.
$$
The proof is similar to the proof of Proposition 3.2 of \cite{MOST}.  Let $\x\in S(g)$.  Typically, every domain that arises as the composition of two empty pentagons or an empty hexagon and an empty rectangle representing terms from $\Phi_{\gamma\beta}\circ \Phi_{\beta\gamma}(\x)$, $\partial^- \circ H_{\beta\gamma\beta}(\x)$, or $H_{\beta\gamma\beta}\circ \partial^-(\x)$ can be decomposed in exactly two ways representing terms from $\Phi_{\gamma\beta}\circ \Phi_{\beta\gamma}(\x)$, $\partial^- \circ H_{\beta\gamma\beta}(\x)$, or $H_{\beta\gamma\beta}\circ \partial^-(\x)$.  The only case when this does not happen is when the domain connects $\x$ to $\x$.  In this case, the domain consists of the region that is either (1) to the left of both $\gamma$ and $\beta$ and to the right of $\beta_{n-1}$ or (2) to the right of both $\gamma$ and $\beta$ and to the left of $\beta_1$.  Such a domain can be decomposed in three ways representing terms in $\Phi_{\gamma\beta}\circ \Phi_{\beta\gamma}(\x)$, $\partial^- \circ H_{\beta\gamma\beta}(\x)$, or $H_{\beta\gamma\beta}\circ \partial^-(\x)$.  The other case follows similarly. 
 
\end{proof}

\begin{remark}$H_{\beta\gamma\beta}$ is a $(H_1(E(f)),\Z)$ bigraded $R_V$-module homomorphism of degree $(0,1)$.  Therefore $(C(g),\partial)$ and $(C(\bar{g}),\partial)$ are chain homotopy equivalent as bigraded $R_V$-module chain complexes.  
\end{remark}

\vspace{12pt}

\subsection{Stabilization$'$ Invariance}\label{stab_inv}

\begin{proposition}\label{prop:stab}Suppose $g$ and $\bar{g}$ are saturated graph grid diagrams that differ by a stabilization$'$ move.  Let $f:G \rightarrow S^3$ be the transverse graph associated to $g$ and $\bar{g}$ and $V$ be the number of vertices of $G$.  Then there is an $(H_1(E(f)),\Z)$ bigraded $R_V$-module quasi-isomorphism $(C(g),\partial^-_g) \rightarrow (C(\bar{g}),\partial^-_{\bar{g}})$  of degree $(\delta(g,\bar{g}),0)$ for some $\delta(g,\bar{g}) \in H_1(E(f))$.  
\end{proposition}

The proof of Proposition~\ref{prop:stab} will take up the rest of this subsection. The proof of stabilization$'$ is similar to the proof of stabilization in \cite{MOST}.  However, because of the fact that we only have a graded theory instead of a filtered theory, the proof becomes drastically simplified.  We also fill in some of the details and clarify some of the arguments in the proof of \cite{MOST}.  We will only prove the case for row stabilization$'$.  The proof of column stabilization$'$ is similar. 

Let $g$ be a graph grid diagram and $\bar{g}$ be obtained from $g$ by a row stabilization$'$ move.  
An example of a row stabilization$'$ move is shown in Figure~\ref{stab_prime2}.  
To get $\bar{g}$, we take some row $\mathbf{R}$ in $g$ with $l$ $X$'s, delete it and replace it with two new rows and then add a new column.  We place $O_k, X_{j_2}, \dots, X_{j_l}$ into one of the new rows (and in the same columns as before) and $X_{j_1}$ into the other new row (and in the same column as before).  We place decorations $O_{n}$ and $X_{m}$ into the new column so that $O_{n}$ occupies the same row as $X_{j_1}$ and $X_{m}$ occupies the same row as $O_k$. 
By Remark~\ref{rem:stab_simple}, we may assume that $X_{j_1}, X_{m}, O_{n}$ share a corner, called $\star$, where $X_{j_1}$ is directly to the left of $O_{n}$, and $O_n$ is directly above $X_m$ (see Figure~\ref{stab_prime2}).  Let $\beta_n$ be the vertical grid circle directly to the left of $O_n$ and let $\beta_1$ be the vertical grid circle directly to the right of $O_n$.  Let $\alpha_n$ be the horizontal grid circle to between $O_n$ and $X_m$.

\begin{figure}[htpb!]
\begin{center}
\begin{picture}(391,167)
\put(0,0){\includegraphics{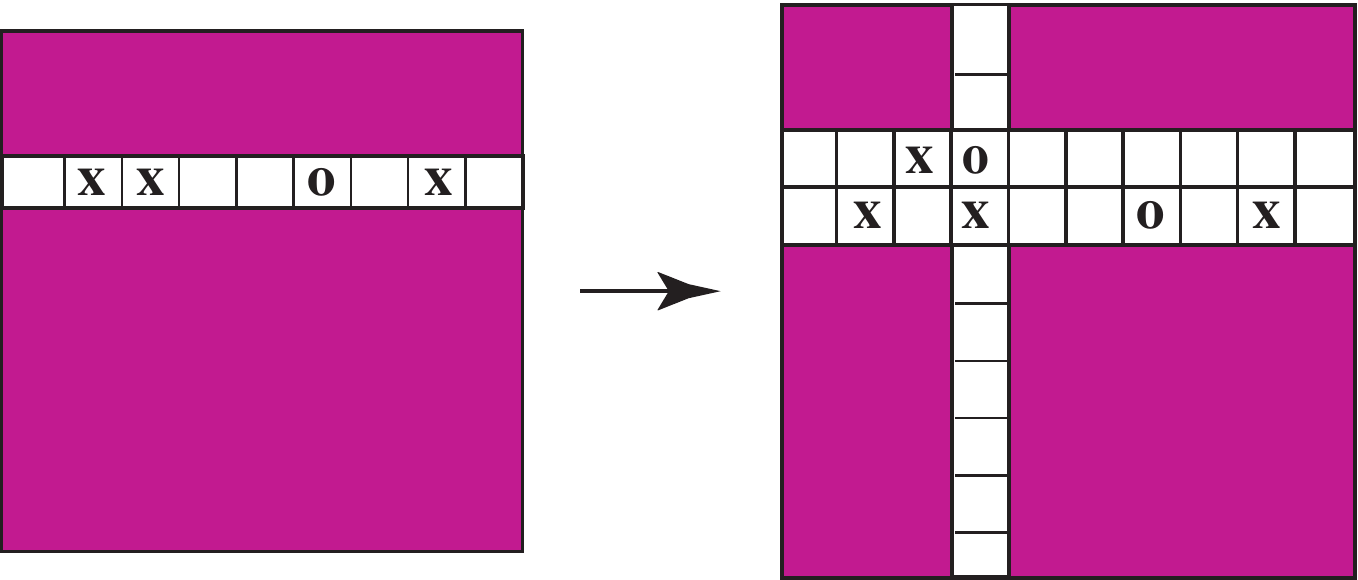}}
\put(30,110){\tiny $1$}
\put(47,110){\tiny $2$}
\put(96,110){\tiny $k$}
\put(130,110){\tiny $3$}
\put(254,100){\tiny $1$}
\put(270,115){\tiny $2$}
\put(273,111){$\star$}
\put(285,115){\tiny $n$}
\put(285,100){\tiny $m$}
\put(336,100){\tiny $k$}
\put(370,100){\tiny $3$}
\put(210,111){\small $\alpha_n$}
\put(271,-10){\small $\beta_n$}
\put(286,-10){\small $\beta_1$}
\put(-13,112){$\mathbf{R}$}
\put(96,118){\small $\ast$}
\put(336,108){\small $\ast$}
\end{picture}
\caption{An example of stabilization$'$.}\label{stab_prime2}
\end{center}
\end{figure}

Let $(B,\partial_B)=(C^-(g),\partial^-_g)$ and  $(C,\partial_C)=(C^-(\bar{g}),\partial^-_{\bar{g}})$.  Let $(B[U_n],\partial_B)$ be the chain complex obtain as follows.  $B[U_n]$ is the free (left) $R_n$-module generated by $\mathbf{S}(g)$ and $\partial_B$ is the unique extension of $\partial_g$ to $B[U_n]$ so that $\partial_B$ is an $R_n$-module homomorphism.  We note that $(B[U_n],\partial_B)$ is isomorphic to the chain complex whose group is $B \otimes_{R_{n-1}} R_n$ and whose boundary map is $\partial_B \otimes id$.  $(B[U_n],\partial_B)$ becomes an $(H_1(E(f)),\Z)$ bigraded $R_n$-module chain complex by setting the $H_1(E(f))$ grading of $U_n$ to be $w(X_{j_1})$ and the $\Z$ grading of $U_n$ to be $-2$.

\begin{definition} 
	Let 
	$$\zeta = 
	\begin{cases}
	        U_n + U_k & \text{if } l=1\\
	        U_n & \text{if } l\geq 2
	 \end{cases}$$ 
	then $\zeta: B[U_n] \rightarrow B[U_n]$ is a bigraded $R_n$-module chain map of degree $(-w(O_n),-2)$.  Let $(C',\partial')$ be the mapping cone complex of $\zeta$.  	
	Since $(B[U_n],\partial_B)$ is an $\HZ$ bigraded $R_n$-module chain complexes, the $(cone(\zeta),\partial')$ is an $\HZ$ bigraded $R_n$-module chain complex.  See p.~\pageref{def:mc} for the definiton of the mapping cone and its grading. 
\end{definition}

We will first show that $C'$ is quasi-isomorphic to $B$ and then we will show that $C$ is quasi-isomorphic to $C'$.  The first step follows from basic facts from homological algebra.  Consider the cokernel of $\zeta$, $B[U_n]/im(\zeta)$.  The $(H_1(E(f)),\mathbb{Z})$ bigrading on $C^-(g)$ descends to a well-defined $(H_1(E(f)),\mathbb{Z})$ bigrading on $B[U_n]/im(\zeta)$ and so $(B[U_n]/im(\zeta),\partial_B)$ is an $\HZ$ bigraded $R_n$-module chain complex.  In addition, using the inclusion $R_{n-1} \subset R_n$, it is also naturally an $\HZ$ bigraded $R_{n-1}$-module chain complex.  Moreover the map $$B \xrightarrow{\cong} B[U_n]/im(\zeta),$$ which sends $b\in B$ to the equivalence class of itself, is a bigraded $R_{n-1}$-module chain isomorphism of degree $(0,0)$.  Thus, we just need to show that $C'$ is quasi-isomorphic to $B/im(\zeta)$.  

\begin{lemma}\label{lem:weib} Let $pr: B[U_n] \rightarrow B[U_n]/im(\zeta)$ be the quotient map.  The map from $C'$ to $B[U_n]/im(\zeta)$ that sends $(a,b)$ to $pr(b)$ is a bigraded $R_n$-module quasi-isomorphism of degree $(0,0)$.
\end{lemma}

\begin{proof}There is a short exact sequence of chain complexes
	
	$$ 0 \rightarrow B[U_n] \xrightarrow{\zeta} B[U_n] \xrightarrow{pr} B[U_n]/im(\zeta) \rightarrow 0.$$
	Therefore, by 1.5.8 in \cite{Weibel}, the map $cone(\zeta) \rightarrow B[U_n]/im(\zeta)$ sending $(a,b)$ to $pr(b)$ is a quasi-isomorphism.   This map is bigraded of degree $(0,0)$ and an $R_n$-module homomorphism. 
\end{proof}

We now define a quasi-isomorphism $F:C \rightarrow C'$ similar to the one in \cite{MOST}.  However, since we are only considering a bigraded chain complex instead of a filtered chain complex and we only need to consider one of the four stabilization$'$s, our map becomes very simple. We first consider some notation.  

Let $\mathbf{I} \subset \S(\bar{g})$ be the set of $\x \in \S(\bar{g})$ that contain $\star$ the intersection of the new grid lines/circles $\alpha_n$ and $\beta_n$.  
There is a natural 1-1 correspondence between $\mathbf{I}$ and $\S(g)$.  For $\x\in \S(g)$ let $\psi(\x)$ be the point in $\mathbf{I}$ defined by $\x \cup \star$. Note that if $\x \in \mathbf{I}$ then $\x = \{x_1, x_2, \dots, x_{n-1}, \star \}$ so $\psi^{-1}(\x) = \{x_1,\dots, x_{n-1}\}$ is the generator of $B$ obtained by removing $\star$. The gradings of $\x \in \S(g)$ and $\psi(\x)$ are related as follows.

\begin{align}
M^g(\x) &= M^{C'}(\x,0) + 1 = M^{C'}(0,\x) = M^{\bar{g}}(\psi(\x)) + 1  \\ 	
A^g(\x) &=  A^{C'}(\x,0) + w(O_n) = A^{C'}(0,\x) = A^{\bar{g}}(\psi(\x)) - A^{\bar{g}}(\star) \label{eq:Alex}
\end{align}

Define $F_L: C \rightarrow B[U_n]$ by
$$F_L(\x) = \begin{cases} 0 & \text{if } \x \not\in \mathbf{I} \\ \psi^{-1}(\x) &\text{if } \x \in \mathbf{I} \end{cases}
$$
and extend it to $C$ so that it is an $R_n$-module homomorphism. 
The reason for this choice of map is that the trivial region is the only type $L$ region in \cite{MOST} that doesn't contain $X_{j_1}$ when $X_{j_1}$ is directly to the left of $O_n$.  
For the definition of $L$ type region see Definition 3.4 and Figure 13 of \cite{MOST}.  

There is one type $R$ region that doesn't contain $X_{j_1}$, the rectangle with upper left corner $\star$.  For $\x \in S(\bar{g})$ and $\y\in S(g)$, let 
$\pi_R(\x,\psi(\y))$ be the set of $p \in \Rect^\circ(\x,\psi(\y))$ whose upper left corner is $\star$ (see Figure~\ref{fig:stab_domain}).  We will call such a domain an \textbf{$R$-domain}. 
\begin{center}
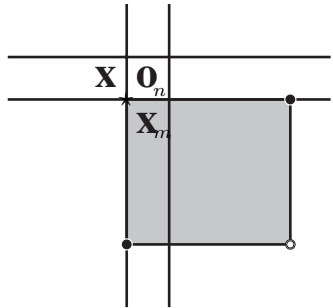
\begin{figure}[htpb]
\begin{picture}(122,115)
\put(0,0){\includegraphics{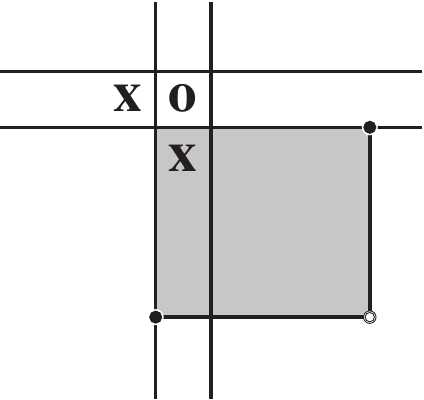}}
\put(42,76){$\star$}
\put(55,80){\tiny $n$}
\put(55,64){\tiny $m$}
\end{picture}
\caption{A domain in $\pi_R(\x,\psi(\y))$.  The points of $\x$ are in black, the points of $\y$ are in white. The domain is shaded. }
\label{fig:stab_domain}
\end{figure}
\end{center}
Define $F_R : C \rightarrow B[U_n]$ by 
$$F_R(\x) = \sum_{\y \in S(g)} \sum_{\substack{p \in \pi_R(\x,\psi(\y)) \\ (\mathbb{X}\smallsetminus X_m) \cap Int(p) = \emptyset}} U_1^{O_1(p)} \cdots U_{n-1}^{O_{n-1}(p)} \y.
$$
Note that we are counting domains in $\bar{g}$ that cannot contain $X_1, \dots, X_{m-1}$ but can contain $X_m$.   Also, there is no factor of $U_n$ in the terms of $F_R(\x)$. Using these, we define $F: C \rightarrow C'$ defined by $$F(\x)= ( F_L(\x), F_R(\x) ).$$

\begin{lemma}$F:C \rightarrow C'$ is a $\HZ$ bigraded $R_n$-module chain map of degree $(- w(X_m) - A^{\bar{g}}(\star),0)$. 
\end{lemma}

\begin{proof} $F$ is an $R_n$-module homomorphism by definition.  It follows from \cite[Lemma 3.5]{MOST} that the Maslov grading is preserved. 
	 
	\vspace{12pt}
	\noindent\textit{$F$ is an $H_1(E(f))$ graded map of degree $-A^{\bar{g}}(\star) - w(O_n)$}:  Let $\x\in \S(\bar{g})$.  We first consider the grading of $(F_L(\x),0)$.  $F_L(\x)$ is either $0$ or is $\psi^{-1}(\x)$.  Suppose  $F_L(\x)\neq 0$ then $F_L(\x)=\psi^{-1}(\x)$ so by (\ref{eq:Alex}),
	$$A^{C'}(F_L(\x),0) = A^g(\psi^{-1}(\x)) - w(O_n) = A^{\bar{g}}(\x) - A^{\bar{g}}(\star) - w(O_n).$$  
	We now consider the grading of $(0,F_R(\x))$.  Let $U_1^{O_1(p)} \cdots U_{n-1}^{O_{n-1}(p)} \y$ be a term in $F_R(\x)$.  
	Then $p$ is a rectangle connecting $\x$ to $\psi(y)$ that doesn't contain any elements of $\mathbb{X}$ except $X_m$.  Moreover, $p$ does not contain $O_n$ and contains $X_m$ exactly once.  So by Lemma~ \ref{lem:alex_count},  
	$$A^{\bar{g}}(\x)-A^{\bar{g}}(\psi(\y)) = \mathbf{n}_\mathbb{X}(p)-\mathbf{n}_\mathbb{O}(p) = w(X_m) - \sum_{i=1}^{n-1}O_i(p)w(O_i).$$  
	Therefore, by (\ref{eq:Alex}), we have 
	
\begin{align*}A^{C'}(0,U_1^{O_1(p)} \cdots U_{n-1}^{O_{n-1}(p)} \y) &= A^{C'}(0,\y) - \sum_{i=1}^{n-1}O_i(p)w(O_i) \\
	&= A^{\bar{g}}(\psi(\y)) - A^{\bar{g}}(\star) -\sum_{i=1}^{n-1}O_i(p)w(O_i) \\
	&= A^{\bar{g}}(\x) - w(X_m) - A^{\bar{g}}(\star) \\
	&= A^{\bar{g}}(\x) - w(O_n) - A^{\bar{g}}(\star).
\end{align*}

\noindent\textit{$F$ is a chain map}: Let $\x \in \S(\bar{g})$.  Recall that 
$$\partial'(F(\x)) = (\partial_B(F_L(\x)),0) + (0, \zeta(\x))  + (0,\partial_B(F_R(\x)) $$ and
$$F(\partial_C(\x)) = (F_L(\partial_C(\x)), 0) + (0, F_R(\partial_C(\x)).$$
The proof that $F$ is a chain map is similar to the proof that the commutation$'$ map $\Phi_{\beta,\gamma}$ is a chain map.  

For this proof, whenever we have an empty rectangle $r$ in $g$ (contributing to a term in $\partial_B$), we will view it as living in $\bar{g}$ in the obvious way.  Note that such a rectangle cannot have a boundary on $\alpha_n$ or $\beta_n$ hence cannot have $\star$ on its boundary.  Usually (in a filtered theory) such a rectangle could contain the point $\star$ in its interior.  However, since $r$ cannot contain $X_{j_1}$, it also cannot contain $\star$ in its interior.  

We first show that $\partial_B(F_L(\x))=F_L(\partial_C(\x)).$
Suppose $D$ is a domain of the form $p\ast r$ representing a non-trivial term in $\partial_B \circ F_L(\x)$.  Then $p$ is a trivial domain connecting a point in $\mathbf{I}$ to itself and $r$ is an empty rectangle in $g$.  We will think of $p$ as a point at $\star$.  Since $\star$ cannot be on the corner of $r$ or in its interior, $r$ and $p$ must be disjoint.  Suppose $D$ is a domain of the form $r'\ast p'$ representing a term in $F_L \circ \partial_C(\x)$.  Then $p'$ is a trivial domain connecting a point in $\mathbf{I}$ to itself and $r$ is an empty rectangle in $\bar{g}$.  Since $r'$ is empty, $\star$ is not in the interior of $p'$.  Since $r'$ cannot contain $X_m$ or $X_{j_1}$, it cannot have $\star$ as one of its corners.  Thus $r'$ and $p'$ are disjoint.  Therefore the terms in  $(\partial_B(F_L(\x)),0)$ and $(F_L(\partial_C(\x)), 0)$ cancel each other out. 

We now wish to show that $\partial_B(F_R(\x)) + \zeta(\x) = F_R(\partial_C(\x)$. Suppose $D$ is a domain of the form $p\ast r$ representing a non-trivial term in $\partial_B \circ F_R(\x)$.  Then $p$ is an empty rectangle in $\bar{g}$ with $\star$ on its upper left corner and $r$ is an empty rectangle in $g$.  If $p$ and $r$ share zero corners or share one corner then there is exactly one empty rectangle $r'$ in $\bar{g}$ and $R$-domain $p$ such that $r' \ast p'$ gives a decomposition of $D$ ($p'$ and $r'$ will also share zero corners or share one corner). These represent terms in $F_R\circ \partial_C(\x)$.  Note that $p$ and $r$ cannot share two or three corners.  Conversely, 
suppose $D$ is a domain of the form $p'\ast r'$ representing a non-trivial term in $F_R\circ \partial_C(\x)$.  Then $p'$ is an empty rectangle in $\bar{g}$ with $\star$ on its upper left corner and $r'$ is an empty rectangle in $\bar{g}$.  If $p'$ and $r'$ share no corners or share one corner then there is exactly one other empty rectangle $r''$ (in $g$ or $\bar{g}$) and $R$-domain $p''$ such that $r'' \ast p''$ or $p''\ast r''$  gives a decomposition of $D$.  Here $p''$ and $r''$ will also share zero corners or share one corner.  These will represent terms in $\partial_B \circ F_R(\x)$ or $F_R \circ \partial_C(\x)$.  Note that $p'$ and $r'$ cannot share two or three corners.  Thus, these terms cancel one another out. 

First note that if $D$ is a domain of the form $p\ast r$ representing a non-trivial term in $\partial_B \circ F_R(\x)$ then $p$ and $r$ cannot share more than one corner (hence cannot share four corners).  Suppose $D$ is a domain of the form $p'\ast r'$ representing a non-trivial term in $F_R\circ \partial_C(\x)$ where $p'$ and $r'$ share four corners.  Then $D$ is either the width one horizontal annulus containing $X_m$ or the width one vertical annulus containing $O_n$.  The width one vertical annulus always contributes $U_n \x$. If there is more than one element of $\mathbb{X}$ in row $\mathbf{R}$ ($l \geq 2$) then the width one horizontal annulus contains an element of $\mathbb{X}$ that is not $X_m$ so is not counted.  If there is exactly one element of $\mathbb{X}$ in row $\mathbf{R}$ ($l=1$) then the width one horizontal annulus containing $X_m$ contributes $U_k \x$.  Thus these terms cancel with $\zeta(\x)$.

\end{proof}

To show that $F$ is a quasi-isomorphism, we will first show that $\widetilde{F}: \widetilde{C} \rightarrow \widetilde{C}'$ is a quasi-isomorphism, where $\widetilde{C}$ is the quotient $C/\mathcal{U}_n$ defined in Subsection~\ref{sec:tilde_hat}.  
To do this, we introduce a filtration on $\widetilde{C}$ and $\widetilde{C}'$ so that $\widetilde{F}$ is a filtered map and show $\widetilde{F}$ induces a quasi-isomorphism on its associated graded object.  The rest of the proof will follow from the well known lemma.

\begin{lemma}[{\cite[Theorem 3.2]{McCleary}}]\label{mc} Suppose that $F:C \rightarrow C'$ is a filtered chain map that induces an isomorphism on the homology of the associated graded object.  Then $F$ is a filtered quasi-isomorphism. 
\end{lemma}

The definition of the filtration and the proof that $\widetilde{F}$ is induces a quasi-isomorphism on its associated graded object is essentially the same as in \cite{MOST}.  However, there is a small mistake in their definition of the filtration which we fix.  We also give more details which we believe clarifies their proof. 

For any $\mathbb{F}[U_1, \dots, U_n]$-module chain complex $(C,\partial)$ one can consider define the chain complex $(\widetilde{C},\widetilde{\partial})$ like we did in Subsection~\ref{sec:tilde_hat}.   Let $\mathcal{U}_V$ be the $\mathbb{F}$-vector subspace of $C$ spanned by $U_1 C \cup \dots \cup U_n C$.  Define $\widetilde{C}$ to be the quotient $C/\mathcal{U}_n$.  Since $\partial^-(\mathcal{U}_n) \subset \mathcal{U}_n$, it follows that $\partial$ descends to an linear map $$\widetilde{\partial}: \widetilde{C} \rightarrow \widetilde{C}$$ of vector spaces over $\mathbb{F}$.  

Define the $Q$-filtration, $\mathcal{F}^Q_k(C)$ on $(C,\partial_C)$ (where $C=C^-(g)$) as follows.  Let $Q$ be the collection of $(n-1)^2$ dots in $\bar{g}$, with one dot placed in each square which does not appear in the row or column containing $O_n$.   For a domain $p\in \pi(\x,\y)$, let $\mathbb{O}(p)$  be the total number of $O$'s in $p$ counted with sign.  That is, $\mathbb{O}(p) = \sum_{i=1}^n O_i(p)$. Similarly, we define $Q(p)$ to  be the total number of dotes in $p$ counted with sign.  Here, we are viewing the points in $\mathbb{O}$ and $Q$ as having positive orientation.  Note with this convention, if $r$ is a rectangle connecting $\x$ to $\y$, then $Q(r)\geq 0$ and $\mathbb{O}(r) \geq 0.$    
\begin{lemma} \label{lem:wd} Let $p$ and $p'$ be domains in $\bar{g}$ connecting $\x$ to $\y$ such that 
	$$O_n(p) = O_n(p') = 0 = \mathbb{O}(p) = \mathbb{O}(p').$$  Then $Q(p) = Q(p')$.
\end{lemma}
\begin{proof} For $1\leq i \leq n$, let $\mathbf{R}_i \in \pi(\x,\x)$ (respectively $\mathbf{C}_i$) be the domain that is the positively oriented row (respectively column) contain $O_i$.  Note that $O_k(\mathbf{R}_i)=O_k(\mathbf{C}_i) = \delta_{ik}$.  Suppose $p,p' \in \pi(\x,\y)$ then $p$ and $p'$ differ by a domain in $\pi(\x,\x)$ so that 
	
\begin{equation}\label{ppprime}p' = p + \sum_{i=1}^n (a_i \mathbf{R}_i + b_i \mathbf{C}_i).\end{equation}  This follows from the fact that the space of domains on the torus of the form $\pi(\x,\x)$ is generated by $\mathbf{R}_i$ and $\mathbf{C}_i$ with the relation that $\sum_{i=1}^n (\mathbf{R}_i - \mathbf{C}_i) =0.$
	
Now, we note that 
$Q(\mathbf{R}_i) = Q(\mathbf{C}_i) = n-1$ for $i\neq n$ and $Q(\mathbf{R}_n) = Q(\mathbf{C}_n) = 0.$  Thus 
$$Q(p') = Q(p) + (n-1) \sum_{i=1}^{n-1} (a_i+b_i).$$
Since $O_n(p)=O_n(p')=0$ by hypothesis, using (\ref{ppprime}) we get 
$$0 = O_n(p') = O_n(p) + O_n\left(\sum_{i=1}^n (a_i \mathbf{R}_i + b_i \mathbf{C}_i)\right) = a_n + b_n.$$
Similarly since $\Os(p)=\Os(p')=0$, 
$$0 = \Os\left(\sum_{i=1}^n (a_i \mathbf{R}_i + b_i \mathbf{C}_i)\right) = \sum_{i=1}^n (a_i + b_i).$$
Using these three equalities, we have that $Q(p)=Q(p')$.
\end{proof}

\begin{lemma}\label{lem:connects} Let $\x,\y \in \S(\bar{g})$, then there is a domain $p\in \pi(\x,\y)$ with $O_n(p)=\Os(p)=0$.\end{lemma}
	\begin{proof} Let $\x,\y \in \S(\bar{g})$.  First we note that there is domain connecting $\x$ to $\y$.  Since $S_n$ is generated by transpositions, there is a sequence of rectangles connecting $\x$ to $\y$ (not necessarily empty).  The sum of these rectangles
is a domain $p_0$ connecting $\x$ to $\y$.	

Let $m_0= O_n(p)$, which is not necessarily zero.  We replace each rectangle containing $O_n$ with the other rectangle connecting its corners as follows.  Let 
$$p_1 = p_0 - m_0 \left(\mathbf{R}_n + \sum_{i=1}^{n-1} C_i\right)$$
then $p_1$ connects $\x$ to $\y$ since $\mathbf{R}_n + \sum_{i=1}^{n-1} \mathbf{C}_i$ is periodic.  Moreover, $O_n(\mathbf{R}_n + \sum_{i=1}^{n-1} \mathbf{C}_i) =1$ so that $O_n(p_1)=0$. (One could also let $p_1 = p_0 - m_0 \mathbf{R}_n$ or $p_1 = p_0 - m_0 \mathbf{C}_n$.)
Now let $m_1= \Os(p_1)$ and define $p_2 = p_1 - m_1 \mathbf{R}_1$.  Then $p_2$ connects $\x$ to $\y$ and since $O_n(\mathbf{R}_1)=0$ and $\Os(\mathbf{R}_1)=1$, we have $O_n(p_2)=\Os(p_2)=0$.
\end{proof}
	
We use this to define a function $\mathbf{F}^Q: \S(\bar{g}) \rightarrow \Z$ by first defining $\mathbf{F}^Q(\x_0)=0$ where $\x_0$ is the lower left corner of the $O$'s.  (It doesn't really matter what value we choose.)  Then for $\x\in \S(\bar{g})$, use Lemma~\ref{lem:connects} to pick a domain $p_\x$ connecting $\x$ to $\x_0$ with $O_n(p_\x)= \Os(p_\x)=0.$  Define $$\mathbf{F}^Q(\x)= \mathbf{F}^Q(x_0) + Q(p_\x) = Q(p_\x).$$  This is well-defined by Lemma~\ref{lem:wd}.

\begin{lemma}\label{lem:Qdom}
	Suppose that $\x,\y\in \S(\bar{g})$ and $p$ is any domain connecting $\x$ to $\y$ with $O_n(p)= \Os(p)=0$.  Then $$\mathbf{F}^Q(\x) - \mathbf{F}^Q(\y) = Q(p).$$
\end{lemma}
\begin{proof} Let $p_\x$ be a domain connecting $\x$ to $\x_0$ with $O_n(p_\x)= \Os(p_\x)=0$.  Define $p_\y$ similarly.  Then $p_\x - p_\y$ is a domain connecting $\x$ to $\y$ with $O_n(p_\x - p_\y)= \Os(p_\x - p_\y)=0$.  So by
	Lemma~\ref{lem:wd}, $\mathbf{F}^Q(\x) - \mathbf{F}^Q(\y) = Q(p_\x) - Q(p_\y) = Q(p_\x - p_\y) = Q(p).$ 
\end{proof}

We now use $\mathbf{F}^Q$ to define the $Q$-filtration on $\widetilde{C}$ by 
$$\mathcal{F}_p^Q(\widetilde{C}) = \left\{ \sum b(\x) \x  \in \widetilde{C} \left| \right. \mathbf{F}^Q(\x) \leq p \text{ whenever } b(\x) \neq 0 \right\}.$$
By Lemma~\ref{lem:Qdom}, $\widetilde{\partial}_C: \mathcal{F}_p^Q(\widetilde{C}) \rightarrow \mathcal{F}_p^Q(\widetilde{C})$ so that $(\widetilde{C},\widetilde{\partial}_C)$ becomes a $\Z$ filtered chain complex.  

Note that we have already shown that $F$ is a $\HZ$ bigraded $R_n$-module chain map. 
Thus, the following proposition will complete the proof that two grids that differ by a stabilization$'$ move are quasi-isomorphic as $\HZ$ bigraded $R_n$-module chain complexes. 

\begin{proposition} $F: C \rightarrow C'$ is a quasi-isomorphism.  
\end{proposition}

\begin{proof} Since $C'$ is an $R_n$-module chain complex, we can consider the chain complex $(\widetilde{C}',\widetilde{\partial}')$ obtained by setting all the $U_i$ equal to zero as explained above, $\widetilde{C}' = C'/\mathcal{U}_n C'$.   Since $F$ is an $R_n$-module homomorphism  and a chain map, it descends to a well-defined chain map $\widetilde{F}:\widetilde{C} \rightarrow \widetilde{C}'$.

	Now we can define a $Q$-filtration on $\widetilde{C}'$ using the $Q$-filtation on $\widetilde{C}$.  
	Specifically, we define $\mathbf{F}^Q : \widetilde{B} \rightarrow \Z$ by $\mathbf{F}^Q(\x) = \mathbf{F}^Q(\psi(\x))$ for all $\x \in \S(g)$.  
	From this we have $$\mathcal{F}_p^Q(\widetilde{B}) = \left\{ \sum b(\x) \x  \in \widetilde{B} \left| \right. \mathbf{F}^Q(\x) \leq p \text{ whenever } b(\x) \neq 0 \right\}.$$ One can identify $(\widetilde{C}',\widetilde{\partial}')$ with the mapping cone of $\tilde{\zeta}: \widetilde{B} \rightarrow \widetilde{B}$ (this is the direct sum since $\widetilde{\zeta}$ is the zero map).
	Then we can define the $Q$-filtration on $C'$ by $\mathcal{F}_p^Q(\widetilde{C}') = \mathcal{F}_p^Q(\widetilde{B}) \oplus \mathcal{F}_p^Q(\widetilde{B})$. It is easy to check that $\widetilde{F}$ preserves this filtration.  
	
	Consider the map induced on their associated graded complexes $$\widetilde{F}_Q: \widetilde{C}_Q \rightarrow \widetilde{C}'_Q.$$  
	The first chain complex $\widetilde{C}_Q$ is the $\mathbb{F}$ vector space generated by $\S(\bar{g})$ whose boundary map counts rectangles supported in the column and row through $O_n$ that do not contain $O_n, X_m$, or $X_{j_1}$.  
	Since the column and row through $O_n$ look exactly the same as the column and row through $O_1$ in \cite{MOST} (after renumbering), this chain complex is the same as it is for links. 
	In particular, Lemma~3.7 of \cite{MOST} holds in our case.   
	Similarly, like in \cite{MOST}, $\widetilde{C}'_Q$ is the chain complex whose underlying group is $\widetilde{B} \oplus \widetilde{B}$ and whose boundary maps are trivial. 
	Moreover, the map $\widetilde{F}_Q$ is exactly the same as in \cite{MOST}.  Thus, their proof holds in our case to show that $\widetilde{F}_Q$ is a quasi-isomorphism. 
		Now by Lemma~\ref{mc}, $\widetilde{F}$ is a quasi-isomorphism.  

		We remark that one can define a filtration on $C$ so that $(\widetilde{C},\widetilde{\partial}_C)$ is its associated graded object. 
		Define $$\mathcal{F}_p^U(C) = \left\{ \sum b(\x) U_1^{a_1(\x)} \cdots U_n^{a_n(\x)} \x \in C \left| \right. \sum a_i(\x) \geq p  \text{ whenever } b(\x) \neq 0\right\}.$$  The boundary preserves the filtration making $(C,\partial_C)$ into a filtered chain complex.  We can do the same with $B$ making $(B,\partial_B)$ into a filtered chain complex.  Since $\zeta$ is a filtered map, we can define a filtration on the mapping cone as before $\mathcal{F}_p^U(C') = \mathcal{F}_p^U(B) \oplus \mathcal{F}_p^U(B)$.  It is easy to see that $F$ is a filtered map.  Moreover, the map on the associated graded objects of $C$ and $C'$ induced by $F$, $F_U: C_U \rightarrow C'_U$, can be identified with $\widetilde{F}: \widetilde{C} \rightarrow \widetilde{C}'$.  Since $\widetilde{F}$ is a quasi-isomorphism, so is $F_U$.  Therefore, $F$ is a quasi-isomorphism by Lemma~\ref{mc}.
\end{proof}

\section{The Alexander Polynomial and Sutured Floer Homology}\label{sec:Alex}


In this section, we will define the Alexander polynomial of a transverse spatial graph $f: G \rightarrow S^3$ as a torsion invariant of a balanced sutured manifold associated to $f$  and show that it agrees with the graded Euler characteristic of $\widehat{HFG}(f)$ (when $f$ is sourceless and sinkless).    In addition, we will relate the sutured Floer homology of this balanced sutured manifold to $\widehat{HFG}(f)$ (when $f$ is sourceless and sinkless).  

\subsection{The Alexander Polynomial of a Spatial Graph}
Let $f: G \rightarrow S^3$ be a transverse spatial graph, and $E(f) = S^3 \smallsetminus N(f(G))$ where $N(f(G))$ is a regular neighborhood of $f(G)$ in $S^3$.   Then $E(f)$ has the structure of a (strongly) balanced sutured manifold $(E(f),\gamma(f))$ which is defined as follows.  There will be one suture per edge and one suture per vertex.  The suture associated to a vertex is the boundary of the transverse disk at that vertex and the suture associated to an edge is the boundary of a disk transverse to that edge of $f$.  The sutures, denoted by $s(\gamma(f))$, are oriented as shown in Figure~\ref{sut_graph}.  $\gamma(f)$ is a collection of annuli that are small neighborhoods of the sutures in $\partial E(f)$.  Recall that $R(\gamma) = \partial E(f) \smallsetminus int(\gamma)$ is the oriented surface where the orientation of $R(\gamma)$ is such that the induced orientation on each component of $\partial R(\gamma)$ agrees with the orientation of the corresponding suture.  Then $R_+(\gamma)$ (respectively $R_-(\gamma)$) is the set of components of $R(\gamma)$ whose normal vectors point out of (respectively into) $E(f)$.  Note that $R_+(\gamma)$ is the set of components of $\partial E(f)$ that have the same orientation as $\partial E(f)$.
 It is easy to check that for each component $\Sigma$ of $\partial E(f)$,  $\chi(\Sigma  \cap R_-(\gamma(f)))=\chi(\Sigma \cap R_+(\gamma(f)))$ so that $(E(f),\gamma(f))$ is strongly balanced (in particular, it is balanced).  In addition, we note that $\gamma(f)$ contains no toroidal components.  Note that we do not need $f$ is be sourceless and sinkless to define this. 
See Sections 2 and 3 of \cite{JuPoly} for the definition of a balanced and strongly balanced sutured manifold.  

\begin{figure}[h]
\begin{center}
\begin{picture}(290,145)
\put(0,0){\includegraphics{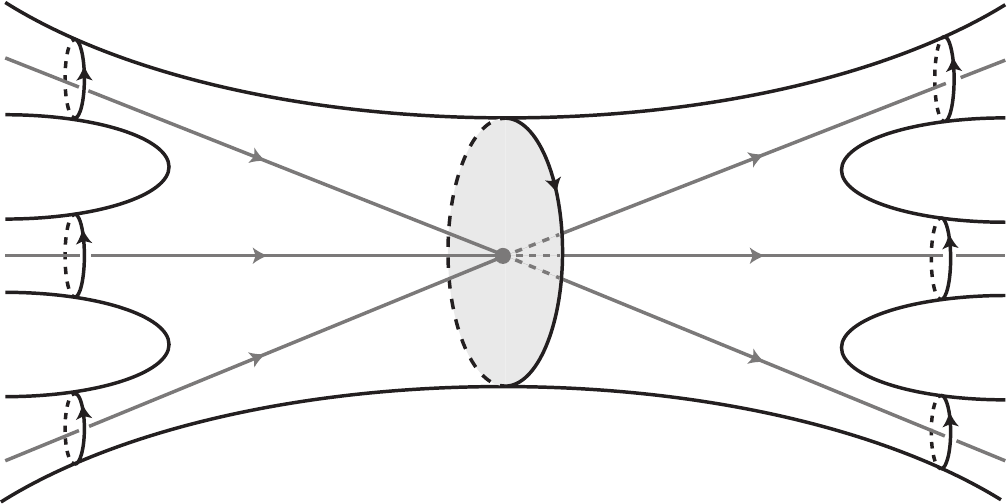}}
\put(135,10){$s(\gamma(f))$}
\put(145,20){\vector(0,1){10}}
\put(130,11){\vector(-1,0){100}}
\put(170,11){\vector(1,0){92}}
\put(170,102){$R_-(\gamma(f))$}
\put(80,102){$R_+(\gamma(f))$}
\end{picture}
\caption{The sutured $(E(f),\gamma(f))$}\label{sut_graph}
\end{center}
\end{figure}
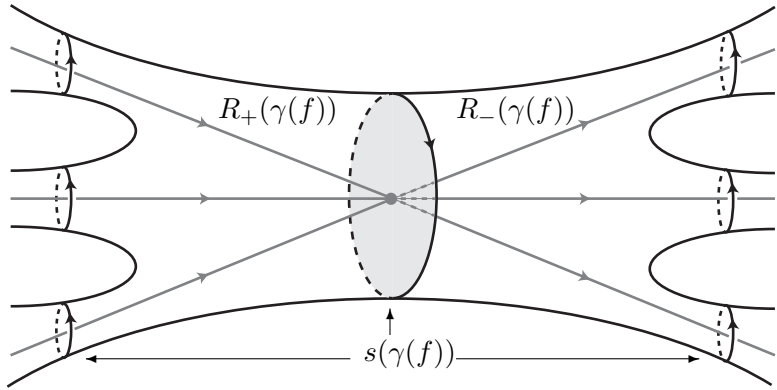

In \cite{FJR11}, S. Friedl, A. Juh{\'a}sz, and J. Rasmussen assign to each balanced sutured manifold $(M, \gamma)$, a torsion invariant $T(M,\gamma) \in \mathbb{Z}[H_1(M)]$ that is well defined up to $\pm h$ for $h \in H_1(M)$.   This invariant is essentially the maximal abelian torsion for the pair $(M, R_-(\gamma))$.  See Sections 3 and 4 of \cite{FJR11} for details. 
Using their torsion invariant and the sutured manifold associated to a transverse spatial graph defined above, we can define our Alexander polynomial. 

\begin{definition} Let $f: G \rightarrow S^3$ be a transverse spatial graph.  The \textbf{(refined) Alexander polynomial of $f$}, denoted by $\Delta_f$, is defined to be $T(E(f),\gamma(f))$ considered as an element of $\mathbb{Z}[H_1(E(f))]$ modulo units.
\end{definition}

\begin{rem} Others have considered Alexander polynomials of spatial graphs in the past. 
	
\noindent \textbf{(1)} The first place this seems to appear is in a paper by Kinoshita in 1958 \cite{Ki58}.  In this paper, Kinoshita defines the Alexander polynomial of a spatial graph as the Alexander polynomial of its exterior.  However, this can be computed using $\pi_1(S^3 \smallsetminus f(G))$ and cannot differentiate between graphs with the same exterior.   Our definition depends on more than just the exterior so it gives more information about the graph than the polynomial defined by Kinoshita. 
\vspace{5pt}

\noindent \textbf{(2)} In 1989, Litherland defined an Alexander polynomial for an embedding of a generalized theta graph that is not determined by its exterior \cite{Li89} (see also \cite{McSW11}).  He considers the Alexander polynomial associated to the torsion free abelian cover of the pair $(S^3 \smallsetminus f(G), R_-)$ where $R_-$ is half of the boundary obtained by cutting open $\partial(S^3 \smallsetminus f(G))$ along the meridians of the edges and throwing away one of the components.  We note that, for theta graphs, Litherland's definition and ours are very similar, the main difference is how we decompose $\partial(S^3 \smallsetminus f(G))$.   In our case the sutures depend on the orientation of the edges.  We note that if all the edges are oriented in the same direction, $\Delta_f$ will be zero since $R_-(\gamma)$ will contain a disjoint disk.  
\vspace{5pt}

\noindent \textbf{(3)}  If $G$ contains a vertex all of whose edges are incoming or outgoing, then for any transverse spatial graph $f:G \rightarrow S^3$, $R_-(\gamma)$ will contain a disjoint disk and hence $\Delta_f=0$.  
\vspace{5pt}

\noindent \textbf{(4)} Like in Litherland's paper, instead of just studying the Alexander polynomial of a transverse spatial graph, one could study the entire Alexander module $\mathbb{Z}[H_1(E(f))]$-module $H_1(E(f), R_-(\gamma(f)))$.  
\end{rem}

\subsection{The sutured Floer homology of a spatial graph}

In the preceding subsection, we defined a balanced sutured manifold $(E(f),\gamma(f))$, associated to a transverse spatial graph $f$.  Instead of just considering the torsion of this sutured manifold, we can consider the sutured Floer homology of it.  We will describe this in more detail in this subsection.   We will also show that this homology theory coincides with our hat theory.  We begin with more definitions and background.  

Associated to each $n\times n$ graph grid diagram $g$ representing $f$, there is another sutured manifold $(E(f), \gamma(g))$ which is defined as follows.  For each $X$ and $O$ on the torus, remove an (open) disk.  Then one obtains an oriented torus with $n+r$ disks removed where $n$ is the size of the grid and $r$ is the number of $X$'s in the grid;  call this surface $\Sigma(g)$.  Note, the orientation of the torus comes from the standard counterclockwise orientation of the plane.  Recall that the horizontal circles of $g$ are called $\alpha_i$, the vertical circles are called $\beta_j$, and $\alphas=\{\alpha_1, \dots, \alpha_n\}$ and $\betas=\{\beta_1, \dots, \beta_n\}$. Thus $(\Sigma(g), \boldsymbol{\alpha}, \boldsymbol{\beta})$ gives a sutured Heegaard decomposition associated to $g$ whose underlying manifold is $E(f)$.   Let $(E(f), \gamma(g))$ be the sutured manifold associated to $(\Sigma(g), \boldsymbol{\alpha}, \boldsymbol{\beta})$.  Recall that $E(f)$ is obtained by from $(\Sigma(g), \boldsymbol{\alpha}, \boldsymbol{\beta})$ by attaching $3$-dimensional $2$-handles to $\Sigma(g) \times I$ along the curves $\alpha_i \times \{0\}$ and $\beta_j \times \{1\}$ for $i=1, \dots, n$ and $j=1,\dots, n$.  The sutures are defined by taking $s(\gamma(g)) = \partial \Sigma(g) \times \{1/2\}$ and $\gamma(g) = \partial \Sigma(g) \times I$.  
Here, we are using the outward normal first convention for the induced orientation on the boundary and we are viewing $I$ with the usual orientation (oriented from $0$ to $1$).  Thus, the induced orientation on the boundary would give $\Sigma(g) \times \{1\}$ the same orientation as $\Sigma(g)$ and $\Sigma(g) \times \{0\}$ the opposite.   Note that $(E(f), \gamma(g))$ is a strongly balanced sutured manifold with one suture for each $X$ and $O$.  An example for the trivial knot is shown in Figures \ref{sut_2x2} and \ref{sut_grid_ex}.  Here, we are viewing one of the $O$'s as being associated with a vertex.  In Figure~\ref{sut_grid_ex}, one needs to attach $2$-handles to $\alpha_i \times \{0\}$ and $\beta_i\times \{1\}$ to obtain $E(f)$. 

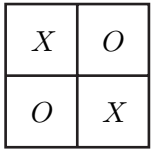
\begin{figure}[h]
\begin{center}
	\begin{picture}(55,55)
\put(0,0){\includegraphics{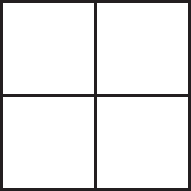}}
\put(10,10){$O$}
\put(37,37){$O$}
\put(10,37){$X$}
\put(37,10){$X$}
\end{picture}
\caption{A graph grid diagram $g$ for the unknot}\label{sut_2x2}
\end{center}
\end{figure}

\begin{figure}[h]
\begin{center}
\begin{picture}(253,185)
\put(0,0){\includegraphics{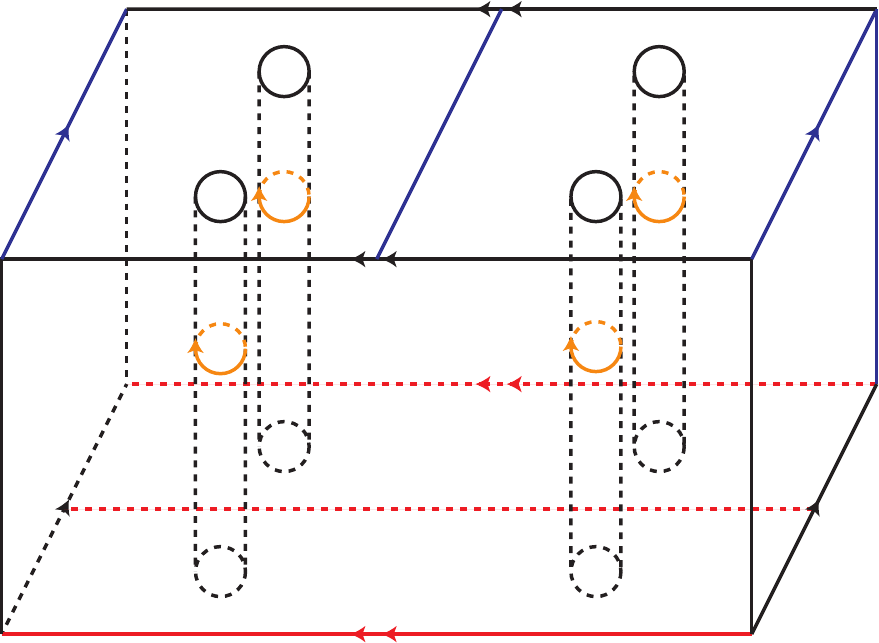}}
\put(40,150){$R_+(\gamma)$}
\end{picture}
\caption{The sutured manifold $(E(f),\gamma(g))$ corresponding to the graph grid diagram $g$ in Figure~\ref{sut_2x2}}\label{sut_grid_ex}
\end{center}
\end{figure}

The sutured manifold $(E(f), \gamma(g))$ is similar to the sutured manifold $(E(f),\gamma(f))$ except that there are an extra $2n_e$ sutures per edge (all parallel and alternating in orientation) where $n_e$ is the number of $O$'s associated to the edge $e$ (this does not count the $O$'s associated to the vertices at the boundary of an edge).   
Since $(E(f), \gamma(f))$ and $(E(f), \gamma(g))$ are balanced sutured manifolds, we can consider their sutured Floer homologies $$SFH(E(f),\gamma(f)) \text{ and } SFH(E(f),\gamma(g))$$ respectively.  See \cite{Ju06} for the definition of sutured Floer homology.  We note that the former group is an invariant of the spatial graph while the latter group $SFH(E(f),\gamma(g))$ depends on the grid $g$.  Each of these groups has two (relative) gradings, an $H_1(M)$ (or equivalently $\text{Spin}^c$) and a homological grading, which we discuss below.  

Let $(M,\gamma)$ be a balanced sutured manifold.   We first discuss the homological grading on $SFH(M,\gamma)$.  Let $(\Sigma, \alphas, \betas)$ be a balanced and admissible diagram for $(M,\gamma)$ where $\alphas$ and $\betas$ each contain $d$ disjointly embedded curves.  Admissible means that every non-trivial periodic domain has both positive and negative coefficients.  Recall $SFH(M,\gamma)$ is defined as the homology of a chain complex $\text{CFH}(\Sigma, \alphas, \betas)$ which is an $\mathbb{F}$-vector space and is roughly defined as follows.   Consider $\mathbb{T}_{\alphas} = \alpha_1 \times \cdots \times \alpha_d$ and $\mathbb{T}_{\betas} = \beta_1 \times \cdots \times \beta_d$, 
the $d$-dimensional tori in $Sym^d(\Sigma)$.  The set of generators of $\text{CFH}(\Sigma, \alphas, \betas)$ is $\mathbf{x} \in \mathbb{T}_{\alphas} \cap \mathbb{T}_{\betas}$ and the differential is defined by counting rigid holomorphic disks in $\Sym^d(\Sigma)$ connecting two points in $\mathbb{T}_{\alphas} \cap \mathbb{T}_{\betas}$.  Choose orientations on $\mathbb{T}_{\betas}$, $\mathbb{T}_{\betas}$, and $\Sym^d(\Sigma)$ and define $\mathbf{m}(\mathbf{x})$ to be the intersection sign of $\mathbb{T}_{\alphas}$ and $\mathbb{T}_{\betas}$ in $\Sym^d(\Sigma)$.  This depends on the choice of orientations but the difference $\mathbf{m}(\mathbf{x})\mathbf{m}(\mathbf{y})^{-1}$ between $\mathbf{m}(\mathbf{x})$ and $\mathbf{m}(\mathbf{y})$ is independent of the choice of orientations.   So $\mathbf{m}$ gives a well-defined relative $\{\pm 1\}$ grading on $\text{CFH}(\Sigma, \alphas, \betas)$.  Let $\mathbb{Z}_2=\{0,1\}$ be the group with 2 elements and let $\text{exp} : \Z_2 \rightarrow \{\pm 1\}$ be the isomorphism sending $l$ to $(-1)^l$.  Using $e\mathbf{m} := \text{exp}^{-1}\circ \mathbf{m}$ (instead of $\mathbf{m}$), we get a relative $\Z_2$ grading on $\text{CFH}(\Sigma, \alphas, \betas)$.  
Since the parity of the Maslov index of a holomorphic disk connecting $\mathbf{x}$ to $\mathbf{y}$ is equal to $\text{exp}^{-1}(\mathbf{m}(\mathbf{x})\mathbf{m}(\mathbf{y})^{-1})$, the differential reduces the homological grading by 1 (mod 2). 

We now briefly review the $\text{Spin}^c$ grading. See Section 3 of \cite{JuPoly} for details. 
Recall that for any balanced sutured manifold $(M,\gamma)$, one can define $\text{Spin}^c(M,\gamma)$, the set of $\text{Spin}^c$ structures of $(M,\gamma)$, as the set of homology classes of nowhere zero vector fields on $M$ that restrict to a special vector field $\nu_0$ on $\partial M$.  The vector field $\nu_0$ depends on the sutures.  One can use obstruction theory to see that $\text{Spin}^c(M,\gamma)$ forms an affine space over $H^2(M,\partial M)$ which we will identify with $H_1(M)$ via the Poincare Duality isomorphism ($\text{PD}: H^2(M,\partial M) \rightarrow H_1(M)$). 
Thus, once we pick an fixed $\mathfrak{s}_0 \in \text{Spin}^c(M,\gamma)$, there is a unique bijective correspondence $\xi_{\mathfrak{s}_0}: H_1(M) \rightarrow \text{Spin}^c(M,\gamma)$, $v \mapsto v + \mathfrak{s}_0$, making $\text{Spin}^c(M,\gamma)$ into an abelian group.  Moreover, for any $\mathfrak{s}, \mathfrak{t} \in \text{Spin}^c(M,\gamma)$, there is a well-defined difference, denoted $\mathfrak{s} - \mathfrak{t} \in H_1(M)$, that is defined as the unique element in $H_1(M)$ such that $(\mathfrak{s} - \mathfrak{t}) + \mathfrak{t} = \mathfrak{s}$.  Note that the difference doesn't depend on any choices. Indeed, it is easy to see that for any $\mathfrak{s}_0 \in \text{Spin}^c(M,\gamma)$ and $v,w \in H_1(M)$,  $(v + \mathfrak{s}_0) - (w + \mathfrak{s}_0) = v-w.$  

To each $\mathbf{x} \in \mathbb{T}_{\alphas} \cap \mathbb{T}_{\betas}$, one can assign an element of $\text{Spin}^c(M,\gamma)$, denoted $\mathfrak{s}(\mathbf{x})$ making $\text{CFH}(\Sigma, \alphas, \betas)$ into a graded vector space over $\text{Spin}^c(M,\gamma)$.  The differential on $\text{CFH}(\Sigma, \alphas, \betas)$ preserves the $\text{Spin}^c(M,\gamma)$ grading, giving a $\text{Spin}^c(M,\gamma)$ grading to $SFH(M,\gamma)$.   
Using the bijection $\xi_{\mathfrak{s}_0}: H_1(M) \rightarrow \text{Spin}^c(M,\gamma)$ as described above, $SFH(M,\gamma)$ becomes a relatively graded vector space over $H_1(M)$.  We can easily compute the relative grading as follows.  
Let $\mathbf{x},\mathbf{y} \in \mathbb{T}_{\alphas} \cap \mathbb{T}_{\betas}$ and choose paths $a: I \rightarrow \mathbb{T}_{\alphas}$ and $b: I \rightarrow \mathbb{T}_{\betas}$ with $\partial a = \partial b = \mathbf{y} - \mathbf{x}$.  Then $a - b$ can be viewed as a 1-cycle in $\Sigma$.  Using the inclusion map of $\Sigma$ into $M$, we can view this as a cycle in $M$.  Let $\varepsilon(\mathbf{y},\mathbf{x}) \in H_1(M)$ be the homology class of $a-b$.  By Lemma 4.7 of \cite{Ju06}, $\varepsilon(\mathbf{y},\mathbf{x})$ is the relative grading in $H_1(M)$ associated to the $\text{Spin}^c(M)$ grading on $(M,\gamma)$. That is, 
\begin{equation}\mathfrak{s}(\mathbf{y}) - \mathfrak{s}(\mathbf{x}) = \varepsilon(\mathbf{y},\mathbf{x}).\label{spin_eq}\end{equation}
Note that \cite{Ju06} actually says that $\text{PD}(\mathfrak{s}(\mathbf{y}) - \mathfrak{s}(\mathbf{x})) = \varepsilon(\mathbf{y},\mathbf{x})$ but we have already identified $\mathfrak{s}(\mathbf{y}) - \mathfrak{s}(\mathbf{x})$ with an element of $H_1(M)$ using Poincare duality.  
Using $(\xi^{-1}_{\mathfrak{s}_0} \circ \mathfrak{s}) \oplus m : \mathbb{T}_{\alphas} \cap \mathbb{T}_{\betas} 
\rightarrow H_1(M) \oplus \mathbb{Z}_2$, 
we see that $\text{CFH}(\Sigma, \alphas, \betas)$ is a well-defined relatively $(H_1(M),\mathbb{Z}_2)$ bigraded chain complex and $SFH(M,\gamma)$ is a well-defined relatively $(H_1(M),\mathbb{Z}_2)$ bigraded $\mathbb{F}$-vector space.  Even though  $\text{CFH}(\Sigma, \alphas, \betas)$ depends on the choice of Heegaard diagram for $(M, \gamma)$,  by \cite{Ju06}, the isomorphism class of  $SFH(M,\gamma)$ as a relatively $(H_1(M),\mathbb{Z}_2)$ bigraded  vector space only depends on the sutured manifold.

\begin{definition} For a transverse spatial graph $f: G \rightarrow S^3$, the \textbf{sutured graph Floer homology} is $SFH(E(f),\gamma(f))$ considered as a relatively $(H_1(E(f)),\mathbb{Z}_2)$ bigraded $\mathbb{F}$-vector space.
\end{definition}

\noindent Unless otherwise stated, when we refer to $SFH(E(f),\gamma(f))$ and $SFH(E(f),\gamma(g))$, we will be considering them as relatively bigraded vector spaces.  


\subsection{Relating $\widetilde{HFG}(f)$ and $SFH(E(f),\gamma(f))$}

The main result of this section is that the graph Floer homology of a sinkless and sourceless transverse spatial graph is isomorphic to its sutured Floer homology (after a slight change of the Alexander grading). To prove this, we first notice that the sutured Floer homology of a grid is the same as $\widetilde{HFG}(g)$.  First we need to discuss how the grading is changed.  

Let $\G$ be an abelian group and $C$ be a $(\G,\mathbb{Z}_2)$ bigraded chain complex or $\mathbb{F}$-vector space with bigrading $C = \oplus_{(g,m) \in \G \oplus \mathbb{Z}_2} C_{(g,m)}$.  Let $(rC)_{(g,m)} = C_{(-g,m)}$ for each $g \in \G$ and $m \in \mathbb{Z}_2$.  Then $C$ has a $(\G,\mathbb{Z}_2)$ bigrading given by $C = \oplus_{(g,m) \in \G \oplus \mathbb{Z}_2} (rC)_{(g,m)}$, which we call the \textbf{reverse bigrading on $C$}.  We denote by $rC$, the underlying vector space $C$ with its reverse bigrading.  
If $C$ is a chain complex then, $\partial: (rC)_{(g,m)} \rightarrow (rC)_{(g,m-1)}.$  
So $\partial$ is a degree $(0,-1)$ on $rC$ and $rC$ is a $(\G,\mathbb{Z}_2)$ bigraded chain complex. If $C$ is has a relative $(\G,\mathbb{Z}_2)$ bigrading then $rC$ has a well-defined relative $(\G,\mathbb{Z}_2)$ bigrading. Note that we are only ``reversing'' the $\G$ grading.

We note that using the natural projection of $\mathbb{Z}$ to $\mathbb{Z}_2$, $\widehat{HFG}(f)$, $\widetilde{HFG}(f)$, $\widehat{HFG}(g)$, and $\widehat{HFG}(g)$ become relatively $(H_1(E(f)),\mathbb{Z}_2)$ bigraded $\mathbb{F}$-vector spaces (similarly for the chain complexes defining them).

\begin{lemma}\label{lem:hfg_sfh}Let $f: G \rightarrow S^3$ be a sinkless and sourceless transverse spatial graph and $g$ be a be a graph grid diagram representing $f$ then 
	$$\widetilde{HFG}(g) \cong rSFH(E(f),\gamma(g))$$
as relatively $(H_1(E(f)),\mathbb{Z}_2)$ bigraded $\mathbb{F}$-vector spaces.
\end{lemma}

\begin{proof} Let $(\Sigma(g), \alphas, \betas)$ be the specific Heegaard decomposition for $E(f)$ associated to the graph grid $g$ as defined beforehand.  Both $CFH(\Sigma(g), \alphas, \betas)$ and $\widetilde{C}(g)$ are $\mathbb{F}$ vector spaces with the same generating set.  One can check that the boundary maps are the same giving an identification of the two chain complexes.  Thus, we just need to compare their (relative) gradings.  
	
	Let $\x$ and $\y$ be generators (in either chain complex).  If there is a rectangle connecting $\textbf{x}$ and $\textbf{y}$, then by Equation~(\ref{maslov_eq}), $M(\textbf{x}) - M(\textbf{y})= 1 \mod 2$.  Moreover, using the definition of $\mathbf{m}$ as described in the previous subsection, it is straightforward to check that $\mathbf{m}(\textbf{x})\mathbf{m}(\textbf{y})^{-1} = -1 \in \{\pm 1\}$.  Now, for any $\textbf{x}$ and $\textbf{y}$, there is a sequence of rectangles $r_1, \dots, r_k$ connecting $\textbf{x}$ and $\textbf{y}$.  Thus, $M(\textbf{x}) - M(\textbf{y})= k \mod 2$ and $\mathbf{m}(\textbf{x})\mathbf{m}(\textbf{y})^{-1} = (-1)^k \in \{\pm 1\}$.  Hence $$(-1)^{M(\textbf{x}) - M(\textbf{y})} = \mathbf{m}(\textbf{x})\mathbf{m}(\textbf{y})^{-1}.$$  By Lemma~\ref{lem:alex_count} and Equation~(\ref{spin_eq}), we see that $$A(\x) - A(\y) = \varepsilon(\y,\x) = - (\mathfrak{s}(\mathbf{x}) - \mathfrak{s}(\mathbf{y})).$$
	
	\end{proof}

Using Proposition~5.4 in \cite{JuPoly}, we can relate $SFH(E(f),\gamma(f))$ and $SFH(E(f),\gamma(g))$.  For $g\in \G$ and $m \in \mathbb{Z}_2$, let $W{(-g,-1)}$ ($=W{(-g,1)}$) be the two dimensional $(\G,\Z_2)$ bigraded vector space over $\mathbb{F}$ spanned by one generator in degree $(0,0)$ and the other in degree $(-g,-1)$.  (Note that we are slightly abusing notation since, in Section~\ref{sec:def}, we defined $W{(-g,-1)}$ to be the $(\G,\Z)$ bigraded vector space over $\mathbb{F}$ spanned by one generator in degree $(0,0)$ and the other in degree $(-g,-1)$.) If $(C,\partial)$ is a relatively bigraded $(\G, \Z_2)$ chain complex over $\mathbb{F}$, then $C \otimes W{(-g,-1)}$ becomes a relatively bigraded $(\G, \Z_2)$ chain complex with boundary $\partial \otimes id$ in the usual way.   

\begin{proposition} \label{prop:sfh_tensor}Let $f: G \rightarrow S^3$ be a transverse spatial graph and let $g$ be a graph grid diagram representing $f$.  Then 
	$$SFH(E(f),\gamma(g)) \cong SFH(E(f),\gamma(f)) \otimes \bigotimes_{e \in E(G)} {W{(w(e),1)}}^{\otimes n_e} $$
as relatively $(H_1(E(f)),\Z_2)$ bigraded $\mathbb{F}$-vector spaces, where $n_e$ is the number of $O$'s in $g$ associated to the interior of $e$ (not including the vertices).
\end{proposition}

\begin{proof} Recall that $(E(f),\gamma(g))$ is a sutured manifold with $2n_e+1$ sutures associated to each edge $e$ ($n_e$ sutures are associated to $O$'s on the interior of $e$ and the other $n_e+1$ sutures are associated to $X$'s on the interior of $e$).  Pick an edge $e$.  If $n_e=0$, leave the sutures on that edge alone.  If $n_e\geq 1$, then there are at least three sutures associated to the edge that are parallel and have alternating orientations.  Let $S$ be the properly embedded surface in $E(f)$ as pictured in Figure~\ref{fig:surface_decomp} with either orientation.  In this figure, the inner annulus is part of the boundary of the neighborhood of the edge of the graph.  It contains the three parallel sutures with alternating orientations.   Since $S$ is a decomposing surface, it defines a sutured manifold decomposition $$(E(f),\gamma(g)) \rightsquigarrow^{S} (E(f)',\gamma(g)').$$  See Definition~2.5 in \cite{JuPoly} for the definition of a decomposing surface and a sutured manifold decomposition.  The resulting manifold $E(f)'$ is defined as $E(f) \smallsetminus int(N(S))$ where $N(S)$ is a neighborhood of $S$ in $E(f)$ and hence is homeomorphic to the disjoint union of $E(f)$ and $S^1 \times D^2$.  To get the sutures on $E(f) \subset E(f)'$, we remove two of the three aforementioned sutures associated to $e$ with opposite orientations.  There are four sutures on $S^1 \times D^2 \subset E(f)'$; they are all parallel to $S^1 \times \{p\}$ for $p\in \partial D^2$ and have alternating orientations.   $(M_1,\gamma_1)$ be the component of $(E(f)',\gamma(g)')$ with $M_1$  homeomorphic to $E(f)$ and let $(M_2,\gamma_2)$ be the component of $(E(f)',\gamma(g)')$ with $M_2$ homeomorphic to $S^1 \times D^2$.

\begin{figure}[h]
\begin{center}
\begin{picture}(288,119)
\put(0,0){\includegraphics{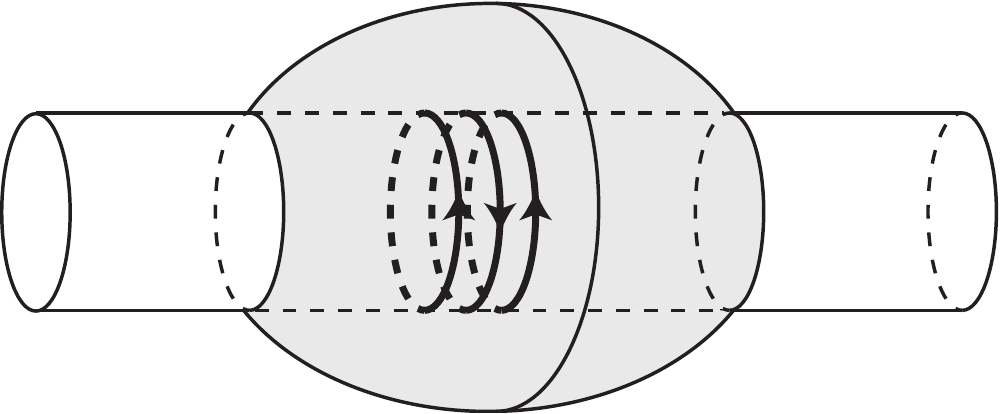}}
\put(105,100){$S$}
\put(25,60){$R_-(\gamma)$}
\put(235,60){$R_+(\gamma)$}
\put(25,100){$E(f)$}
\end{picture}
\caption{The decomposing surface $S$ (shaded) properly embedded in $E(f)$ }\label{fig:surface_decomp}
\end{center}
\end{figure}
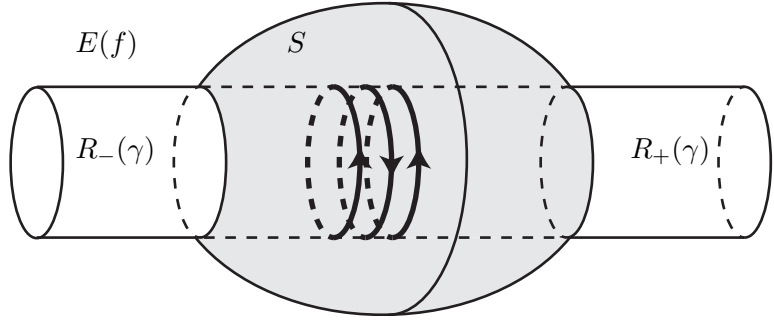

Since $S$ is a nice decomposing surface (see Definition 3.22 in \cite{JuPoly} for the definition of nice), we can use the Proposition 5.4 of \cite{JuPoly} to compute $SFH(E(f)', \gamma(g)')$.  
First we need some notation. Let $i: E(f)' \rightarrow E(f)$ be the inclusion map and $i_\ast:H_1(E(f)') \rightarrow H_1(E(f))$ be the induced map on $H_1(-)$.
For $\mathfrak{s} \in \text{Spin}^c(E(f),\gamma(g))$, $\text{CFH}(\Sigma, \alphas, \betas, \mathfrak{s})$ is defined as the chain complex generated by 
$\x \in \mathbb{T}_{\alphas} \cap \mathbb{T}_{\betas}$ with $\mathfrak{s}(\x)=\mathfrak{s}$ where  $(\Sigma, \alphas, \betas)$ is a balanced admissible diagram for $(E(f),\gamma(g))$. $SFH(E(f),\gamma(g),\mathfrak{s})$ is the homology of $\text{CFH}(\Sigma, \alphas, \betas, \mathfrak{s})$ (similarly for $(E(f)',\gamma(g)')$).  
 By Proposition~5.4 of \cite{JuPoly}, there is an affine map $$f_S: \text{Spin}^c(E(f)',\gamma(g)') \rightarrow \text{Spin}^c(E(f),\gamma(g))$$ satisfying the following two conditions.
 \begin{enumerate}
 \item There is an isomorphism 
$$ \widehat{p}: SFH(E(f)',\gamma(g)') \xrightarrow{\cong} \bigoplus_{ \mathfrak{s} \in  \text{Im}(f_S)} SFH(E(f),\gamma(g),\mathfrak{s}) $$
such that for every $\mathfrak{s}' \in \text{Spin}^c(E(f)',\gamma(g)')$  
we have $$\widehat{p}(SFH(E(f)',\gamma(g)',\mathfrak{s}') ) \subset SFH(E(f),\gamma(g),f_S(\mathfrak{s}')).$$

  \item If $\mathfrak{s}'_1, \mathfrak{s}'_2 \in \text{Spin}^c(E(f)',\gamma(g)')$, then 
 $$i_\ast(\mathfrak{s}'_1 - \mathfrak{s}'_2) = f_S(\mathfrak{s}'_1) - f_S(\mathfrak{s}'_2) \in H_1(E(f)).$$
 \end{enumerate}
 Note that in \cite{Ju06}, Juh{\'a}sz is identifying $\text{Spin}^c(E(f),\gamma(g))$ with $H^2(E(f),\partial E(f))$ so the statement looks slightly different. 
Since $i_{|M_1}: M_1 \rightarrow E(f)$ is a homotopy equivalence, $i_\ast: H_1(E(f)') \rightarrow H_1(E_f)$ is surjective.  
Fix an $\mathfrak{s}_0\in \text{Spin}^c(E(f)',\gamma(g)')$ and let $\mathfrak{s}_0 := f_S(\mathfrak{s}'_0)$. 
Use $\mathfrak{s}_0$ to identify $H_1(E(f))$ and $\text{Spin}^c(E(f),\gamma(g))$.  This identifies $v\in H_1(E(f))$ with $v + \mathfrak{s}_0 \in \text{Spin}^c(E(f),\gamma(g))$.  Let $\mathfrak{s}\in \text{Spin}^c(E(f),\partial E(f))$.  Then $\mathfrak{s}=v + \mathfrak{s}_0$ for some $v\in H_1(E(f))$.  Since $i_\ast$ is surjective, there is a $v'\in H_1(E(f)')$ with $i_\ast(v')=v$.  The second statement above implies that  
$$i_\ast(v') = i_\ast((v' + \mathfrak{s}'_0) - \mathfrak{s}'_0) = f_S(v'+ \mathfrak{s}'_0) - f_S(\mathfrak{s}'_0)=f_S(v'+ \mathfrak{s}'_0) - \mathfrak{s}_0.$$
Therefore, $\mathfrak{s} = v + \mathfrak{s}_0 = i_\ast(v') + \mathfrak{s}_0 = f_S(v' + \mathfrak{s}'_0)$ so $f_S$ is surjective.   

Choose an orientation on $\Sigma$, $\alpha_i$ and $\beta_j$ for all $i,j$.  Then we have an absolute $\mathbb{Z}_2$ grading on $SFH(E(f),\gamma(g))$ given by $e\mathbf{m} = exp^{-1} \circ \mathbf{m}$.  It can be shown that, as a relative grading,  it agrees with $gr$ (mod 2).  
For a definition of $gr$ see Definition 8.1 of \cite{Ju06}.   However, $gr$ is only defined for two generators that have the same $Spin^c$ class.  Thus, we will need to consider the proof Proposition~5.4 of \cite{JuPoly}, in order to show that $e\mathbf{m}$ is preserved under $\widehat{p}$.  In this proof, he considers a balanced diagram $(\Sigma, \alphas,\betas,P)$ adapted to the surface $S$ in $(E(f),\gamma(g))$.  Here  $(\Sigma, \alphas,\betas)$ is a Heegaard diagram for  $(E(f),\gamma(g))$ and $P\subset \Sigma$ is a quasipolygon.  Using $P$, he then constructs a Heegaard diagram $(\Sigma', \alphas',\betas')$ for $(E(f)',\gamma(g)')$ and a map $$p: \Sigma' \rightarrow \Sigma$$ such that $p$ sends $\alpha_i'$ to $\alpha_i$ and $\beta'_j$ to $\beta_j$ and
$$p_{|  \Sigma' \smallsetminus p^{-1}(P)} : \Sigma' \smallsetminus p^{-1}(P) \rightarrow \Sigma \smallsetminus P$$ is a diffeomorphism.  Moreover, since $f_S$ is onto in our case, it follows that all the intersections of $\alpha_i$ and $\beta_j$ lie in $\Sigma \smallsetminus P$ and all the intersections of $\alpha'_i$ and $\beta'_j$ lie in $\Sigma' \smallsetminus p^{-1}(P)$.  $p$ induces a bijection $\widehat{p}: \mathbb{T}_{\alphas'} \cap \mathbb{T}_{\betas'} \rightarrow \mathbb{T}_{\alphas} \cap \mathbb{T}_{\betas}$.  This gives the isomorphism  
$\widehat{p}: SFH(E(f)',\gamma(g)') \rightarrow SFH(E(f),\gamma(g)).$
Choose the orientations of 
$\Sigma'$, $\alpha'_i$ and $\beta'_j$ coming from $\Sigma$, $\alpha_i$ and $\beta_j$ so that $p$ preserves all the orientations.  It then follows that if $\x' \in \mathbb{T}_{\alphas'} \cap \mathbb{T}_{\betas'}$ then 
\begin{equation}\label{eq:meq}
\mathbf{m}(\widehat{p}(\x')) = \mathbf{m}(\x').
\end{equation}

Using $f_S$ and using the identification of $H_1(E(f))$ and $Spin^c(E(f),\gamma(g))$, $SFH(E(f)',\gamma(g)')$ inherits an $H_1(E(f))$ grading.  This makes both $SFH(E(f),\gamma(g))$ into $(H_1(E(f)),\mathbb{Z}_2)$-bigraded $\mathbb{F}$-vector spaces. We can now show that $\widehat{p}: SFH(E(f)',\gamma(g)') \rightarrow SFH(E(f),\gamma(g))$ is an $(H_1(E(f)),\mathbb{Z}_2)$-bigraded map.  Let $x' \in SFH(E(f)',\gamma(g)')_{(v,i)}$ where $(v,i) \in H_1(E(f)) \oplus \mathbb{Z}_2$.  Since $f_S(v' + \mathfrak{s}'_0) = i_\ast(v') + \mathfrak{s}_0$ this means that $x' \in SFH(E(f)',\gamma(g)',v' + \mathfrak{s}'_0)$ for some $v' \in H_1(E(f)')$ with $i_\ast(v')=v$.   So by (1) of Proposition~5.4 in \cite{JuPoly},
 $$\widehat{p}(SFH(E(f)',\gamma(g)',v' + \mathfrak{s}'_0) ) \subset SFH(E(f),\gamma(g),v + \mathfrak{s}_0).$$  Using this and Equation~(\ref{eq:meq}), it follows that $\widehat{p}(x') \in SFH(E(f),\gamma(g))_{(v,i)}$.
 
We show that $$SFH(E(f)',\gamma(g)') \cong SFH(M_1,\gamma_1) \otimes {W{(w(e),1)}}.$$  To see this, note that $(E(f)', \gamma(g'))$ is the disjoint union of $(M_1,\gamma_1)$ and $(M_2,\gamma_2)$.  Moreover, it is easy to see that 
$$SFH(M_2,\gamma_2) \cong W(e(g),1)$$ as a relatively $(H_1(E(f)),\mathbb{Z}_2)$ bigraded $\mathbb{F}$-vector space.
To complete the proof, we continue removing pairs of sutures on each edge until we are left with $(E(f), \gamma(f))$.  
  \end{proof}

We can use this to complete the relationship between $\widehat{HFG}(f)$ and $rSFH(E(f),\gamma(f))$.

\begin{thm}\label{thm:same} Let  $f: G \rightarrow S^3$ be a sinkless and sourceless transverse spatial graph.  Then 
$$\widehat{HFG}(f) \cong rSFH(E(f),\gamma(f))$$
as relatively $(H_1(E(f)),\mathbb{Z}_2)$ bigraded $\mathbb{F}$-vector spaces.
\end{thm}

\begin{proof}  Let $g$ be a graph grid diagram representing $f$.  By Proposition~\ref{prop:hfg_tensor},  Lemma~\ref{lem:hfg_sfh}, and Proposition~\ref{prop:sfh_tensor}, we have that
\begin{align*}\widehat{HFG}(f) \otimes \bigotimes_{e \in E(G)} {W{(-w(e),-1)}}^{\otimes n_e} &\cong \widetilde{HFG}(g)\\
& \cong rSFH(E(f),\gamma(g)) \\
& \cong r\left(SFH(E(f),\gamma(f)) \otimes \bigotimes_{e \in E(G)} {W{(w(e),1)}}^{\otimes n_e} \right) \\
& \cong rSFH(E(f),\gamma(f)) \otimes \bigotimes_{e \in E(G)} {W{(-w(e),-1)}}^{\otimes n_e} 
\end{align*}
as relatively $(H_1(E(f)),\mathbb{Z}_2)$ bigraded $\mathbb{F}$-vector spaces.
The result follows from a slight generalization of Lemma~3.18 in \cite{ET14} (replace $\mathbb{Z}^2$ with $H_1(E(f)) \oplus \mathbb{Z}_2$ in the proof) that $\widehat{HFG}(f) \cong rSFH(E(f),\gamma(f))$ as relatively $(H_1(E(f)),\mathbb{Z}_2)$ bigraded $\mathbb{F}$-vector spaces.
\end{proof}

As a result we see that the decategorification of $\widehat{HFG}(f)$ is essential the torsion invariant $T(E(f),\gamma(f))$ associated to sutured manifold $(E(f),\gamma(f))$.  

\begin{definition}Let $f: G \rightarrow S^3$ be a sinkless and sourceless transverse spatial graph.  Define
$$\chi(\widehat{HFG}(f)) =  \sum_{(h,i) \in (H_1(E(f)),\mathbb{Z})} (-1)^i \text{rank}_\mathbb{F} \widehat{HFG}(f)_{(h,i)} h$$
considered as an element of $\mathbb{Z}[H_1(E(f)]$ modulo positive units (i.e. elements of $H_1(E(f))$). 
\end{definition}

If $r \in \mathbb{Z}[H_1(E(f))]$ then $r= \sum_i a_i h_i$ where $a_i \in \mathbb{Z}$ and $h_i \in H_1(E(f))$.  Define $\overline{r} := \sum_i a_i h^{-1}_i$ where here we are viewing $H_1(E(f))$ as a multiplicative group. 

\begin{corollary}\label{cor:Alex}If $f: G \rightarrow S^3$ is a sinkless and sourceless transverse spatial graph then
	$$ \chi(\widehat{HFG}(f)) \doteq \overline{\Delta}_f. $$ 
	That is, they are the same up to multiplication by units in $\mathbb{Z}[H_1(E(f))]$.
\end{corollary}

\begin{proof}Choose an $\mathfrak{s}_0 \in Spin^c(E(f),\gamma(f))$.  By Theorem~1.1 of \cite{FJR11}, $\chi(SFH(E(f),\gamma(g))) \doteq \Delta_f$ where 
	$$\chi(SFH(E(f),\gamma(g)))= \sum_{\x \in \mathbb{T}_{\alphas} \cap \mathbb{T}_{\betas}} \mathbf{m}(\x) \xi^{-1}_{\mathfrak{s}_0}(\mathfrak{s}(\x))$$ for any admissible balanced Heegaard diagram $(\Sigma,\alphas,\betas)$ for $(E(f),\gamma(f))$. By a standard argument,  $$\sum_{\x \in \mathbb{T}_{\alphas} \cap \mathbb{T}_{\betas}} \mathbf{m}(\x) \xi^{-1}_{\mathfrak{s}_0}(\mathfrak{s}(\x)) = \sum_{(h,i) \in (H_1(E(f)),\mathbb{Z})} (-1)^i \text{rank}_\mathbb{F} \widehat{SFH}(E(f),\gamma(g))_{(h,i)} h.$$  Hence, the result follows from Theorem~\ref{thm:same}.
\end{proof}

\bibliography{HFGs_trans}{}

\begin{thebibliography}{10}

\bibitem{Bao14}
Yuanyuan Bao.
\newblock On the {H}eegaard {F}loer homology for bipartite spatial graphs and
  its properties.
\newblock http://arxiv.org/abs/1401.6608.

\bibitem{Crom}
Peter~R. Cromwell.
\newblock Embedding knots and links in an open book. {I}. {B}asic properties.
\newblock {\em Topology Appl.}, 64(1):37--58, 1995.

\bibitem{Dy}
I.~A. Dynnikov.
\newblock Arc-presentations of links: monotonic simplification.
\newblock {\em Fund. Math.}, 190:29--76, 2006.

\bibitem{FJR11}
Stefan Friedl, Andr{\'a}s Juh{\'a}sz, and Jacob Rasmussen.
\newblock The decategorification of sutured {F}loer homology.
\newblock {\em J. Topol.}, 4(2):431--478, 2011.

\bibitem{Gh}
Paolo Ghiggini.
\newblock Knot {F}loer homology detects genus-one fibred knots.
\newblock {\em Amer. J. Math.}, 130(5):1151--1169, 2008.

\bibitem{Ju06}
Andr{\'a}s Juh{\'a}sz.
\newblock Holomorphic discs and sutured manifolds.
\newblock {\em Algebr. Geom. Topol.}, 6:1429--1457, 2006.

\bibitem{JuPoly}
Andr{\'a}s Juh{\'a}sz.
\newblock The sutured {F}loer homology polytope.
\newblock {\em Geom. Topol.}, 14(3):1303--1354, 2010.

\bibitem{Kauf}
Louis~H. Kauffman.
\newblock Invariants of graphs in three-space.
\newblock {\em Trans. Amer. Math. Soc.}, 311(2):697--710, 1989.

\bibitem{Ki58}
Shin'ichi Kinoshita.
\newblock Alexander polynomials as isotopy invariants. {I}.
\newblock {\em Osaka Math. J.}, 10:263--271, 1958.

\bibitem{Li89}
Rick Litherland.
\newblock The {A}lexander module of a knotted theta-curve.
\newblock {\em Math. Proc. Cambridge Philos. Soc.}, 106(1):95--106, 1989.

\bibitem{MOS}
Ciprian Manolescu, Peter Ozsv{\'a}th, and Sucharit Sarkar.
\newblock A combinatorial description of knot {F}loer homology.
\newblock {\em Ann. of Math. (2)}, 169(2):633--660, 2009.

\bibitem{MOST}
Ciprian Manolescu, Peter Ozsv{\'a}th, Zolt{\'a}n Szab{\'o}, and Dylan Thurston.
\newblock On combinatorial link {F}loer homology.
\newblock {\em Geom. Topol.}, 11:2339--2412, 2007.

\bibitem{McSW11}
Jenelle McAtee, Daniel~S. Silver, and Susan~G. Williams.
\newblock Coloring spatial graphs.
\newblock {\em J. Knot Theory Ramifications}, 10(1):109--120, 2001.

\bibitem{McCleary}
John McCleary.
\newblock {\em User's guide to spectral sequences}, volume~12 of {\em
  Mathematics Lecture Series}.
\newblock Publish or Perish, Inc., Wilmington, DE, 1985.

\bibitem{Ni}
Yi~Ni.
\newblock Link {F}loer homology detects the {T}hurston norm.
\newblock {\em Geom. Topol.}, 13(5):2991--3019, 2009.

\bibitem{OS4}
Peter Ozsv{\'a}th and Zolt{\'a}n Szab{\'o}.
\newblock Holomorphic disks and genus bounds.
\newblock {\em Geom. Topol.}, 8:311--334, 2004.

\bibitem{OS3}
Peter Ozsv{\'a}th and Zolt{\'a}n Szab{\'o}.
\newblock Holomorphic disks and knot invariants.
\newblock {\em Adv. Math.}, 186(1):58--116, 2004.

\bibitem{OS5}
Peter Ozsv{\'a}th and Zolt{\'a}n Szab{\'o}.
\newblock Holomorphic disks, link invariants and the multi-variable {A}lexander
  polynomial.
\newblock {\em Algebr. Geom. Topol.}, 8(2):615--692, 2008.

\bibitem{Ras}
Jacob~Andrew Rasmussen.
\newblock {\em Floer homology and knot complements}.
\newblock ProQuest LLC, Ann Arbor, MI, 2003.
\newblock Thesis (Ph.D.)--Harvard University.

\bibitem{Sa11}
Sucharit Sarkar.
\newblock Grid diagrams and the {O}zsv\'ath-{S}zab\'o tau-invariant.
\newblock {\em Math. Res. Lett.}, 18(6):1239--1257, 2011.

\bibitem{ET14}
Eamonn Tweedy.
\newblock The anti-diagonal filtration: Reduced theory and applications.
\newblock {\em International Mathematics Research Notices}, 2014.

\bibitem{Weibel}
Charles~A. Weibel.
\newblock {\em An introduction to homological algebra}, volume~38 of {\em
  Cambridge Studies in Advanced Mathematics}.
\newblock Cambridge University Press, Cambridge, 1994.

\end{thebibliography}
\bibliographystyle{plain}

\end{document}